\pdfoutput=1
\documentclass[pdflatex]{sn-jnl}
\usepackage{natbib}
\usepackage{amssymb}
\usepackage{amsmath}
\usepackage{amsthm}
\usepackage{enumerate}
\usepackage{mathtools}
\usepackage{bbm}
\usepackage{mathrsfs}
\usepackage{extarrows}
\usepackage{graphicx}
\usepackage[caption=false]{subfig}
\usepackage{tensor}
\usepackage{float}
\usepackage{accents}
\usepackage{url}
\usepackage{color}
\usepackage{anyfontsize}
\usepackage{microtype}
\usepackage{tabularray}

\jyear{2021}

\theoremstyle{thmstyleone}
\newtheorem{theorem}{Theorem}[section]
\newtheorem{proposition}[theorem]{Proposition}
\newtheorem{lemma}[theorem]{Lemma}
\newtheorem{corollary}[theorem]{Corollary}

\theoremstyle{thmstyletwo}

\newtheorem{remark}[theorem]{Remark}

\theoremstyle{thmstylethree}
\newtheorem{definition}[theorem]{Definition}

\theoremstyle{plain}
\newcommand{\thistheoremname}{}
\newtheorem*{genericthm*}{\thistheoremname}
\newenvironment{namedthm*}[1]
  {\renewcommand{\thistheoremname}{#1}
   \begin{genericthm*}}
  {\end{genericthm*}}

\newcommand{\iid}{\textnormal{i.i.d.}}
\newcommand\End{\operatorname{End}}
\newcommand\Img{\operatorname{Im}}
\newcommand\dis{\operatorname{dis}}
\newcommand{\Homl}[1]{\operatorname{\mathnormal{H}_{\mathnormal{#1}}}}

\newcommand{\nnegR}{\mathbb{R}_{\scalebox{0.65}{$\geqslant\hspace{-2pt}0$}}}
\newcommand{\euler}{\mathsf{e}}
\newcommand\dt[1]{\dot{#1}}
\newcommand{\bfone}{\mathbf{1}}
\newcommand{\Corr}{\mathcal{C}}
\newcommand{\bfZ}{\mathbf{Z}}
\newcommand{\bfG}{\mathbf{G}}
\newcommand{\bfH}{\mathbf{H}}
\newcommand{\bfS}{\mathbf{S}}
\newcommand{\bfMo}{\mathbf{M}_{0}}
\newcommand{\bfP}{\mathbf{P}\hspace{-0.3ex}}
\newcommand{\bfgma}{\mathbf{\Gamma}}
\newcommand{\bfDel}{\mathbf{\Delta}}

\DeclarePairedDelimiterX\setc[2]{\{}{\}}{\,#1 \;\delimsize\vert\; #2\,}

\AtBeginDocument{\microtypesetup{protrusion=false,expansion=false}}

\begin{document}

\title[Curse of dimensionality on persistence diagrams]{Curse of dimensionality on persistence diagrams}

\author[2,3]{\fnm{Yasuaki} \sur{Hiraoka}}\email{hiraoka.yasuaki.6z@kyoto-u.ac.jp}
\equalcont{These authors contributed equally to this work.}

\author[2]{\fnm{Yusuke} \sur{Imoto}}\email{imoto.yusuke.4e@kyoto-u.ac.jp}
\equalcont{These authors contributed equally to this work.}

\author[3]{\fnm{Shu} \sur{Kanazawa}}\email{kanazawa.shu.2w@kyoto-u.ac.jp}
\equalcont{These authors contributed equally to this work.}

\author*[1]{\fnm{Enhao} \sur{Liu}}\email{liu.enhao.b93@kyoto-u.jp}

\affil*[1]{\orgdiv{Department of Mathematics}, \orgname{Kyoto University}, \orgaddress{\street{Kitashirakawa Oiwake-cho, Sakyo-ku}, \city{Kyoto}, \postcode{6068502}, \country{Japan}}}

\affil[2]{\orgdiv{Institute for the Advanced Study of Human Biology}, \orgname{Kyoto University}, \orgaddress{\street{Yoshida-Konoe-cho, Sakyo-ku}, \city{Kyoto}, \postcode{6068501}, \country{Japan}}}

\affil[3]{\orgdiv{Kyoto University Institute for Advanced Study}, \orgname{Kyoto University}, \orgaddress{\street{Yoshida Ushinomiya-cho, Sakyo-ku}, \city{Kyoto}, \postcode{6068501}, \country{Japan}}}

\abstract{The stability of persistent homology has led to wide applications of the persistence diagram as a trusted topological descriptor in the presence of noise. However, with the increasing demand for high-dimension and low-sample-size data processing in modern science, it is questionable whether persistence diagrams retain their reliability in the presence of high-dimensional noise. This work aims to study the reliability of persistence diagrams in the high-dimension low-sample-size data setting. By analyzing the asymptotic behavior of persistence diagrams for high-dimensional random data, we show that persistence diagrams are no longer reliable descriptors of low-sample-size data under high-dimensional noise perturbations. We refer to this loss of reliability of persistence diagrams in such data settings as the \textit{curse of dimensionality on persistence diagrams}. Next, we investigate the possibility of using normalized principal component analysis as a method for reducing the dimensionality of the high-dimensional observed data to resolve the curse of dimensionality. We show that this method can mitigate the curse of dimensionality on persistence diagrams. Our results shed some new light on the challenges of processing high-dimension low-sample-size data by persistence diagrams and provide a starting point for future research in this area.}



\maketitle

\section{Introduction}\label{sec1}
\subsection{Background}\label{Background and notation}
Topological data analysis (TDA) has arisen in recent decades, combining the topological method with data analysis \cite{MR2476414}. It is now widely applied in many fields, such as medical science \cite{loughrey2021topology}, \cite{ellis2019feasibility}, \cite{bukkuri2021applications}, and materials science \cite{hiraoka2016hierarchical}, \cite{torshin2020topological}.

Let $P$ be a finite set of points in the $d$-dimensional Euclidean space $\mathbb{R}^{d}$, equipped with the usual Euclidean metric. The finite set $P$ is termed as a point cloud in the literature. An abstract simplicial complex $\mathcal{K}$ with vertices in $P$ is a finite collection of subsets of $P$ such that $\sigma\in \mathcal{K}$ and $\tau\subseteq\sigma$ implies $\tau\in \mathcal{K}$. A filtration of $P$ is a strictly increasing sequence of abstract simplicial complexes with vertices in $P$: 
\begin{equation*}
\mathcal{K}_{0}\subsetneqq\mathcal{K}_{1}\subsetneqq\cdots \subsetneqq \mathcal{K}_{t}.
\end{equation*}
Let $0\leq r_{0}<r_{1}<\cdots<r_{t}$ be a strictly increasing sequence of real parameters. In persistent homology, two filtrations are commonly considered. One is called the \v{C}ech filtration, in which the abstract simplicial complex $\mathcal{K}_{t}$ is a collection of subsets of $P$ such that the closed balls with radius $r_{t}$ centered at points in each subset have a non-empty intersection:
\begin{equation*}
\mathcal{K}_{t} = \setc*{\sigma \subseteq P}{\bigcap_{x\in \sigma }B_{x}\left( r_{t}\right) \neq \emptyset}.
\end{equation*}
Here $B_{x}\left( r_{t}\right)$ denotes the closed ball with radius $r_{t}$ centered at $x$. Another one is called the Vietoris-Rips filtration (Rips filtration for short), in which the abstract simplicial complex $\mathcal{K}_{t}$ is defined as:
\begin{equation*}
\mathcal{K}_{t} = \setc*{\sigma \subseteq P}{B_{x}\left( r_{t}\right)\cap B_{y}\left( r_{t}\right)\neq \emptyset\ \text{for all}\ x,y\in \sigma}.
\end{equation*}
It is worth pointing out that computing the Rips filtration is more efficient than the \v{C}ech filtration because only pairwise intersections need to be checked in the Rips filtration case.

By taking the $N$-th homology with the field $\mathbb{F}$ coefficient, the filtration provides a sequence of $\mathbb{F}$-vector spaces connected by linear maps that are induced from inclusions:
\begin{equation*}
\Homl{N}\left(\mathcal{K}_{0};\mathbb{F}\right)\rightarrow\Homl{N}\left(\mathcal{K}_{1};\mathbb{F}\right)\rightarrow \cdots \rightarrow\Homl{N}\left(\mathcal{K}_{t};\mathbb{F}\right).
\end{equation*}
Such a sequence is then called the $N$-th persistent homology. This idea goes back as far as \cite{MR1949898}. Persistent homology records topological features of the point cloud, encoded and visualized by utilizing the persistence diagram.

Roughly speaking, the $N$-th persistence diagram collects all the $N$-th homological classes (these classes will be called the generators later) that are born and persist in the filtration. It is defined as a multi-set of points in the extended real plane $\bar{\mathbb{R}}^{2}=(\mathbb{R}\cup\{\pm\infty\})^{2}$. The point $\omega = (r_{i},r_{j})$ along with its multiplicity $\mu_{ij}^{N}$ ($i<j$) in the persistence diagram indicates that there are $\mu_{ij}^{N}$ generators that are born at $\Homl{N}\left(\mathcal{K}_{i};\mathbb{F}\right)$ but die at $\Homl{N}\left(\mathcal{K}_{j};\mathbb{F}\right)$. For a detailed introduction, we refer the reader to \cite{edelsbrunner2010computational}. In the literature, the pair $\omega = (r_{i},r_{j})$ is called a birth-death pair in the $N$-th persistence diagram; the first coordinate $r_{i}$ is called the birth time, denoted by $b_{\omega}$; the second coordinate $r_{j}$ is called the death time, denoted by $d_{\omega}$; their difference $(r_{j}-r_{i})$ characterizes the lifetime of the generator $\omega$ and thus this quantity is called the persistence of $\omega$, denoted by ${\rm pers}(\omega)$.

In practical data analysis, the true data we aim to analyze will be referred to as the original point cloud, and it is typically latent. However, technical noises during data collection may result in perturbations to the original point cloud. The data we actually get or observe will be referred to as the observed point cloud. The stability of persistent homology states that if the original data change slightly, the resultant persistence diagrams will not change too much \cite{MR2279866}. This property ensures that the persistence diagram obtained from the observed point cloud is trustworthy if the noise perturbs the original point cloud slightly. Therefore, we can use the persistence diagram obtained from the observed point cloud as a descriptor to capture the original topological features.

To avoid the inconvenience of considering generators that never die, we shall consider the reduced homology throughout this paper. Thus, the persistence diagram is defined as a multi-set on $\nnegR^{2}$, where $\nnegR = \setc*{x\in \mathbb{R}}{x\geq 0}$. In what follows, we denote the $N$-th persistence diagram of $P$ by ${\rm D}_{N}(P)$. For clarity, we will use $\Bar{{\rm D}}_{N}(P)$ to denote the disjoint union of ${\rm D}_{N}(P)$ and all points on the diagonal $\setc*{(x,x)}{x \in \nnegR}$, counted with infinite multiplicity. We set the dimension $N$ of the persistence diagram to be the non-negative integer less than or equal to the cardinality of $P$ minus $2$ to avoid the triviality throughout this paper. Moreover, if the dimension $N$ is invisible in the notation of the persistence diagram later, we would then refer to every non-trivial $N$ without specifying.

\subsection{Motivation}\label{motivation}
In recent decades, analyzing data with high dimensionality and low sample size has become a challenging problem in statistics. This type of data, called high-dimension low-sample-size (HDLSS) data, is prevalent in biomedical science \cite{HaoKimKimKang2018, MahmudFuHuangMasud2019, imoto2022resolution}. As the name suggests, the dimension $d$ is significantly larger than the sample size $n$. In molecular biology, for example, single-cell sequencing data contain thousands of single cells and over ten thousand genes or epigenomic features, corresponding to the sample size $n$ and the dimension $d$, respectively. Such genomic data are classified as HDLSS data \cite{imoto2022resolution}. Recently, \cite{Flores-Bautista2023.07.28.551057} applied the persistent homology to the single-cell RNA-sequencing data to understand cell differentiation.

The limitations of traditional multivariate statistical analysis in ensuring consistency in the statistical inference of HDLSS data have been well documented \cite{MR2568176}. The inconsistencies in inference resulting from HDLSS data are referred to as the ``curse of dimensionality". To address this challenge, high-dimensional statistics emerged from the ground as an important branch of statistical data science, with the aim of mitigating or even eliminating the curse of dimensionality.

Statistical analysis suffers the curse of dimensionality, which has raised questions about the reliability of using persistent homology in analyzing HDLSS data. In addition, the potential benefits of using persistent homology in practical single-cell sequencing data analysis are too great to be ignored. Consequently, we are highly motivated to explore the feasibility of using persistent homology in such a situation.

\subsection{Related work}
\cite{damrich2023persistent} investigated the applicability of persistent homology for high-dimensional data. They obtained an experimental observation of the disappearance of the ground-truth hole in the persistence diagram using the Euclidean distance in the Rips filtration for high-dimensional data and concluded that the standard persistent homology fails to detect the holes correctly. Their idea of evaluating the efficiency of detecting true holes in the persistent homology procedure, called the hole detection score, is to estimate the relative gap between the $m$-th and $(m+1)$-th most persistent features in the persistence diagram. To get a higher hole detection score even in the high-dimensional data context, they suggest using spectral distances, such as effective resistance and diffusion distance, rather than the conventional Euclidean distance for persistent homology.

\subsection{Settings}\label{settings}
In application, the HDLSS setting reflects the significant increase in the number of features rather than the number of samples with the recent rapid development of data correction. For example in biology, single-cell RNA sequencing detects RNA counts for over $20{,}000$ genes \cite{imoto2022resolution}. Furthermore, recent single-cell sequencing, called single-cell multiomics, additionally extracts epigenome features like chromatin accessibility with over one hundred thousand dimensions \cite{Lee2020}. However, many genome and epigenome features, such as housekeeping genes, maintain cell status by preserving constant values. Namely, such features do not contribute to constructing the topology of the original point cloud, but they contain noises in the observed point cloud due to the random sampling process of molecules. Therefore, we can consider a fixed lower dimension as the number of genome and epigenome features except for the abovementioned features by centralizing the original point cloud.

Taking into account these situations and for well-fitting the practical application, we let $P = \{x_{1},x_{2},\ldots,x_{n}\}$ be the original point cloud in $\mathbb{R}^{d}$, where $x_{i}=(x_{1i},\ldots,x_{si},0,\ldots,0)^{T}$ for $i = 1, 2, \ldots, n$. Here $n$ is called the sample size ($3< n< d$), $d$ is called the dimension, and $s$ is called the essential dimension, representing the number of genome and epigenome features except for the constant-preserving features. The essential dimension $s$ is set to be independent of $d$ and less than $n$. 

Let $E=\{e_{1},e_{2},\ldots,e_{n}\}$ be the noise point cloud in $\mathbb{R}^{d}$, and suppose $e_{1},e_{2},\ldots,e_{n}$ are independent and identically distributed (\iid\ for short) random vectors following a continuous distribution with mean zero and covariance matrix $\nu\cdot I_{d}$, where $\nu>0$ and $I_{d}$ denotes the $d\times d$ identity matrix. We further assume that the coordinates of every point in $E$ are \iid\ as well. In this work, ``continuous" means the absolute continuity of the probability distribution with respect to the Lebesgue measure.

Let $P'=\{x_{1}',x_{2}',\ldots,x_{n}'\}$ be the observed point cloud, defined as the sum of $P$ and $E$, i.e., $x_{i}'=x^{}_{i}+e^{}_{i}$ for $i = 1, 2, \ldots, n$. The terminology HDLSS used in this paper means we consider the asymptotic behavior in the setting where $d$ tends to infinity but $n$ is fixed. It is important to investigate the difference between the persistence diagrams of the original point cloud $P$ and observed point cloud $P'$ in the asymptotic sense to assess the reliability of the persistence diagram in this context.

From now on, persistence diagrams obtained from the original and the observed point clouds will be called the original and the observed persistence diagrams, respectively, to shorten the names.

\subsection{Our contributions}\label{contributions}
The aim of this paper is to theoretically study whether the persistence diagram is still a reliable topological descriptor of the HDLSS data. To this end, we measure the difference between the original and the observed persistence diagrams by utilizing the bottleneck distance (denoted by $d_{B}$), and the Hausdorff distance (denoted by $d_{H}$). For details about $d_{B}$ and $d_{H}$, the reader may refer to Section~\ref{Stability}.

\begin{namedthm*}{Main Result A}[Theorems~\ref{thm1b}, \ref{db will not converge to 0 in probability}, \ref{dHunboundedinprobability}]
Let $P$ and $P'$ be the original and the observed point clouds in $\mathbb{R}^{d}$, respectively. Then for the Rips filtration,
\begin{align*}
  d_{B}\left(\bar{{\rm D}}_{N}\left(P\right), \bar{{\rm D}}_{N}\left( P'\right) \right) = 
    \begin{cases}
        \frac{\sqrt{2\nu d}}{4}+O_{\mathbb{P}}(1),& \mbox{if $N=0$},\\
         O_{\mathbb{P}}(1), & \mbox{if $N>0$},
    \end{cases}
\end{align*}
holds as $d\rightarrow \infty$. Here and subsequently, $O_{\mathbb{P}}$ denotes the asymptotic notation (see Section~\ref{Geo}). Moreover, suppose that the $N$-th persistence diagram ${\rm D}_{N}\left(P\right)\neq \emptyset$. Then $d_{B}\left(\bar{{\rm D}}_{N}\left(P\right), \bar{{\rm D}}_{N}\left( P'\right) \right)$ does not converge to zero in probability as $d\rightarrow \infty$, and $d_{H}\left({\rm D}_{N}\left(P\right), {\rm D}_{N}\left( P'\right) \right)$ is eventually unbounded in probability (see Definition~\ref{def0b}) as $d\rightarrow \infty$.
\end{namedthm*}

Main Result A shows that the observed persistence diagram moves far away from the original one (due to the unbounded Hausdorff distance), within a bounded region from the diagonal (due to the bounded bottleneck distance) as $d$ tends to infinity. See Fig.~\ref{fig:PDnoises} for an illustration.

\begin{figure}
    \centering
    \includegraphics[width=119mm]{"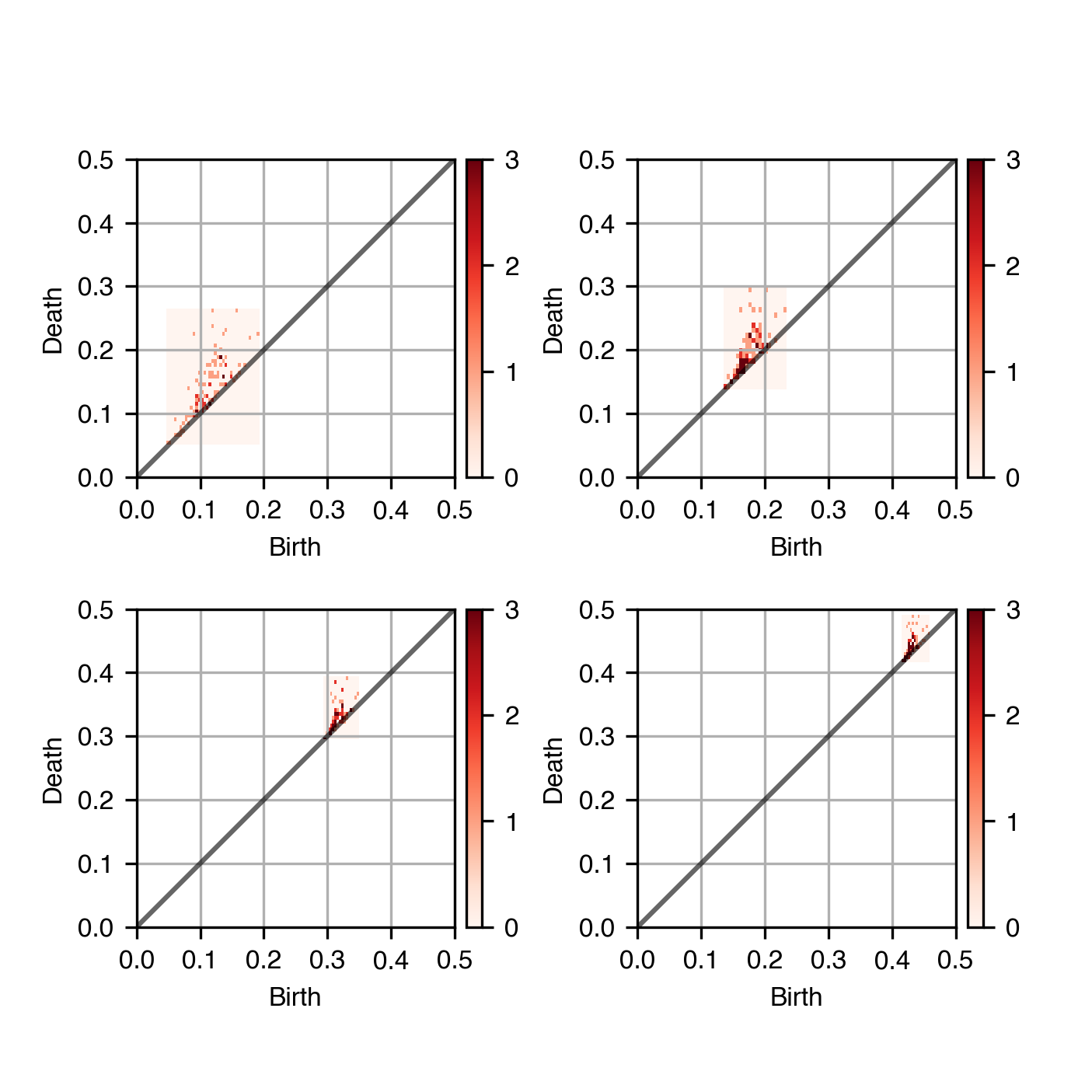"}
    \caption{Numerical results for comparison between the $1$st original persistence diagram (upper left) and the $1$st observed persistence diagrams in different dimensions (upper right: $d=1000$; lower left: $d=5000$; lower right: $d=10{,}000$). Different colors indicate different multiplicities of generators in persistence diagrams.}
    \label{fig:PDnoises}
\end{figure}

\begin{namedthm*}{Main Result B}[Theorems~\ref{thm1d}, \ref{cech strong inconsistency}]
Let $P$ and $P'$ be the original and the observed point clouds in $\mathbb{R}^{d}$, respectively. Then for the \v{C}ech filtration,
\begin{align*}
  d_{B}\left(\bar{{\rm D}}\left(P\right), \bar{{\rm D}}\left( P'\right) \right) = O_{\mathbb{P}}(\sqrt{d})
\end{align*}
holds as $d\rightarrow \infty$. Moreover, $d_{B}\left(\bar{{\rm D}}\left(P\right), \bar{{\rm D}}\left( P'\right) \right)$ is eventually unbounded in probability as $d\rightarrow \infty$.
\end{namedthm*}

Main Result B exhibits a different behavior from Main Result A, claiming that the observed persistence diagram moves far away from the original one and deviates from the diagonal as $d$ tends to infinity. Both Main Results A and B show inconsistencies in the observed persistence diagram when the dimension $d$ is large enough. The inconsistency phenomenon is called the curse of dimensionality on persistence diagrams. We note that the inconsistency demonstrated by Main Result B is more severe than that of Main Result A. For this reason, we further propose a classification based on similarity in Section~\ref{subsec4}.

We are thus led to consider ways to mitigate the curse of dimensionality. Our attempt made here is a dimensionality reduction technique called the normalized principal component analysis, a variation of principal component analysis (PCA for short). Briefly speaking, we will apply this technique to the observed point cloud $P'$ by projecting it onto the principal subspace spanned by the first $s$ principal components, with the use of normalized principal component scores as the projected coordinates. The newly obtained point cloud after applying the normalized PCA to $P'$ is called the compressed point cloud and denoted by $\hat{P}$.

Our last result is proved for the Rips filtration, under the assumption that $e_{1},e_{2},\ldots,e_{n}$ are \iid\ random vectors following the standard Gaussian distribution $\mathcal{N}_{d}(0, I_{d})$.

\begin{namedthm*}{Main Result C}[Theorems~\ref{bottleneckdistofreducedPD}, \ref{hausdorffdistofreducedPD}]
Let $P$ and $P'$ be the original and the observed point clouds in $\mathbb{R}^{d}$, respectively. Let $\hat{P}$ denote the compressed point cloud of $P'$ by normalized PCA. Then for the Rips filtration,
$$d_{B}(\bar{{\rm D}}(P), \bar{{\rm D}}(\hat{P})) = O_{\mathbb{P}}(1)$$
holds as $d\rightarrow \infty$. Moreover, suppose that ${\rm D}_{N}(P)$ is non-empty and that ${\rm D}_{N}(\hat{P})$ is non-empty almost surely. Then
$$d_{H}({\rm D}_{N}(P), {\rm D}_{N}( \hat{P})) = O_{\mathbb{P}}(1)$$
holds as $d\rightarrow \infty$.
\end{namedthm*}

We prove Main Result C by firstly showing that the persistence diagrams of compressed point clouds $\hat{P}$ and $\hat{E}$ are getting close to each other with high probability as $d$ tends to infinity (Theorem~\ref{convergeresult}). This fact makes us turn to compare the persistence diagrams of $P$ and $\hat{E}$. As a side result during the proof, we partially solve the eigengap problem of the real Wishart distributed matrix in random matrix theory under the HDLSS framework (Theorems~\ref{corethm}, \ref{maintheorem02}, \ref{maintheorem01}). We proceed with the proof by evaluating the first and the second moments of squared pairwise distances of $\hat{E}$ (Propositions~\ref{1st and 2nd moments of pairwise distance}, \ref{asymptomatic behavior of pairwise distance of reduced noise}). Combining these results, we are able to prove Main Result C.

Comparing Main Result A with Main Result C, it is evident that the use of normalized PCA can mitigate the curse of dimensionality on persistence diagrams because the Hausdorff distance becomes eventually bounded in probability even $d$ tends to infinity. See Fig.~\ref{fig:redPD} for an illustration.

\begin{figure}
    \centering
    \includegraphics[width=119mm]{"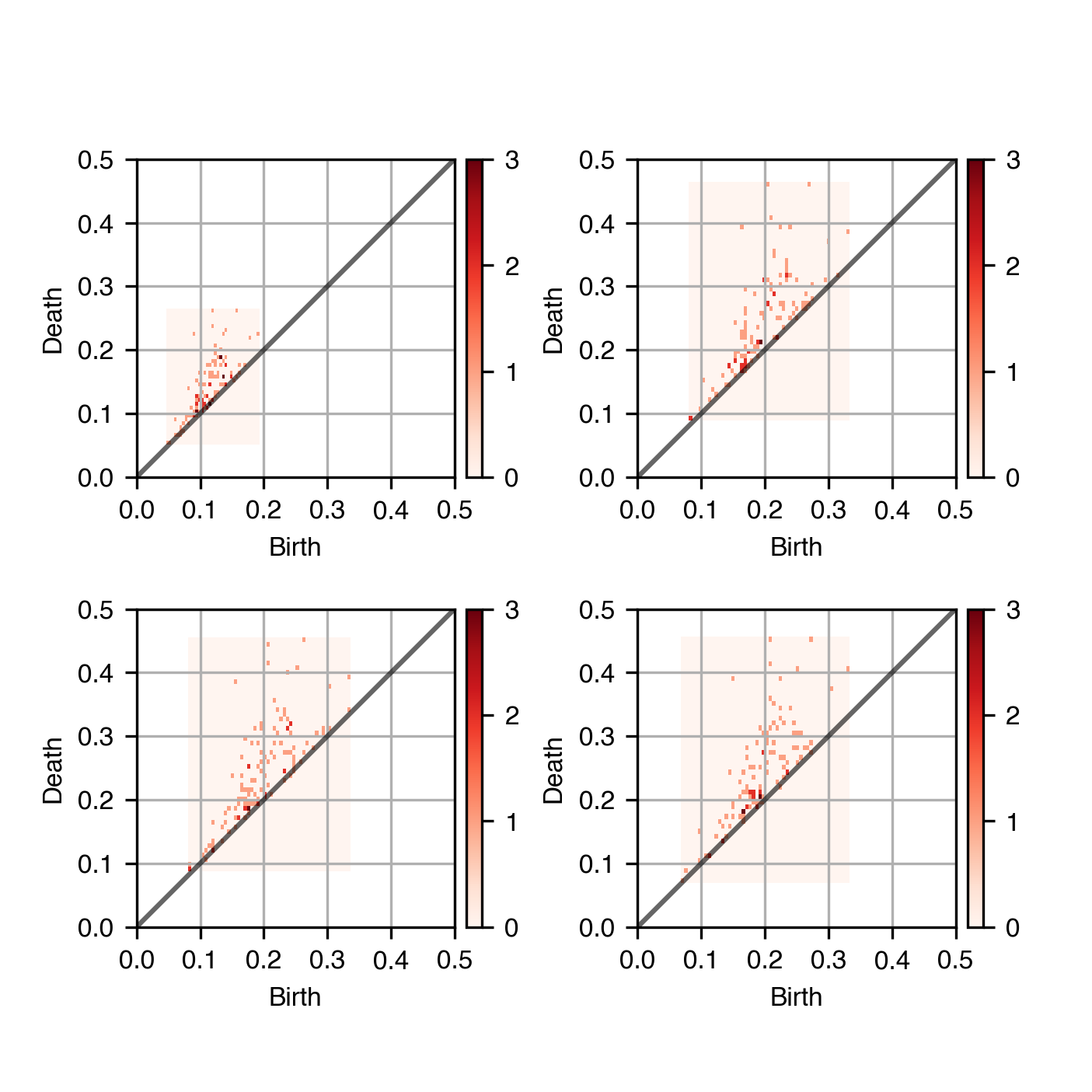"}
    \caption{Numerical results for comparison between the $1$st original persistence diagram (upper left) and the $1$st compressed persistence diagrams in different dimensions (upper right: $d=1000$; lower left: $d=5000$; lower right: $d=10{,}000$). Different colors indicate different multiplicities of generators in persistence diagrams.}
    \label{fig:redPD}
\end{figure}

\subsection{Organization}\label{Organization}
In Section~\ref{preliminaries}, we review the stability of persistent homology for geometric complexes and then introduce the geometric representation theorem of high-dimensional random vectors. In Section~\ref{sec3}, we analyze the asymptotic behavior of the observed persistence diagram (Theorem~\ref{thm0a} for the Rips filtration case and Theorem~\ref{thm0b} for the \v{C}ech filtration case). Based on this analysis, we will show that the observed persistence diagram is inconsistent with the original one as the dimension $d$ tends to infinity (Theorems~\ref{thm1b}, \ref{db will not converge to 0 in probability}, \ref{dHunboundedinprobability} for the Rips filtration case and Theorems~\ref{thm1d}, \ref{cech strong inconsistency} for the \v{C}ech filtration case). In Section~\ref{sec4}, we examine the use of normalized PCA to the observed point cloud to mitigate the curse of dimensionality on persistence diagrams. By proving the asymptotically infinite divergence of the minimum eigengap of real Wishart matrices (Theorems~\ref{corethm}, \ref{maintheorem01}), we demonstrate that the PCA result is mainly dominated by noise (Theorems~\ref{convergeresult}, \ref{PCAPDlimit}). Through detailed computations utilizing Weingarten calculus, we conclude that normalized PCA can improve the consistency level (Theorems~\ref{bottleneckdistofreducedPD}, \ref{hausdorffdistofreducedPD}). Finally, we present unsolved problems and future directions for research in Section~\ref{sec9}.   

\section{Preliminaries}\label{preliminaries}
Throughout the paper, the sets of non-negative integers and positive integers will be denoted by $\mathbb{N}$ and $\mathbb{N}_{+}$, respectively. Let us denote by $\mathbb{R}^{m\times n}$ the set of real $m\times n$ matrices, and $[n]$ the set of $n$ consecutive integers $\{1,2,\ldots,n\}$ for $n\in \mathbb{N}_{+}$. We will drop the subscript of $\left\|\cdot\right\|_{2}$ and write $\left\|\cdot\right\|$ instead to denote the $2$-norm of vectors for simplicity.

\subsection{Stability of persistent homology for geometric complexes}\label{Stability}

Stability is a crucial property of persistent homology. It guarantees that the persistence diagram is stable: small changes in the input data imply only small changes in the diagram. The study of stability can be traced back to \cite{MR2279866}. In this work, we will mainly use the stability theorem for geometric complexes. For the convenience of the reader, we repeat the relevant material from \cite{chazal2014persistence} without proofs, thus making our exposition self-contained.

We begin this subsection by introducing two metrics on persistence diagrams.

\begin{definition}\label{bottleneck}
\normalfont Let $P$ and $P'$ be two point clouds. The bottleneck distance between $\Bar{{\rm D}}(P)$ and $\Bar{{\rm D}}(P')$ is defined by
\begin{equation*}
    d_{B}(\Bar{{\rm D}}(P),\Bar{{\rm D}}(P'))=\inf _{\eta :\bar{{\rm D}}\left( P\right) \rightarrow \bar{{\rm D}}\left( P'\right) }\sup _{x\in \bar{{\rm D}}\left( P\right)}\left\| x-\eta \left( x\right) \right\| _{\infty },
\end{equation*}
where $\eta$ ranges over all bijections from $\Bar{{\rm D}}(P)$ to $\Bar{{\rm D}}(P')$, and $\left\| \cdot\right\| _{\infty}$ denotes the infinity norm. 
\end{definition}

The Hausdorff distance is a common metric defined on subsets of any metric space. Let $A$ and $B$ be two non-empty subsets of a metric space $(M,\mathcal{D})$. The Hausdorff distance between $A$ and $B$ is given by
\begin{equation*}
    d_{H}\left(A,B \right) =\max \left\{ \sup _{a\in  A }\inf _{b\in B} \mathcal{D}(a,b),\sup _{b\in B}\inf _{a\in  A}\mathcal{D}(a,b)\right\}.
\end{equation*}
We further define that $d_{H}\left(A,\emptyset\right)=d_{H}\left(\emptyset,A\right)=+\infty$ if $A$ is non-empty, and $d_{H}\left(\emptyset,\emptyset\right)=0$. 

In this paper, when referring to the Hausdorff distance between persistence diagrams, the ambient metric space $(M,\mathcal{D})$ is $\nnegR^{2}$ with the metric induced by the infinity norm $\left\| \cdot\right\| _{\infty}$. Hence, the Hausdorff distance between persistence diagrams is defined below.

\begin{definition}\label{Hausdorff}
\normalfont Let $P$ and $P'$ be two point clouds. The Hausdorff distance between $\Bar{{\rm D}}(P)$ and $\Bar{{\rm D}}(P')$ is defined by
\begin{equation*}
    d_{H}\left( \bar{{\rm D}}\left( P\right) ,\bar{{\rm D}}\left( P'\right) \right) =\max \left\{ \sup _{x\in \bar{{\rm D}}\left( P\right) }\inf _{y\in \bar{{\rm D}}\left( P'\right)}\left\| x-y\right\| _{\infty },\sup _{y\in \bar{{\rm D}}\left( P'\right)}\inf _{x\in \bar{{\rm D}}\left( P\right)}\left\| x-y\right\| _{\infty }\right\}.
\end{equation*}
\end{definition}

Here we would like to draw the reader's attention to the distinction between these two metrics on persistence diagrams. Since the bottleneck distance is a matching distance, it is sensitive to multiplicities. On the other hand, from the definition of the Hausdorff distance, it is clear that the Hausdorff distance is insensitive to multiplicities and only depends on the position of the point in the extended real plane. 

As another easy observation, the Hausdorff distance between persistence diagrams is bounded from above by the bottleneck distance.

\begin{proposition}\label{relofdB&dH}
Let $P$ and $P'$ be two point clouds. Then
\begin{equation*}
    d_{H}\left( \bar{{\rm D}}\left( P\right) ,\bar{{\rm D}}\left( P'\right) \right)\leq d_{B}\left( \bar{{\rm D}}\left( P\right) ,\bar{{\rm D}}\left( P'\right) \right).
\end{equation*}
\end{proposition}

\begin{proof}[Proof.\nopunct]
For every bijection $\eta:\bar{{\rm D}}\left( P\right)\rightarrow \bar{{\rm D}}\left( P'\right)$, we have
\begin{equation*}
    \inf _{y\in \bar{{\rm D}}\left( P'\right) }\left\| x-y\right\| _{\infty }\leq \| x-\eta \left( x\right) \| _{\infty }\leq \sup _{x\in \bar{{\rm D}}\left( P\right) }\left\| x-\eta \left( x\right) \right\| _{\infty }.
\end{equation*}
Taking the infimum of $\eta$ yields
\begin{equation*}
    \inf _{y\in \bar{{\rm D}}\left( P'\right) }\left\| x-y\right\| _{\infty }\leq \inf _{\eta :\bar{{\rm D}}\left( P\right) \rightarrow \bar{{\rm D}}\left( P'\right) }\sup _{x\in \bar{{\rm D}}\left( P\right) }\left\| x-\eta \left( x\right) \right\| _{\infty }.
\end{equation*}
Taking the supremum of $x$ yields
\begin{equation*}
    \sup _{x\in \bar{{\rm D}}(P)}\inf _{y\in \bar{{\rm D}}\left( P'\right) }\left\| x-y\right\| _{\infty }\leq \inf _{\eta :\bar{{\rm D}}\left( P\right) \rightarrow \bar{{\rm D}}\left( P'\right) }\sup _{x\in \bar{{\rm D}}\left( P\right) }\left\| x-\eta \left( x\right) \right\| _{\infty }.
\end{equation*}
By the symmetry, it is similar to show that
\begin{equation*}
    \sup _{y\in \bar{{\rm D}}(P')}\inf _{x\in \bar{{\rm D}}\left( P\right) }\left\| x-y\right\| _{\infty }\leq \inf _{\eta :\bar{{\rm D}}\left( P\right) \rightarrow \bar{{\rm D}}\left( P'\right) }\sup _{x\in \bar{{\rm D}}\left( P\right) }\left\| x-\eta \left( x\right) \right\| _{\infty }.
\end{equation*}
Therefore, $d_{H}\left( \bar{{\rm D}}\left( P\right) ,\bar{{\rm D}}\left( P'\right) \right)\leq d_{B}\left( \bar{{\rm D}}\left( P\right) ,\bar{{\rm D}}\left( P'\right) \right)$.
\end{proof}

In order to state the stability theorem, it is also necessary to introduce how to measure the distance between point clouds. One such measure is the Hausdorff distance, as introduced above. When referring to the Hausdorff distance between point clouds, the ambient metric space $(M,\mathcal{D})$ is the Euclidean space $\mathbb{R}^{d}$ with the Euclidean metric. Another measure to be introduced next is the Gromov-Hausdorff distance.

\begin{definition}\label{def0e}
\normalfont Let $X$ and $Y$ be metric spaces. The Gromov-Hausdorff distance $d_{GH}$ is defined to be the infimum of all numbers $d_{H}(\phi(X),\psi(Y))$ for all metric spaces $M$ and all isometric embeddings $\phi: X \rightarrow M$ and $\psi: Y \rightarrow M$.
\end{definition}

To compute the Gromov-Hausdorff distance, it is necessary to construct a new metric space $M$ from the definition. However, this process is not easy even in simple cases. It would be more convenient to compute the Gromov-Hausdorff distance by comparing the distances within $X$ and $Y$.

\begin{definition}\label{def0c}
\normalfont Let $X$ and $Y$ be two sets. A multivalued map $\Corr: X\rightrightarrows Y$ is defined as a subset of $X\times Y$, projecting surjectively onto $X$ through the canonical projection $\pi_{X}: X\times Y\rightarrow X$. Moreover, a multivalued map $\Corr: X\rightrightarrows Y$ is called a correspondence if the canonical projection $\Corr \rightarrow Y$ is surjective. 
\end{definition}

\begin{definition}\label{def0d}
\normalfont Let $(X, d_X)$ and $(Y, d_Y)$ be metric spaces. The distortion $\dis(\Corr)$ of a correspondence $\Corr: X\rightrightarrows Y$ is defined by
\begin{align}
    \dis\left( \Corr\right) =\sup \setc*{\lvert d_{X}\left( x,x'\right) -d_{Y}\left( y,y'\right) \rvert}{\left( x,y\right) ,\left( x',y'\right) \in \Corr}. \label{defofdistortion}
\end{align}
\end{definition}

\cite{burago2001course} reformulated the definition of the Gromov-Hausdorff distance by using distortion:
\begin{align}
    d_{GH}\left( X,Y\right) =\dfrac{1}{2}\inf_{\Corr: X\rightrightarrows Y} \dis\left( \Corr\right).\label{GromovHausdorffdistdef}
\end{align}

Now, we are equipped with all the necessary materials to present the stability theorem for geometric complexes. 

\begin{theorem}[\cite{chazal2014persistence}]\label{lem0b}
Let $P$ and $P'$ be two point clouds in a Euclidean space. Then
\begin{align}
    d_{B}\left(\bar{{\rm D}}(P), \bar{{\rm D}}(P')\right) \leq 2d_{GH}(P,P')\label{stabilityforRipsfiltration}
\end{align}
holds for the Rips filtration, and
\begin{align}
    d_{B}\left(\bar{{\rm D}}(P), \bar{{\rm D}}(P')\right) \leq d_{H}(P,P')
\end{align}
holds for the \v{C}ech filtration.
\end{theorem}

\begin{remark}
\normalfont In fact, there are two different types of \v{C}ech complex introduced in \cite{chazal2014persistence}, namely, the intrinsic \v{C}ech complex and the ambient \v{C}ech complex. In our work, the \v{C}ech complex means an ambient \v{C}ech complex ${\rm \check{C}ech}(L,\mathbb{R}^{d}; a)$, where $L$ is a point cloud in $\mathbb{R}^{d}$ and $a$ is the radius parameter.
\end{remark}

\subsection{Geometric representation of high-dimensional random vectors}\label{Geo}
We start to explain the small \emph{o} and the big \emph{O} notations for stochastic sequences. In what follows, $\mathbb{P}$ denotes the probability measure, $\mathbb{E}$ and $\mathbb{V}$ denote the expected value and variance, respectively.

\begin{definition}\label{def0a}
\normalfont A sequence of random variables $\{X_{m}\}_{m\in \mathbb{N}_{+}}$ is said to converge to zero in probability, written $X_{m}=o_{\mathbb{P}}(1)$, if for every $\varepsilon>0$, $\lim_{m\rightarrow \infty }\mathbb{P}\left( \lvert X_{m}\rvert \leq \varepsilon \right) =1$.
\end{definition}

\begin{remark}
\normalfont If $\{b_{m}\}_{m\in \mathbb{N}_{+}}$ is another sequence of random variables, $X_{m}=o_{\mathbb{P}}(b_{m})$ means $X_{m}=b_{m}Y_{m}$ for some $Y_{m}=o_{\mathbb{P}}(1)$.
\end{remark}

To motivate the definition of the big \emph{O} notation, it is useful to restate the above definition in the following way. We say $X_{m}=o_{\mathbb{P}}(1)$ if and only if for every positive $\varepsilon$ and $\eta$, there exists an integer $M_{(\varepsilon,\eta)}\in \mathbb{N}_{+}$ such that $\mathbb{P}\left( \lvert  X_{m}\rvert \leq \eta \right) \geq 1-\varepsilon$ holds for every $m>M_{(\varepsilon,\eta)}$.

\begin{definition}\label{def0b}
\normalfont A sequence of random variables $\{X_{m}\}_{m\in \mathbb{N}_{+}}$ is said to be eventually bounded in probability, written $X_{m}=O_{\mathbb{P}}(1)$, if for every $\varepsilon>0$, there exist a constant $C_{(\varepsilon)}$ and a positive integer $M_{(\varepsilon)}$ such that $\mathbb{P}\left( \lvert  X_{m}\rvert \leq C_{(\varepsilon)} \right) \geq 1-\varepsilon$ holds for every $m>M_{(\varepsilon)}$.
\end{definition}

\begin{remark}
\normalfont If $\{b_{m}\}_{m\in \mathbb{N}_{+}}$ is another sequence of random variables, $X_{m}=O_{\mathbb{P}}(b_{m})$ means $X_{m}=b_{m}Y_{m}$ for some $Y_{m}=O_{\mathbb{P}}(1)$.
\end{remark}

We now provide some elementary properties of asymptotic notations that can be derived easily from the definitions.

\begin{proposition}[\cite{MR1652247}]\label{prop0a}
The notations $o_{\mathbb{P}}$ and $O_{\mathbb{P}}$ satisfy the following rules of calculus.
\begin{itemize}
  \item [1)] 
  $o_{\mathbb{P}}(1)\pm o_{\mathbb{P}}(1)=o_{\mathbb{P}}(1)$ and $o_{\mathbb{P}}(1)\cdot o_{\mathbb{P}}(1)=o_{\mathbb{P}}(1)$.      
  \item [2)]
  $O_{\mathbb{P}}(1)\pm O_{\mathbb{P}}(1)=O_{\mathbb{P}}(1)$ and $O_{\mathbb{P}}(1)\cdot O_{\mathbb{P}}(1)=O_{\mathbb{P}}(1)$.
  \item [3)]
  $o_{\mathbb{P}}(1)\pm O_{\mathbb{P}}(1)=O_{\mathbb{P}}(1)$ and $o_{\mathbb{P}}(1)\cdot O_{\mathbb{P}}(1)=o_{\mathbb{P}}(1)$.
  \item [4)]
  $\left( 1+o_{\mathbb{P}}\left( 1\right) \right) ^{-1}=O_{\mathbb{P}}\left( 1\right)$.
  \item [5)]
  For every $\gamma>0$, $O_{\mathbb{P}}(m^{-\gamma})=o_{\mathbb{P}}(1)$ as $m\rightarrow \infty$.
\end{itemize}
\end{proposition}

\begin{remark}
\normalfont The equal signs in these equalities indicate that the left-hand side implies the right-hand side.
\end{remark}

\begin{proposition}\label{prop0b}
  Let $\{X_{m}\}_{m\in \mathbb{N}_{+}}$ and $\{Y_{m}\}_{m\in \mathbb{N}_{+}}$ be two sequences of non-negative random variables. Suppose that $X_{m}\leq Y_{m}$ holds for every $m\in \mathbb{N}_{+}$. Then as $m\rightarrow \infty$,
  \begin{itemize}
      \item [1)]
       $Y_{m}=o_{\mathbb{P}}(1)$ implies $X_{m}=o_{\mathbb{P}}(1)$.
      \item [2)]
       $Y_{m}=O_{\mathbb{P}}(1)$ implies $X_{m}=O_{\mathbb{P}}(1)$.
  \end{itemize}
\end{proposition}

\begin{proposition}\label{prop0c}
  Let $\{X_{m}\}_{m\in \mathbb{N}_{+}}$ be a sequence of random variables, and $f:\mathbb{N}_{+}\rightarrow \mathbb{R}$ a real function. Then as $m\rightarrow \infty$,
  \begin{itemize}
      \item [1)]
       $X_{m}=f(m)+o_{\mathbb{P}}(1)$ implies $\lvert  X_{m}\rvert=\lvert  f(m)\rvert+o_{\mathbb{P}}(1)$.
      \item [2)]
       $X_{m}=f(m)+O_{\mathbb{P}}(1)$ implies $\lvert  X_{m}\rvert=\lvert  f(m)\rvert+O_{\mathbb{P}}(1)$.
  \end{itemize}
\end{proposition}

\begin{remark}\label{Gtsdc}
\normalfont If $f$ is zero, then $X_{m}=o_{\mathbb{P}}(1)$ if and only if $\lvert  X_{m}\rvert=o_{\mathbb{P}}(1)$, and $X_{m}=O_{\mathbb{P}}(1)$ if and only if $\lvert  X_{m}\rvert=O_{\mathbb{P}}(1)$.
\end{remark}

\begin{proposition}\label{prop0d}
Let $\{X_{m}^{1}\}_{m\in \mathbb{N}_{+}}, \ldots, \{X_{m}^{n}\}_{m\in \mathbb{N}_{+}}$
be $n$ sequences of random variables and $f:\mathbb{N}_{+}\rightarrow \mathbb{R}$ a real function. Set $Y_{m} = \max_{i\in [n]} X_{m}^{i}$ and $Z_{m} =~\min_{i\in [n]}X_{m}^{i}$.
  \begin{itemize}
      \item [1)]
       If $X_{m}^{i}=f(m)+o_{\mathbb{P}}(1)$ for all $i\in [n]$, then $Y_{m}=f(m)+o_{\mathbb{P}}(1)$ and $Z_{m}=f(m)+o_{\mathbb{P}}(1)$ as $m\rightarrow \infty$.
      \item [2)]
       If $X_{m}^{i}=f(m)+O_{\mathbb{P}}(1)$ for all $i\in [n]$, then $Y_{m}=f(m)+O_{\mathbb{P}}(1)$ and $Z_{m}=f(m)+O_{\mathbb{P}}(1)$ as $m\rightarrow \infty$.
  \end{itemize}
\end{proposition}

With the use of the small \emph{o} and the big \emph{O} notations, we present the original geometric representation theorem of high-dimensional random vectors and generalize it for our needs.

\begin{theorem}[\cite{hall2005geometric}]\label{lem0a}
Let $z\in \mathbb{R}^{d}$ be a real random vector following the Gaussian distribution $\mathcal{N}_{d}(0,I_{d})$. Here and subsequently, $I_{d}$ denotes the $d\times d$ identity matrix. Then
\begin{equation}
    \left\| z\right\| =\sqrt{d}+O_{\mathbb{P}}\left( 1\right)\label{1.1}
\end{equation}
holds as $d\rightarrow \infty$. Let $w\in \mathbb{R}^{d}$ be another random vector that is identically distributed and independent of $z$. Then
\begin{align}
    \left\| z-w\right\| &=\sqrt{2d}+O_{\mathbb{P}}\left( 1\right)\label{1.2}
\end{align}
and
\begin{align}
    {\rm Angle}\left( z,w\right) &=\dfrac{\pi }{2}+O_{\mathbb{P}}\left( d^{-1/2}\right)\label{1.3}
\end{align}
hold as $d\rightarrow \infty$. Here ${\rm Angle}\left( z,w\right)$ denotes the angle, in measured radians at the origin, between vectors $z$ and $w$.
\end{theorem}

Theorem \ref{lem0a} characterizes the geometric behavior of the high-dimensional random vectors that are drawn independently from a Gaussian distribution. As the dimension $d$ increases to infinity, these $d$-dimensional random vectors tend to concentrate on a sphere centered at the origin with the radius $\sqrt{d}$, and the distance and angle between two distinct vectors tend to $\sqrt{2d}$ and $\pi/2$, respectively. However, the assumption that the random vector follows the Gaussian distribution is overly restrictive; it is necessary to extend the Gaussian distribution in this geometric representation theorem to more general probability distributions.

\begin{theorem}\label{YataGRT}
Let $z=(z_{1},z_{2},\ldots,z_{d})^{T} \in \mathbb{R}^{d}$. Suppose that $z_{1},z_{2},\ldots,z_{d}$ are \iid\ random variables with the expected value $\epsilon$ and the variance $\nu$, and ${\mathbb{V}}[(z_{1}-\epsilon) ^{2}]\neq 0$. Then
\begin{align}
    &\left\| z-\mu \right\| =\sqrt{\nu d}+O_{\mathbb{P}}\left( 1\right)\label{distancetomean1}
\end{align}
holds as $d\rightarrow \infty$, where $\mu=(\epsilon,\epsilon,\ldots ,\epsilon)^{T}\in \mathbb{R} ^{d}$. Let $w=(w_{1},w_{2},\ldots,w_{d})^{T}$ be another random vector that is identically distributed and independent of $z$. Then
\begin{align}
    & \left\| z-w\right\| =\sqrt{2\nu d}+O_{\mathbb{P}}\left( 1\right)\label{pairwisedistance1}
\end{align}
and
\begin{align}
    & \dfrac{\langle  z-\mu, w-\mu\rangle }{\left\| z-\mu\right\| \cdot \left\| w-\mu\right\| } =O_{\mathbb{P}}\left( d^{-1/2}\right)\label{innerproduct}
\end{align}
hold as $d\rightarrow \infty$. Here and subsequently, $\langle \cdot, \cdot\rangle$ denotes the scalar product. 
\end{theorem}

\begin{proof}[Proof.\nopunct]
Since $\left\| z-\mu \right\| ^{2}=\sum_{i\in [d]}( z_{i}-\epsilon) ^{2}$, we have
\begin{equation*}
    {\mathbb{E}}\left\| z-\mu \right\| ^{2}=\sum_{i\in [d]}{\mathbb{E}}\left( z_{i}-\epsilon\right) ^{2}=\sum_{i\in [d]}{\mathbb{V}}\left( z_{i}\right)=\nu d
\end{equation*}
and
\begin{equation*}
{\mathbb{V}}\left(\left\| z-\mu \right\| ^{2}\right)=\sum_{i\in [d]}{\mathbb{V}}\left( z_{i}-\epsilon\right) ^{2}=d\cdot {\mathbb{V}}\left[\left(z_{1}-\epsilon\right) ^{2}\right].
\end{equation*}
By Chebyshev's inequality, for every $\eta>0$ and $d\geq 1$,
\begin{equation*}
    \mathbb{P} \left(\bigg\lvert  \dfrac{\left\| z-\mu \right\| ^{2}}{\sqrt{d}}-\dfrac{{\mathbb{E}}\left\| z-\mu \right\| ^{2}}{\sqrt{d}}\bigg\rvert \geq \eta \right)\leq \dfrac{{\mathbb{V}}\left( \dfrac{\left\| z-\mu \right\| ^{2}}{\sqrt{d}}\right) }{\eta ^{2}}.
\end{equation*}
Combining these we obtain
\begin{equation}
\mathbb{P} \left(\bigg\lvert  \dfrac{\left\| z-\mu \right\| ^{2}}{\sqrt{d}}-\dfrac{\nu d}{\sqrt{d}}\bigg\rvert \geq \eta \right)\leq \dfrac{{\mathbb{V}}\left[\left( z_{1}-\epsilon\right) ^{2}\right]}{\eta ^{2}}.\label{PPPncvj}
\end{equation}
Let $C=\sqrt{{\mathbb{V}}[(z_{1}-\epsilon) ^{2}]}/\sqrt{\eta}$. It is clear that $C$ is positive since ${\mathbb{V}}[(z_{1}-\epsilon) ^{2}]\neq 0$. Now~\eqref{PPPncvj} beomes
$$\mathbb{P} \left(\big\lvert  \left\| z-\mu \right\| ^{2}-\nu d\big\rvert \geq C \cdot \sqrt{d} \right)\leq \eta,$$
i.e., $\left\| z-\mu \right\| ^{2} =\nu d+O_{\mathbb{P}}(\sqrt{d})$. Hence $\left\| z-\mu \right\| =\sqrt{\nu d}+O_{\mathbb{P}}\left( 1\right)$.

Because all the coordinates of $z$ and $w$ are \iid, the coordinates $z_{1}-w_{1}, z_{2}-w_{2}, \ldots, z_{d}-w_{d}$ of $z-w$ are \iid\ with expected values $0$ and variances $2\nu$. Hence~\eqref{pairwisedistance1} holds by noticing~\eqref{distancetomean1}.

Finally, we notice that
\begin{align*}
    \langle  z-\mu, w-\mu\rangle=\frac{\left\| z-\mu\right\|^{2}+\left\| w-\mu\right\|^{2}-\left\| z-w\right\|^{2}}{2}.
\end{align*}
From~\eqref{distancetomean1}, \eqref{pairwisedistance1}, we have
\begin{align*}
    \dfrac{\langle  z-\mu, w-\mu\rangle  }{\left\| z-\mu\right\| \cdot \left\| w-\mu\right\| } &=\dfrac{\left\| z-\mu\right\|^{2}+\left\| w-\mu\right\|^{2}-\left\| z-w\right\|^{2}}{2\cdot \left\| z-\mu\right\| \cdot \left\| w-\mu\right\| } \\
    &= \frac{O_{\mathbb{P}}(\sqrt{d})}{2\nu d+O_{\mathbb{P}}(\sqrt{d})}\\
    &= O_{\mathbb{P}}(d^{-1/2}). \tag*{\qedhere}
\end{align*}
\end{proof}

\section{Existence of curse of dimensionality on persistence diagrams}\label{sec3}
This section is devoted to the theoretical analysis of the asymptotic behavior of persistence diagrams. We propose that the observed persistence diagram becomes inconsistent with the original one as the dimension increases, resulting in the phenomenon called the ``curse of dimensionality on persistence diagrams". To further describe the level of inconsistency, we provide a classification of the asymptotic similarity between the original and the observed persistence diagrams at the end of this section.

\subsection{Asymptomatic behavior of the observed persistence diagram}\label{subsec2}

\subsubsection{Geometric representation of the observed point cloud}\label{grotopc}
Inspired by the previously mentioned geometric representation of high-dimensional random data, the following consequence for the observed point cloud arises in our scenario.

\begin{proposition}\label{prop1a}
Let $P'=\{x_{1}',x_{2}',\ldots,x_{n}'\}$ be the observed point cloud in $\mathbb{R}^{d}$. Then
\begin{align}
    \left\| x_{i}'\right\|=\sqrt{\nu d}+O_{\mathbb{P}}(1)\label{CFGxd}
\end{align}
holds for every $i\in [n]$ as $d\rightarrow \infty$. Furthermore,
\begin{align}
    &\left\| x_{i}'-x_{j}'\right\|=\sqrt{2\nu d}+O_{\mathbb{P}}(1)\label{obpairwisedist} 
\end{align}
and
\begin{align}
    &\frac{\langle  x_{i}', x_{j}'\rangle  }{\| x_{i}'\| \cdot \| x_{j}'\|}=O_{\mathbb{P}}\left( d^{-1/2}\right)\label{innerproductconver0}
\end{align}
hold for arbitrary distinct $i,j\in [n]$ as $d\rightarrow \infty$.
\end{proposition}

\begin{proof}[Proof.\nopunct]
For each $i\in [n]$,
\begin{align*}
    \left\| x_{i}'\right\|^{2} =\left\|x_{i}+e_{i}\right\|^{2}=\left\| x_{i} \right\|^{2}+\left\| e_{i} \right\|^{2}+2 x_{i}^{T} e_{i}^{}.
\end{align*}
Note that $\left\| x_{i}\right\|^{2}$ and $x_{i}^{T} e_{i}^{}$ only depend on $s$, and $s$ is independent of $d$ by our setting. Combining this observation with~\eqref{distancetomean1} yields
\begin{align*}
    \left\| x_{i}'\right\|^{2}=\left\| x_{i} \right\|^{2}+\left\| e_{i} \right\|^{2}+2 x_{i}^{T} e_{i}^{}= \nu d+O_{\mathbb{P}}(\sqrt{d})
\end{align*}
as $d\rightarrow \infty$. Hence $\left\| x_{i}'\right\|=\sqrt{\nu d}+O_{\mathbb{P}}(1)$ as $d\rightarrow \infty$.

For arbitrary distinct $i,j\in [n]$,
\begin{align*}
    \left\| x_{i}'-x_{j}'\right\|^{2}
    & =\left\| \left(x_{i}-x_{j}\right) +\left( e_{i}-e_{j}\right)\right\|^{2}\\
    & =\left\| x_{i}-x_{j} \right\|^{2}+\left\| e_{i}-e_{j} \right\|^{2}+2 (x_{i}-x_{j})^{T} (e_{i}-e_{j})^{}.
\end{align*}
Note that $\left\| x_{i}-x_{j}\right\|^2$ and $(x_{i}-x_{j})^{T} (e_{i}-e_{j})^{}$ only depend on $s$. From~\eqref{pairwisedistance1} we have
\begin{align*}
    \left\| x_{i}'-x_{j}'\right\|^{2}=\left\| x_{i}-x_{j} \right\|^{2}+\left\| e_{i}-e_{j} \right\|^{2}+2 (x_{i}-x_{j})^{T} (e_{i}-e_{j})^{}=2\nu d+O_{\mathbb{P}}(\sqrt{d})
\end{align*}
as $d\rightarrow \infty$. Hence $\left\| x_{i}'-x_{j}'\right\|=\sqrt{2\nu d}+O_{\mathbb{P}}(1)$ as $d\rightarrow \infty$.

Finally, combining~\eqref{CFGxd} with \eqref{obpairwisedist} yields
\begin{align*}
    \frac{\langle  x_{i}', x_{j}'\rangle }{\| x_{i}'\| \cdot \| x_{j}'\|}=\frac{\left\| x_{i}'\right\|^{2}+\left\| x_{j}'\right\|^{2}-\left\| x_{i}'-x_{j}'\right\|^{2}}{2\cdot \| x_{i}'\| \cdot \| x_{j}'\|}=\frac{O_{\mathbb{P}}(\sqrt{d})}{2 \big [\nu d+O_{\mathbb{P}}(\sqrt{d})\big ]}=O_{\mathbb{P}}\left(d^{-1/2}\right)
\end{align*}
as $d\rightarrow \infty$.
\end{proof}

Proposition \ref{prop1a} shows a tendency for the observed point cloud $P'$ to lie at the vertices of a regular simplex as $d$ tends to infinity. 

Using some basic knowledge of probability theory, it is not hard to obtain the following.

\begin{proposition}\label{LDwithprob0}
Let $E=\{e_{1},e_{2},\ldots,e_{n}\}$ and $P'=\{x_{1}',x_{2}',\ldots,x_{n}'\}$ be the noise and the observed point clouds in $\mathbb{R}^{d}$, respectively. Then
\begin{enumerate}[(1)]
    \item $e_{1},e_{2},\ldots,e_{n}$ are linearly independent almost surely.
    \item $x_{1}',x_{2}',\ldots,x_{n}'$ are linearly independent almost surely.
\end{enumerate}
\end{proposition}

\subsubsection{Rips filtration case}\label{subsubsec2}
Based on the geometric representation of the observed point cloud, we start to show the asymptotic behavior of the observed persistence diagram case by case. We first consider the Rips filtration case in this subsection.

For the convenience of stating the results, we begin with notations. Let $Q=\{y_{1},y_{2},\ldots,y_{n}\}\subseteq \mathbb{R}^{d}$ be the vertex set of the regular simplex $\bfDel^{n-1}$ with scalar $\sqrt{\nu d}$, where $y_{i}=( 0,\ldots,\underbrace{\sqrt{\nu d}}_{i\mbox{-}\rm{th}},\ldots,0)^{T}\in \mathbb{R}^{d}$ ($i\in [n]$). In what follows, the notation $Q$ abbreviates the dependency on $d$ unless there is a specific declaration. To achieve the asymptotic behavior of the observed persistence diagram, we need to know the persistent homology of $Q$.

\begin{lemma}\label{regularrips}
Let $Q\subseteq \mathbb{R}^{d}$ be the vertex set of the scaling regular simplex $\bfDel^{n-1}$. Then the $0$th persistence diagram of $Q$ is given by
\begin{equation*}
    {\rm D}_{0}\left(Q\right)=\left\{\left(0,\sqrt{2\nu d}/2\right)\right\},
\end{equation*}
where $(0,\sqrt{2\nu d}/2)$ has multiplicity $(n-1)$. For $N\in \mathbb{N}_{+}$, the $N$-th persistence diagram of $Q$ is given by
\begin{equation*}
    {\rm D}_{N}\left(Q\right)=\emptyset.
\end{equation*}
\end{lemma}

\begin{proof}[Proof.\nopunct]
The non-trivial Rips filtration of $Q$ is
\begin{equation*}
    {\rm Rips}\left( 0\right)\subsetneqq {\rm Rips}\left( \sqrt{2\nu d}/2\right),
\end{equation*}
where ${\rm Rips}(0)$ is a simplicial complex consisting of $n$ vertices and ${\rm Rips}( \sqrt{2\nu d}/2)$ is a simplicial complex consisting of all faces of $\bfDel^{n-1}$. Then the geometric realization of ${\rm Rips}( \sqrt{2\nu d}/2)$ is homeomorphic to the $(n-1)$-dimensional closed unit ball. At the homology level, we have the sequence
\begin{equation*}
    \Homl{N}\left({\rm Rips}\left( 0\right)\right)\rightarrow \Homl{N}\left({\rm Rips}\left( \sqrt{2\nu d}/2\right)\right),
\end{equation*}
from which the conclusion follows.
\end{proof}

As the dimension $d$ tends to infinity, Proposition~\ref{prop1a} reveals a high similarity between $P'$ and $Q$. This similarity between point clouds brings the stability theorem to mind, leading to the following results.

\begin{theorem}\label{thm0a}
Let $P' \subseteq \mathbb{R}^{d}$ be the observed point cloud, and $Q\subseteq \mathbb{R}^{d}$ the vertex set of the scaling regular simplex $\bfDel^{n-1}$. Then 
\begin{align*}
  d_{B}\left(\bar{{\rm D}}\left(P'\right),\bar{{\rm D}}\left(Q\right)\right)=O_{\mathbb{P}}(1)  
\end{align*}
holds as $d\rightarrow \infty$.
\end{theorem}

\begin{proof}[Proof.\nopunct]
Define a correspondence $\Corr'$$: P'\rightrightarrows Q$ by setting
\begin{equation*}
    \Corr' = \setc*{(x_{i}',y^{}_{i})\in P'\times Q}{i\in [n]}.
\end{equation*}
By the definition of Gromov-Hausdorff distance, $2d_{GH}(P',Q)\leq \dis(\Corr')$. Note that 
\begin{align*}
    \dis(\Corr') &= \max_{\substack{i,j\in [n] \\ i<j}}  \Big\lvert \left\| x_{i}'-x_{j}'\right\| -\left\| y_{i}-y_{j}\right\| \Big\rvert \\
    &= \max_{\substack{i,j\in [n] \\ i<j}}  \Big\lvert \left\| x_{i}'-x_{j}'\right\| - \sqrt{2\nu d} \Big\rvert.
\end{align*}
By Proposition \ref{prop1a} we have $\dis(\Corr')=O_{\mathbb{P}}(1)$. Finally, Theorem \ref{lem0b} now yields
\begin{equation*}
    d_{B}\left(\bar{{\rm D}}\left(P'\right),\bar{{\rm D}}\left(Q\right)\right)= O_{\mathbb{P}}(1)
\end{equation*}
as $d\rightarrow \infty$.
\end{proof}

The greater the persistence of the generator in the persistent homology, the more significant the holes in the data become. To better describe the asymptotic behavior of the persistence diagram ${\rm D}$ (In fact, a sequence of persistence diagrams), we introduce two quantities as follows. The maximum persistence, denoted by $\sup \setc*{{\rm pers}(\omega)}{\omega\in {\rm D}}$, represents the proximity of the region of all birth-death pairs to the diagonal. The maximum relative persistence, denoted by $\sup \setc*{{\rm pers}(\omega)/d_{\omega}}{\omega\in {\rm D}}$, represents the relative proximity. For the empty persistence diagram ${\rm D}$, we additionally define
\begin{equation*}
    \sup \setc*{{\rm pers}(\omega)}{\omega\in {\rm D}}=0\ \ \text{and} \ \ \sup \setc*{\frac{{\rm pers}(\omega)}{d_{\omega}}}{\omega\in {\rm D}}=0.
\end{equation*}
By ${\rm D}(\{\cdot\})$ we denote the persistence diagram of a single point in $\mathbb{R}^{d}$. We note here that ${\rm D}_{N}(\{\cdot\})=\emptyset$ for each $N\in\mathbb{N}$ in either Rips or \v{C}ech filtration case because the reduced homology is always considered as mentioned before. Using these notations and additional definitions, we have the following consequences of the theorem.

\begin{corollary}\label{suppersbounded}
Let $P' \subseteq \mathbb{R}^{d}$ be the observed point cloud. Then 
\begin{equation*}
    \sup \setc*{{\rm pers}(\omega)}{\omega\in {\rm D}_{N}(P')}=
    \begin{cases}
        \frac{\sqrt{2\nu d}}{2}+O_{\mathbb{P}}(1),& \mbox{if $N=0$},\\
        O_{\mathbb{P}}(1), & \mbox{if $N>0$},
    \end{cases}
\end{equation*}
holds as $d\rightarrow \infty$.
\end{corollary}

\begin{proof}[Proof.\nopunct]
    The proof is completed by noticing that for every persistence diagram ${\rm D}$,
    $$\sup \setc*{{\rm pers}(\omega)}{\omega\in {\rm D}} = 2d_{B}\left(\bar{{\rm D}},\bar{{\rm D}}\left(\{\cdot\}\right)\right)$$ holds and utilizing the triangle inequality
    \begin{align*}
        d_{B}\left(\bar{{\rm D}}\left(P'\right),\bar{{\rm D}}\left(Q\right)\right) &\geq \left\lvert d_{B}\left(\bar{{\rm D}}\left(P'\right),\bar{{\rm D}}\left(\{\cdot\}\right)\right)-d_{B}\left(\bar{{\rm D}}\left(Q\right),\bar{{\rm D}}\left(\{\cdot\}\right)\right)\right\rvert \\
        &=\frac{1}{2}\left\lvert \sup \setc*{{\rm pers}(\omega)}{\omega\in {\rm D}(P')}-\sup \setc*{{\rm pers}(\omega)}{\omega\in {\rm D}(Q)}\right\rvert \\
        &=\begin{cases}
            \frac{1}{2}\left\lvert \sup \setc*{{\rm pers}(\omega)}{\omega\in {\rm D}_{0}(P')}-\frac{\sqrt{2\nu d}}{2}\right\rvert,\ N=0,\\
            \frac{1}{2}\sup \setc*{{\rm pers}(\omega)}{\omega\in {\rm D}_{N}(P')},\ \forall N>0.
        \end{cases} \tag*{\qedhere}
    \end{align*}
\end{proof}

\begin{corollary}\label{cor0a}
Let $P' \subseteq \mathbb{R}^{d}$ be the observed point cloud. Then
\begin{equation*}
    \sup \setc*{\frac{{\rm pers}(\omega)}{d_{\omega}}}{\omega\in {\rm D}_{N}(P')} = 
    \begin{cases}
        1,& \mbox{if $N=0$},\\
        o_{\mathbb{P}}(1), & \mbox{if $N>0$},
    \end{cases}
\end{equation*}
holds as $d\rightarrow \infty$.
\end{corollary}

\begin{proof}[Proof.\nopunct]
The result for $N=0$ holds trivially since the birth time $b_{\omega}$ is zero for every $\omega\in {\rm D}_{0}(P')$. For the case $N>0$. Notice that the death time of each generator in the persistence diagram is half the pairwise distance. It then follows that
\begin{align*}
    \sup \setc*{\frac{{\rm pers}(\omega)}{d_{\omega}}}{\omega\in {\rm D}_{N}(P')} & \leq  \frac{\sup \setc*{{\rm pers}(\omega)}{\omega\in {\rm D}_{N}(P')}}{\min \setc*{d_{\omega}}{\omega\in {\rm D}_{N}(P')}}\\
    & \leq \frac{\sup \setc*{{\rm pers}(\omega)}{\omega\in {\rm D}_{N}(P')}}{\min \setc*{\left\|x'_{i}-x'_{j}\right\|}{i,j\in [n]\ with\ i<j}}.
\end{align*}
By~\eqref{obpairwisedist} and Corollary~\ref{suppersbounded}, we obtain
\begin{equation*}
    \frac{\sup \setc*{{\rm pers}(\omega)}{\omega\in {\rm D}_{N}(P')}}{\min \setc*{\left\|x'_{i}-x'_{j}\right\|}{i,j\in [n]\ with\ i<j}}  = \frac{O_{\mathbb{P}}(1)}{\sqrt{2\nu d}+O_{\mathbb{P}}(1)} = o_{\mathbb{P}}(1)
\end{equation*}
as $d\rightarrow \infty$. The desired formula follows by noticing Proposition~\ref{prop0b}.
\end{proof}

\subsubsection{\v{C}ech filtration case}\label{subsubsec3}
We turn to the \v{C}ech filtration case in this subsection. The following lemma gives the persistence diagram of $Q$ using the \v{C}ech filtration.

\begin{lemma}\label{regularcech}
Let $Q\subseteq \mathbb{R}^{d}$ be the vertex set of the scaling regular simplex $\bfDel^{n-1}$. Then for every $N\in \mathbb{N}$, the $N$-th persistence diagram of $Q$ is given by
\begin{equation*}
    {\rm D}_{N}\left(Q\right)=\left\{\left(\sqrt{\frac{\nu dN}{N+1}},\sqrt{\frac{\nu d(N+1)}{N+2}}\right)\right\}
\end{equation*}
with a certain multiplicity.
\end{lemma}

\begin{proof}[Proof.\nopunct]
Let $\sigma_{q}=\{y_{1},\ldots,y_{q+1}\}$ be one of $q$-faces of $\bfDel^{n-1}$ ($0\leq q\leq n-1$). The barycenter of $\sigma_{q}$, denoted by $\mathcal{B}(\sigma_{q})$, is given by
\begin{equation*}
    \mathcal{B}(\sigma_{q})=\sum_{i\in [q+1]}\dfrac{1}{q+1}y_{i}=\Bigg(\underbrace{\dfrac{\sqrt{\nu d}}{q+1},\dfrac{\sqrt{\nu d}}{q+1},\ldots,\dfrac{\sqrt{\nu d}}{q+1}}_{q+1},0,\ldots,0\Bigg)^{T}.
\end{equation*}
The distance $r_{q}$ between $y_{1}$ and $\mathcal{B}(\sigma_{q})$ is given by
\begin{align*}
    r_{q}=\|\mathcal{B}(\sigma_{q})-y_{1}\|=\sqrt{\left(-\dfrac{q\sqrt{\nu d}}{q+1}\right)^{2}+q\left(\dfrac{\sqrt{\nu d}}{q+1}\right)^{2}}=\sqrt{\dfrac{\nu dq}{q+1}}.
\end{align*}
Then the \v{C}ech filtration of $Q$ is
\begin{equation*}
    {\rm \check{C}ech}\left( r_{0}\right)\subsetneqq {\rm \check{C}ech}\left( r_{1}\right)\subsetneqq \cdots \subsetneqq {\rm \check{C}ech}\left( r_{n-1}\right),
\end{equation*}
where ${\rm \check{C}ech}\left( r_{q}\right)$ is a subcomplex consisting of all $q$-faces of $\bfDel^{n-1}$ ($0\leq q\leq n-1$). The statement follows from this observation.
\end{proof}

In statistics, it is common to regard a dataset as a data matrix. For the point cloud $P=\left\{ x_{1},x_{2},\ldots,x_{n}\right\} \subseteq\mathbb{R}^{d}$, it can be viewed as a $d$-by-$n$ matrix by placing $x_{i}$ in the $i$-th column of this matrix, i.e., $(x_{1}\ x_{2}\ \cdots\ x_{n})\in \mathbb{R}^{d\times n}$. Conversely, for a $d$-by-$n$ matrix $(x_{1}\ x_{2}\ \cdots\ x_{n})$, it can also be viewed as a subset $\left\{ x_{1},x_{2},\ldots,x_{n}\right\}\subseteq \mathbb{R} ^{d}$ containing $n$ indexing elements. Therefore, we will use the same notation $P$ to denote both a subset of $\mathbb{R}^{d}$ containing $n$ indexing elements and a $d$-by-$n$ matrix without ambiguity later.

The following two lemmas are useful in the proof of the main theorem in this subsection.

\begin{lemma}\label{normineq}
Let $A\in \mathbb{R}^{m \times n}$ and $B\in \mathbb{R}^{n \times r}$. Then 
\begin{equation}
    \left\| AB\right\| _{F}\leq \left\| A\right\| _{2}\left\| B\right\| _{F}.\label{normineq2F}
\end{equation}
Here and subsequently, $\left\|\cdot\right\|_{2}$ and $\left\|\cdot\right\|_{F}$ denote the $2$-norm and the Frobenius norm of matrices, respectively. 
\end{lemma}

\begin{proof}[Proof.\nopunct]
    Write $B=({\rm b}_{1}\cdots {\rm b}_{r})$, where ${\rm b}_{i}$ is the $i$-th column vector of $B$ ($i\in [r]$). Then
\begin{align*}
    \left\| AB\right\| _{F} = \left\| (A{\rm b}_{1}\cdots A{\rm b}_{r})\right\| _{F} =\sqrt{\sum \limits_{i\in [r]}\left\| A{\rm b}_{i}\right\|^{2}} \leq \sqrt{ \sum \limits_{i\in [r]}\left\| A\right\|_{2}^{2}\left\| {\rm b}_{i}\right\|^{2}} =\left\| A\right\| _{2}\left\| B\right\| _{F}.
\end{align*}
This finishes the proof.
\end{proof}

We define the matrix convergence in the following sense. 

\begin{definition}
\normalfont Let $\{M_{k}\}_{k\in \mathbb{N}_{+}}$ be a sequence of real $m$-by-$n$ matrices. $M_{k}$ is said to converge to a real $m$-by-$n$ matrix $M$ as $k\rightarrow \infty$ if every $(i,j)$-entry of $M_{k}$ converges to the $(i,j)$-entry of $M$ as $k\rightarrow \infty$. This is denoted by $M_{k}\rightarrow M$ as $k\rightarrow \infty$.
\end{definition}

The convergence defined here is stronger than the convergence in the matrix norm, as highlighted by the following lemma.

\begin{lemma}\label{limitofnorm}
A sequence $\{M_{k}\}_{k\in \mathbb{N}_{+}}$ of real $m$-by-$n$ matrices. If $M_{k}\rightarrow M$ as $k\rightarrow \infty$, then $\left\| M_{k}\right\|_{F}\rightarrow \left\| M\right\|_{F}$ and $\left\| M_{k}\right\|_{2}\rightarrow \left\| M\right\|_{2}$ as $k\rightarrow \infty$.
\end{lemma}

\begin{proof}[Proof.\nopunct]
The definition of the Frobenius norm makes it obvious that $\left\| M_{k}\right\|_{F}\rightarrow \left\| M\right\|_{F}$ and $\left\| M_{k}- M\right\|_{F}\rightarrow 0$ as $k\rightarrow \infty$. On the other hand, by the triangle inequality, we have
\begin{equation*}
    \big \lvert\left\| M_{k}\right\|_{2}-\left\| M\right\|_{2}\big \rvert\leq \left\| M_{k}- M\right\|_{2}\leq \left\| M_{k}- M\right\|_{F}.
\end{equation*}
It follows that $\big \lvert\left\| M_{k}\right\|_{2}-\left\| M\right\|_{2}\big \rvert\rightarrow 0$ as $k\rightarrow \infty$.
\end{proof}

We can now state the main theorem in the \v{C}ech filtration case.

\begin{theorem}\label{thm0b}
Let $P' \subseteq \mathbb{R}^{d}$ be the observed point cloud, and $Q\subseteq \mathbb{R}^{d}$ the vertex set of the scaling regular simplex $\bfDel^{n-1}$. Then
\begin{align*}
  d_{B}\left(\bar{{\rm D}}\left( P'\right), \bar{{\rm D}}\left(Q\right) \right)=o_{\mathbb{P}}(\sqrt{d})
\end{align*}
holds as $d\rightarrow \infty$.
\end{theorem}

\begin{proof}[Proof.\nopunct]
To address that point clouds change with respect to the dimension $d$, we add a subscript $d$ to their notations in this proof. Let us denote by $P'_{d}$ and $Q_{d}$ the observed point cloud and the vertex set of the scaling regular simplex $\bfDel^{n-1}$ with scalar $\sqrt{\nu d}$, respectively. Here and subsequently, ${\rm O}(d)$ denotes the orthogonal group in the dimension $d$.

Consider the left multiplication group action of ${\rm O}\left( d\right)$ on $\mathbb{R} ^{d}$. For each $M\in {\rm O}(d)$, the image of the action of $M$ on $Q_d$, denoted by $M\cdot Q_{d}$, is given by
\begin{align*}
    M\cdot Q_{d}=\{M y_{1}, M y_{2}, \ldots, M y_{n}\}\subseteq\mathbb{R} ^{d}. 
\end{align*}
Since the orthogonal action is a rigid motion, the shape of the point cloud is preserved. Thus, $\bar{{\rm D}}\left(Q_{d} \right)=\bar{{\rm D}}\left(M\cdot Q_{d}\right)$. Theorem~\ref{lem0b} now leads to
\begin{align*}
d_{B}\left(\bar{{\rm D}}\left( P'_{d}\right), \bar{{\rm D}}\left(Q_{d}\right) \right)\leq \inf _{M\in {\rm O}\left( d\right) }d_{H}(P'_{d},M\cdot Q_{d}).
\end{align*}
The Hausdorff distance between two point clouds with the same sample size is bounded from above by the Frobenius distance between the corresponding data matrices, i.e., for $\mathcal{X}=\left\{\mathrm{x}_{1}, \mathrm{x}_{2}, \ldots, \mathrm{x}_{n}\right\}$ and $\mathcal{Y}=\left\{ \mathrm{y}_{1}, \mathrm{y}_{2}, \ldots, \mathrm{y}_{n}\right\}$, we have
\begin{align}
d_{H}(\mathcal{X},\mathcal{Y})&= \max \left\{ \max_{i\in [n]}\min_{j\in [n]}\left\| \mathrm{x}_{i}-\mathrm{y}_{j}\right\| ,\max_{j\in [n]}\min_{i\in [n]}\left\| \mathrm{x}_{i}-\mathrm{y}_{j}\right\|\right\}\nonumber\\
&\leq \max_{i\in [n]}\left\| \mathrm{x}_{i}-\mathrm{y}_{i}\right\|\label{hausupperbound00} \\ 
&\leq \left\| \mathcal{X}-\mathcal{Y}\right\|_{F}.\label{hausupperbound}
\end{align}
Let $R_{d}=\Big\{ \frac{x'_{1}}{\| x'_{1}\| },\frac{x'_{2}}{\| x'_{2}\| },\ldots ,\frac{x'_{n}}{\| x'_{n}\| }\Big\} \subseteq \mathbb{R} ^{d}$. Since $\|x'_{i}\|\neq 0$ almost surely for each $i\in [n]$, $R_{d}$ is a set of $n$ points on the unit sphere $\mathbb{S} ^{d-1}$ almost surely. By the triangle inequality,
\begin{align}
  &d_{B}\left(\bar{{\rm D}}\left( P'_{d}\right), \bar{{\rm D}}\left(Q_{d}\right) \right) \nonumber\\
  \leq &\inf _{M\in {\rm O}\left( d\right) }d_{H}(P'_{d},M\cdot Q_{d})\nonumber\\
\leq &\inf _{M\in {\rm O}\left( d\right) }[d_{H}(P'_{d},\sqrt{\nu d}\cdot R_{d})+d_{H}(\sqrt{\nu  d}\cdot R_{d},M\cdot Q_{d})]\nonumber\\
 = & d_{H}(P'_{d},\sqrt{\nu d}\cdot R_{d})+\inf _{M\in {\rm O}\left( d\right)}d_{H}(\sqrt{\nu d}\cdot R_{d},M\cdot Q_{d}).\label{1.12}
\end{align}

From (\ref{hausupperbound00}) we have
\begin{equation*}
    d_{H}(P'_{d},\sqrt{\nu d}\cdot R_{d}) \leq \max_{i\in [n]}\left\| x'_{i}-\sqrt{\nu d}\cdot \dfrac{x'_{i}}{\| x'_{i}\| }\right\| = \max_{i\in [n]}\big\lvert \| x'_{i}\| -\sqrt{\nu d}\big\rvert.
\end{equation*}
Since $\max_{i\in [n]}\big\lvert \| x'_{i}\| -\sqrt{\nu d}\big\rvert=O_{\mathbb{P}}(1)$ as $d\rightarrow \infty$ by~\eqref{CFGxd}, it follows that
\begin{equation}
    d_{H}(P'_{d},\sqrt{\nu d}\cdot R_{d})=O_{\mathbb{P}}(1)\label{1.13}
\end{equation}
as $d\rightarrow \infty$. 

For every $d$-by-$d$ matrix $M$, we use the notation $\widetilde{M}$ to denote the corresponding $d$-by-$n$ matrix, in which column vectors are the first $n$ column vectors of $M$. Set $\widetilde{{\rm O}(d)} = \setc*{\widetilde{M}\in \mathbb{R} ^{d\times n}}{M\in {\rm O}(d)}$. Then
\begin{align}
    \inf _{M\in {\rm O}\left( d\right)}d_{H}(\sqrt{\nu d}\cdot R_{d},M\cdot Q_{d})& =\inf _{M\in {\rm O}\left( d\right)}d_{H}(\sqrt{\nu d}\cdot R_{d},\sqrt{\nu d}\cdot \widetilde{M})\nonumber\\
    & =\sqrt{\nu d}\cdot \inf _{M\in {\rm O}\left( d\right)}d_{H}(R_{d}, \widetilde{M})\nonumber\\
    & =\sqrt{\nu d}\cdot \inf _{Z\in \widetilde{{\rm O}\left( d\right)}}d_{H}(R_{d}, Z).\label{1.14}
\end{align}
By noticing~\eqref{innerproductconver0}, for every $1\leq i <j\leq n$, $\langle  \frac{x'_{i}}{\| x'_{i}\| },\frac{x'_{j}}{\| x'_{j}\| }\rangle $ converges to zero in probability as $d\rightarrow \infty$. Using Theorem~2.3.2 in \cite{DU04}, for every subsequence $\left\{ d_{k}\right\} _{k\in \mathbb{N}_{+}}$ of dimensions, there exists a further subsequence $\left\{ d_{k_{l}}\right\} _{l\in \mathbb{N}_{+}}$ of $\left\{ d_{k}\right\} _{k\in \mathbb{N}_{+}}$ such that $\langle  \frac{x'_{i}}{\| x'_{i}\| },\frac{x'_{j}}{\| x'_{j}\| }\rangle $ converges to zero almost surely as $l\rightarrow \infty$, i.e.,
\begin{equation*}
    \mathbb{P}\left(\langle  \dfrac{x'_{i}}{\| x'_{i}\| },\dfrac{x'_{j}}{\| x'_{j}\| }\rangle \rightarrow 0\ as\ l \rightarrow \infty\right)=1,\ 1\leq \forall i <\forall j\leq n.
\end{equation*}
It follows that
\begin{equation*}
    \mathbb{P}\left(\langle  \dfrac{x'_{i}}{\| x'_{i}\| },\dfrac{x'_{j}}{\| x'_{j}\| }\rangle \rightarrow 0\ as\ l \rightarrow \infty,\ 1\leq \forall i <\forall j\leq n\right)=1.
\end{equation*}
Note that the event $\big\{\langle  \frac{x'_{i}}{\| x'_{i}\| },\frac{x'_{j}}{\| x'_{j}\| }\rangle \rightarrow 0\ as\ l \rightarrow \infty,\ 1\leq \forall i <\forall j\leq n \big\}$ is equivalent to the event $\{ R_{d_{k_{l}}}^{T}R_{d_{k_{l}}}\rightarrow I_{n}\ as\ l\rightarrow \infty \}$, and the latter implies the event $\{ \| R_{d_{k_{l}}}\| _{2}\rightarrow 1\ as\ l\rightarrow \infty\}$ by Lemma~\ref{limitofnorm}. Thus, $\| R_{d_{k_{l}}}\| _{2}\rightarrow 1\ as\ l\rightarrow \infty$ almost surely. Note that $\frac{x'_{1}}{\| x'_{1}\| },\frac{x'_{2}}{\| x'_{2}\| },\ldots ,\frac{x'_{n}}{\| x'_{n}\| }$ are linearly independent almost surely by Proposition~\ref{LDwithprob0}, we apply the Schmidt orthogonalization for $\frac{x'_{1}}{\| x'_{1}\| },\frac{x'_{2}}{\| x'_{2}\| },\ldots ,\frac{x'_{n}}{\| x'_{n}\| }$. Then there exist a matrix $Z'\in \widetilde{{\rm O}(d_{k_{l}})}$ and an $n$-by-$n$ upper triangular matrix $V_{d_{k_{l}}}$ such that $Z'=R_{d_{k_{l}}}V_{d_{k_{l}}}$, and the event $\{\langle  \frac{x'_{i}}{\| x'_{i}\| },\frac{x'_{j}}{\| x'_{j}\| }\rangle \rightarrow 0\ as\ l \rightarrow \infty,\  1\leq \forall i <\forall j\leq n\}$ implies the event $\{V_{d_{k_{l}}}\rightarrow I_{n}\ as\ l \rightarrow \infty\}$ from the process of orthogonalization. Hence $V_{d_{k_{l}}}\rightarrow I_{n}\ as\ l \rightarrow \infty$ almost surely. From~\eqref{normineq2F}, \eqref{hausupperbound}, we obtain
\begin{align*}
    \inf _{Z\in \widetilde{{\rm O}\left( d_{k_{l}}\right)}}d_{H}(R_{d_{k_{l}}}, Z)& \leq \inf _{Z\in \widetilde{{\rm O}\left( d_{k_{l}}\right)}} \left\|R_{d_{k_{l}}}-Z\right\|_{F}\\
    & \leq \left\|R_{d_{k_{l}}}-Z'\right\|_{F}\\
    & = \left\|R_{d_{k_{l}}}-R_{d_{k_{l}}}V_{d_{k_{l}}}\right\|_{F}\\
    & \leq \left\|R_{d_{k_{l}}}\right\|_{2}\cdot \left\|I_{n}-V_{d_{k_{l}}}\right\|_{F}.
\end{align*}
Therefore, $\inf_{Z\in \widetilde{{\rm O}\left( d_{k_{l}}\right)}}d_{H}(R_{d_{k_{l}}}, Z)\rightarrow 0 \ as\ l\rightarrow \infty$ almost surely. Again, using Theorem~2.3.2 in \cite{DU04},
\begin{equation}
  \inf_{Z\in \widetilde{{\rm O}\left( d\right)}}d_{H}(R_{d}, Z)= o_{\mathbb{P}}(1)\label{1.15}
\end{equation}
holds as $d\rightarrow \infty$.
By (\ref{1.12}), (\ref{1.13}), (\ref{1.14}), (\ref{1.15}) and Proposition \ref{prop0b}, 
\begin{align*}
  d_{B}\left(\bar{{\rm D}}\left( P'_{d}\right), \bar{{\rm D}}\left(Q_{d}\right) \right) = o_{\mathbb{P}}(\sqrt{d})
\end{align*}
holds as $d\rightarrow \infty$.
\end{proof}

The following corollary is a fairly straightforward consequence of Theorem \ref{thm0b} by noticing the calculus of the asymptotic notations (Proposition \ref{prop0a}). We state it without proof.

\begin{corollary}\label{cor0b}
Let $P' \subseteq \mathbb{R}^{d}$ be the observed point cloud and $N\in \mathbb{N}$. Then
\begin{align*}
    \frac{1}{\sqrt{d}}\sup \setc*{{\rm pers}(\omega)}{\omega\in {\rm D}_{N}(P')} = \sqrt{\nu}\left(\sqrt{\frac{N+1}{N+2}}-\sqrt{\frac{N}{N+1}}\right)+o_{\mathbb{P}}(1)
\end{align*}
and
\begin{align*}
    \sup \setc*{\frac{{\rm pers}(\omega)}{d_{\omega}}}{\omega \in {\rm D}_{N}(P')} = O_{\mathbb{P}}(1)
\end{align*}
hold as $d\rightarrow \infty$.
\end{corollary}

\subsection{Inconsistency between the original and the observed persistence diagrams}\label{subsec3}

After analyzing the asymptotic behavior of the observed persistence diagram, we can now estimate the difference between the original and observed persistence diagrams.

\subsubsection{Rips filtration case}\label{subsubsec4}

\begin{theorem}\label{thm1b}
Let $P$ and $P'$ be the original and the observed point clouds in $\mathbb{R}^{d}$, respectively. Then
\begin{align*}
  d_{B}\left(\bar{{\rm D}}_{N}\left(P\right), \bar{{\rm D}}_{N}\left( P'\right) \right) = 
    \begin{cases}
        \frac{\sqrt{2\nu d}}{4}+O_{\mathbb{P}}(1),& \mbox{if $N=0$},\\
         O_{\mathbb{P}}(1), & \mbox{if $N>0$},
    \end{cases}
\end{align*}
holds as $d\rightarrow \infty$.
\end{theorem}

\begin{proof}[Proof.\nopunct]
By the triangle inequality,
\begin{align*}
    &\ \ \ \ \left\lvert d_{B}\left(\bar{{\rm D}}\left(P\right), \bar{{\rm D}}\left( P'\right) \right) - \frac{1}{2}\sup \setc*{{\rm pers}(\omega)}{\omega\in {\rm D}(P')}\right\rvert\\
    &=\left\lvert d_{B}\left(\bar{{\rm D}}\left(P\right), \bar{{\rm D}}\left( P'\right) \right)-d_{B}\left(\bar{{\rm D}}\left(P'\right), \bar{{\rm D}}\left(\{\cdot\}\right) \right)\right\rvert \leq d_{B}\left(\bar{{\rm D}}\left(P\right), \bar{{\rm D}}\left(\{\cdot\}\right) \right).
\end{align*}
We note that $d_{B}\left(\bar{{\rm D}}\left(P\right), \bar{{\rm D}}\left(\{\cdot\}\right) \right)$ is a deterministic constant. It implies that
\begin{align*}
    d_{B}\left(\bar{{\rm D}}\left(P\right), \bar{{\rm D}}\left( P'\right) \right) = \frac{1}{2}\sup \setc*{{\rm pers}(\omega)}{\omega\in {\rm D}(P')}+O_{\mathbb{P}}(1).
\end{align*}
as $d\rightarrow \infty$. The theorem follows by Corollary~\ref{suppersbounded}.
\end{proof}

The convergence to zero in probability implies eventual boundedness in probability, but the converse may not hold true generally. The following theorem demonstrates that, despite being eventually bounded in probability, the bottleneck distance between the persistence diagrams of $P$ and $P'$ does not converge to zero in probability.

\begin{theorem}\label{db will not converge to 0 in probability}
Let $P$ and $P'$ be the original and the observed point clouds in $\mathbb{R}^{d}$, respectively. Suppose that the $N$-th  persistence diagram ${\rm D}_{N}\left(P\right)\neq \emptyset$. Then $d_{B}\left(\bar{{\rm D}}_{N}\left(P\right), \bar{{\rm D}}_{N}\left( P'\right) \right)$ does not converge to zero in probability as $d\rightarrow \infty$.
\end{theorem}

\begin{proof}[Proof.\nopunct]
    The proof for the case $N=0$ is trivial from Theorem~\ref{thm1b}. Thus, we prove for $N\in \mathbb{N}_{+}$. Suppose, contrary to our claim, that $d_{B}\left(\bar{{\rm D}}_{N}\left(P\right), \bar{{\rm D}}_{N}\left( P'\right) \right)$ converges to zero in probability as $d\rightarrow \infty$. By Proposition \ref{relofdB&dH}, $d_{H}\left(\bar{{\rm D}}_{N}\left(P\right), \bar{{\rm D}}_{N}\left( P'\right) \right)=o_{\mathbb{P}}(1)$ as $d\rightarrow \infty$. 
Since ${\rm D}_{N}\left(P\right)\neq \emptyset$, we can take $\omega_{0} = (b_{\omega_{0}},d_{\omega_{0}})\in {\rm D}_{N}\left(P\right)$. Then
\begin{align*}
    d_{H}\left(\bar{{\rm D}}_{N}\left(P\right), \bar{{\rm D}}_{N}\left( P'\right) \right)& \geq \sup _{\omega\in \bar{{\rm D}}_{N}\left( P\right) }\inf _{\omega'\in \bar{{\rm D}}_{N}\left( P'\right)}\left\| \omega-\omega'\right\| _{\infty }\geq \inf _{\omega'\in \bar{{\rm D}}_{N}\left( P'\right)}\left\| \omega_{0}-\omega'\right\| _{\infty }.
\end{align*}
From~\eqref{obpairwisedist} and Proposition \ref{prop0d}, we have
\begin{equation}
    \min_{\substack{i,j\in [n] \\ i<j}}\frac{\left\| x_{i}'-x_{j}'\right\|}{2}=\frac{\sqrt{2\nu d}}{2}+O_{\mathbb{P}}(1)\label{RTECn234}
\end{equation}
as $d\rightarrow \infty$. By definition, for every $\varepsilon>0$, there exist a constant $C$ and a positive integer $D_{1}$ such that
\begin{equation*}
    \mathbb{P}\left( \Bigg\lvert \min_{\substack{i,j\in [n] \\ i<j}}\frac{\left\| x_{i}'-x_{j}'\right\|}{2}-\frac{\sqrt{2\nu d}}{2}\Bigg\rvert \leq C \right) \geq 1-\varepsilon 
\end{equation*}
holds for every $d>D_{1}$. This clearly forces 
\begin{equation*}
    \mathbb{P}\left(  \min_{\substack{i,j\in [n] \\ i<j}}\frac{\left\| x_{i}'-x_{j}'\right\|}{2} \geq \frac{\sqrt{2\nu d}}{2}-C \right) \geq 1-\varepsilon
\end{equation*}
for every $d>D_{1}$. Since $d$ tends to infinity and ${\rm pers}(\omega_{0})/2+d_{\omega_{0}}$ is a constant, there exists $D_{2}>0$ such that $\sqrt{2\nu d}/2-C\geq {\rm pers}(\omega_{0})/2+d_{\omega_{0}}$ whenever $d>D_{2}$. Letting $D'=\max\{D_{1},D_{2}\}$ we have
\begin{equation}
    \mathbb{P}\left(\min_{\substack{i,j\in [n] \\ i<j}}\frac{\left\| x_{i}'-x_{j}'\right\|}{2} \geq \frac{{\rm pers}(\omega_{0})}{2}+d_{\omega_{0}} \right) \geq 1-\varepsilon\label{estimate00}
\end{equation}
for every $d>D'$.

Regardless of whether ${\rm D}_{N}\left( P'\right)$ is empty or not, we have
\begin{align}
    \inf _{\omega'\in {\rm D}_{N}\left( P'\right)}\left\| \omega_{0}-\omega'\right\| _{\infty}&=\inf _{\omega'\in {\rm D}_{N}\left( P'\right)}\max\left\{\lvert b_{\omega'}-b_{\omega_{0}}\rvert,\lvert d_{\omega'}-d_{\omega_{0}}\rvert\right\}\nonumber\\
    &\geq \inf _{\omega'\in {\rm D}_{N}\left( P'\right)}\lvert d_{\omega'}-d_{\omega_{0}}\rvert\nonumber\\
    &= \min_{\omega'\in {\rm D}_{N}\left( P'\right)}\lvert d_{\omega'}-d_{\omega_{0}}\rvert\nonumber\\
    &\geq \min_{\substack{i,j\in [n] \\ i<j}}\Bigg\lvert \frac{\left\| x_{i}'-x_{j}'\right\|}{2}-d_{\omega_{0}}\Bigg\rvert.\label{estimate01}
\end{align}
By~\eqref{estimate00}, \eqref{estimate01},
\begin{align*}
&\ \ \ \ \mathbb{P}\left(\inf_{\omega'\in {\rm D}_{N}\left( P'\right)}\left\| \omega_{0}-\omega'\right\| _{\infty}\geq~\frac{{\rm pers}(\omega_{0})}{2}\right)\\
&\geq \mathbb{P}\left(\min_{\substack{i,j\in [n] \\ i<j}}\frac{\left\| x_{i}'-x_{j}'\right\|}{2} \geq \frac{{\rm pers}(\omega_{0})}{2}+d_{\omega_{0}}\right)\geq 1-\varepsilon
\end{align*}
holds for every $d>D'$. Since
\begin{align*}
    \inf _{\omega'\in \bar{{\rm D}}_{N}\left( P'\right)}\left\| \omega_{0}-\omega'\right\| _{\infty }=\min\left\{\frac{{\rm pers}(\omega_{0})}{2},\inf _{\omega'\in {\rm D}_{N}\left( P'\right)}\left\| \omega_{0}-\omega'\right\| _{\infty }\right\},
\end{align*}
it follows that
\begin{equation*}
    \mathbb{P}\left(\inf _{\omega'\in \bar{{\rm D}}_{N}\left( P'\right)}\left\| \omega_{0}-\omega'\right\| _{\infty }=\frac{{\rm pers}(\omega_{0})}{2} \right)\geq 1-\varepsilon
\end{equation*}
for every $d>D'$. Consequently,  
\begin{equation}
    \lim _{d\rightarrow \infty}\mathbb{P}\left(\inf _{\omega'\in \bar{{\rm D}}_{N}\left( P'\right)}\left\| \omega_{0}-\omega'\right\| _{\infty }=\frac{{\rm pers}(\omega_{0})}{2} \right) =1.\label{BEJDdi}
\end{equation}
However, since $\inf_{\omega'\in \bar{{\rm D}}_{N}\left( P'\right)}\left\| \omega_{0}-\omega'\right\| _{\infty }=o_{\mathbb{P}}(1)$ by the assumption, \eqref{BEJDdi} contradicts the definition of the small $o$ notation. Consequently, the assumption does not hold, and the proof is completed.
\end{proof}

The following conclusion can be drawn regarding their Hausdorff distance.

\begin{theorem}\label{dHunboundedinprobability}
Let $P$ and $P'$ be the original and the observed point clouds in $\mathbb{R}^{d}$, respectively. Suppose that the $N$-th  persistence diagram ${\rm D}_{N}\left(P\right)\neq \emptyset$. Then $d_{H}\left({\rm D}_{N}\left(P\right), {\rm D}_{N}\left( P'\right) \right)$ is eventually unbounded in probability as $d\rightarrow \infty$.
\end{theorem}

\begin{proof}[Proof.\nopunct]
Suppose, contrary to our claim, that $d_{H}\left({\rm D}_{N}\left(P\right), {\rm D}_{N}\left( P'\right) \right)=O_{\mathbb{P}}(1)$ as $d\rightarrow \infty$. By the definition of big $O$ notation, for $\varepsilon=1/4$, there exist constants $C_{1}, D_{1}>0$ such that
\begin{equation}
    \mathbb{P}\left(d_{H}\left({\rm D}_{N}\left(P\right), {\rm D}_{N}\left( P'\right) \right)<C_{1}\right)\geq \frac{3}{4}\label{nxakds}
\end{equation}
holds for every $d>D_{1}$. Since ${\rm D}_{N}\left(P\right)\neq \emptyset$, we can take $\omega_{0}\in {\rm D}_{N}\left(P\right)$. Then
\begin{align*}
    d_{H}\left({\rm D}_{N}\left(P\right), {\rm D}_{N}\left( P'\right) \right)\geq \sup _{\omega\in {\rm D}_{N}\left( P\right) }\inf _{\omega'\in {\rm D}_{N}\left( P'\right)}\left\| \omega-\omega'\right\| _{\infty }\geq \inf _{\omega'\in {\rm D}_{N}\left( P'\right)}\left\| \omega_{0}-\omega'\right\| _{\infty }.
\end{align*}
Combining \eqref{estimate01} with \eqref{nxakds}, we have
\begin{align}
    \mathbb{P}\left(\min_{\substack{i,j\in [n] \\ i<j}}\Bigg\lvert \frac{\left\| x_{i}'-x_{j}'\right\|}{2}-d_{\omega_{0}}\Bigg\rvert <C_{1}\right)\geq \frac{3}{4}.\label{lowerboundofeventB1}
\end{align}

On the other hand. By~\eqref{RTECn234} and Definition~\ref{def0b}, for $\varepsilon=1/4$, there exist constants $C_{2}, D_{2}>0$ such that

\begin{equation*}
    \mathbb{P}\left(  \min_{\substack{i,j\in [n] \\ i<j}}\frac{\left\| x_{i}'-x_{j}'\right\|}{2} \geq \frac{\sqrt{2\nu d}}{2}-C_{2} \right) \geq 1-\frac{1}{4}=\frac{3}{4}
\end{equation*}
holds for every $d>D_{2}$. Since $d$ tends to infinity and $2C_{1}+d_{\omega_{0}}$ is a constant, there exists $D_{3}>0$ such that $\sqrt{2\nu d}/2-C_{2}> 2C_{1}+d_{\omega_{0}}$ whenever $d>D_{3}$. Hence for every $d>\max\{D_{2},D_{3}\}$,
\begin{equation}
    \mathbb{P}\left(\min_{\substack{i,j\in [n] \\ i<j}}\frac{\left\| x_{i}'-x_{j}'\right\|}{2} \geq 2C_{1}+d_{\omega_{0}} \right) \geq \frac{3}{4}.\label{erdrg}
\end{equation}
Let us denote by $B_{1}$ the event in~\eqref{lowerboundofeventB1}, and $B_{2}$ the event in~\eqref{erdrg}. Then it is evident that $B_{1}\cap B_{2}=\emptyset$. However,
\begin{equation*}
    \mathbb{P}\left(B_{1}\cap B_{2} \right) \geq \mathbb{P}\left(B_{1} \right)+\mathbb{P}\left(B_{2} \right)-1\geq \frac{3}{4}+\frac{3}{4}-1=\frac{1}{2}
\end{equation*}
holds for every $d>\max\{D_{1},D_{2},D_{3}\}$, contradicting $B_{1}\cap B_{2}=\emptyset$. Therefore, $d_{H}\left({\rm D}_{N}\left(P\right), {\rm D}_{N}\left( P'\right) \right)$ is eventually unbounded in probability as $d\rightarrow \infty$.
\end{proof}

To validate our conclusions, we conduct a numerical experiment using HomCloud (Version~3.2.1), developed by \cite{obayashi2022HomCloud}. The original point cloud $P$ is generated by uniformly sampling $500$ points from a closed 2-cube $[-1,1]^2$. We generate random noises following the Gaussian distribution $\mathcal{N}_{d}(0, \nu I_{d})$ in which $\nu=1\times 10^{-5}$ for various dimensions ($d=1000,5000$, and $10{,}000$). The $1$st persistence diagrams of $P$ and $P'$ are shown in Fig.~\ref{fig:PDnoises}.

\subsubsection{\v{C}ech filtration case}\label{subsubsec5}
\begin{theorem}\label{thm1d}
Let $P$ and $P'$ be the original and the observed point clouds in $\mathbb{R}^{d}$, respectively. Then
\begin{align*}
  d_{B}\left(\bar{{\rm D}}_{N}\left(P\right), \bar{{\rm D}}_{N}\left( P'\right) \right) = O_{\mathbb{P}}(\sqrt{d})
\end{align*}
holds for every $N\in \mathbb{N}$ less than or equal to $(n-2)$ as $d\rightarrow \infty$.
\end{theorem}

\begin{proof}[Proof.\nopunct]
By the triangle inequality,
\begin{equation*}
    d_{B}\left(\bar{{\rm D}}_{N}\left(P\right), \bar{{\rm D}}_{N}\left( P'\right) \right)\leq d_{B}\left(\bar{{\rm D}}_{N}\left(P\right), \bar{{\rm D}}_{N}\left( Q\right) \right)+d_{B}\left(\bar{{\rm D}}_{N}\left(Q\right), \bar{{\rm D}}_{N}\left( P'\right) \right).
\end{equation*}
By Lemma~\ref{regularcech}, there is only one pair $\big(\sqrt{\nu dN}/\sqrt{N+1},\sqrt{\nu d(N+1)}/\sqrt{N+2}\big)$ with a certain multiplicity in ${\rm D}_{N}(Q)$, indicating that $d_{B}\left(\bar{{\rm D}}_{N}\left(P\right), \bar{{\rm D}}_{N}\left( Q\right) \right)=O_{\mathbb{P}}(\sqrt{d})$ as $d\rightarrow \infty$. Since $d_{B}(\bar{{\rm D}}_{N}\left(P'\right), \bar{{\rm D}}_{N}\left( Q\right))$\ $=o_{\mathbb{P}}(\sqrt{d})$ by Theorem~\ref{thm0b}, it follows that the right-hand side of the above triangle inequality is $O_{\mathbb{P}}(\sqrt{d})$. Consequently, $d_{B}\left(\bar{{\rm D}}_{N}\left(P\right), \bar{{\rm D}}_{N}\left( P'\right) \right) = O_{\mathbb{P}}(\sqrt{d})$ as $d\rightarrow \infty$.
\end{proof}

Observed from Theorems~\ref{thm0b}, \ref{thm1d} and Corollary \ref{cor0b}, we shall conclude that in the \v{C}ech filtration case, the bottleneck distance is eventually unbounded in probability.

\begin{theorem}\label{cech strong inconsistency}
Let $P$ and $P'$ be the original and the observed point clouds in $\mathbb{R}^{d}$, respectively. Then for every $N\in \mathbb{N}$, $d_{B}\left(\bar{{\rm D}}_{N}\left(P\right), \bar{{\rm D}}_{N}\left( P'\right) \right)$ is eventually unbounded in probability as $d\rightarrow \infty$.
\end{theorem}

\begin{proof}[Proof.\nopunct]
Suppose, contrary to our claim, that $d_{B}\left(\bar{{\rm D}}_{N}\left(P\right), \bar{{\rm D}}_{N}\left( P'\right) \right)$ is eventually bounded in probability as $d\rightarrow \infty$. By the triangle inequality, 
\begin{equation*}
    d_{B}\left(\bar{{\rm D}}_{N}\left(P\right), \bar{{\rm D}}_{N}\left( Q\right) \right)\leq d_{B}\left(\bar{{\rm D}}_{N}\left(P\right), \bar{{\rm D}}_{N}\left( P'\right) \right)+d_{B}\left(\bar{{\rm D}}_{N}\left(P'\right), \bar{{\rm D}}_{N}\left( Q\right) \right).
\end{equation*}
By Theorem~\ref{thm0b} we have $d_{B}(\bar{{\rm D}}_{N}\left(P'\right), \bar{{\rm D}}_{N}\left( Q\right) )=o_{\mathbb{P}}(\sqrt{d})$ as $d\rightarrow \infty$. Then $d_{B}\left(\bar{{\rm D}}_{N}\left(P\right), \bar{{\rm D}}_{N}\left( Q\right) \right)=o_{\mathbb{P}}(\sqrt{d})$ as $d\rightarrow \infty$.

On the other hand, by Lemma~\ref{regularcech}, there is only one pair $q_{0}=(c_{1} \sqrt{d},c_{2} \sqrt{d})$ with multiplicity in ${\rm D}_{N}\left( Q\right)$, where $c_{1}=\sqrt{\nu N}/\sqrt{N+1}$, $c_{2}=\sqrt{\nu (N+1)}/\sqrt{N+2}$. Then
\begin{align*}
    d_{B}\left(\bar{{\rm D}}_{N}\left(P\right), \bar{{\rm D}}_{N}\left( Q\right) \right)&\geq d_{H}\left(\bar{{\rm D}}_{N}\left(P\right), \bar{{\rm D}}_{N}\left( Q\right) \right)\\
    & \geq \sup _{q\in \bar{{\rm D}}_{N}\left(Q\right)}\inf _{\omega\in \bar{{\rm D}}_{N}\left( P\right)}\left\| \omega-q\right\| _{\infty }\\
    & \geq \inf _{\omega\in \bar{{\rm D}}_{N}\left( P\right)}\left\| \omega-q_{0}\right\| _{\infty }.
\end{align*}
If ${\rm D}_{N}\left(P\right)=\emptyset$, then $\inf_{\omega\in \bar{{\rm D}}_{N}\left( P\right)}\left\| \omega-q_{0}\right\| _{\infty }={\rm pers}(q_{0})/2=(c_{2}-c_{1})\sqrt{d}/2$. This contradicts $d_{B}\left(\bar{{\rm D}}_{N}\left(P\right), \bar{{\rm D}}_{N}\left( Q\right) \right)=o_{\mathbb{P}}(\sqrt{d})$ as $d\rightarrow \infty$. If ${\rm D}_{N}\left(P\right)\neq \emptyset$, then
\begin{align*}
    d_{B}\left(\bar{{\rm D}}_{N}\left(P\right), \bar{{\rm D}}_{N}\left( Q\right) \right)&\geq \inf _{\omega\in \bar{{\rm D}}_{N}\left( P\right)}\left\| \omega-q_{0}\right\| _{\infty }\\
    & =\min\left\{\frac{{\rm pers}(q_{0})}{2},\min _{\omega\in {\rm D}_{N}\left( P\right)}\left\| \omega-q_{0}\right\| _{\infty }\right\}.
\end{align*}
Since 
\begin{align*}
\min \limits_{\omega\in {\rm D}_{N}\left( P\right)}\left\| \omega-q_{0}\right\| _{\infty}& =\min \limits_{\omega\in {\rm D}_{N}\left( P\right)}\max \{\vert b_{\omega}-b_{q_{0}}\vert,\vert d_{\omega}-d_{q_{0}}\vert\}\\
& =\min \limits_{\omega\in {\rm D}_{N}\left( P\right)}\max \{\big\vert b_{\omega}-c_{1}\cdot \sqrt{d}\big\vert,\big\vert d_{\omega}-c_{2}\cdot \sqrt{d}\big\vert\},
\end{align*}
it follows that
\begin{align*}
\lim _{d\rightarrow \infty}\frac{\min \limits_{\omega\in {\rm D}_{N}\left( P\right)}\left\| \omega-q_{0}\right\| _{\infty}}{\sqrt{d}}& =\lim _{d\rightarrow \infty}\min \limits_{\omega\in {\rm D}_{N}\left( P\right)}\max \bigg\{\bigg\vert\frac{b_{\omega}}{\sqrt{d}}-c_{1}\bigg\vert,\bigg\vert\frac{d_{\omega}}{\sqrt{d}}-c_{2}\bigg\vert \bigg\}=c_{2}.
\end{align*}
Hence
\begin{align*}
\lim _{d\rightarrow \infty}\frac{\inf \limits_{\omega\in \bar{{\rm D}}_{N}\left( P\right)}\left\| \omega-q_{0}\right\| _{\infty }}{\sqrt{d}}& =\lim _{d\rightarrow \infty}\min \Bigg\{\frac{{\rm pers}(q_{0})}{2\sqrt{d}},\frac{\min \limits_{\omega\in {\rm D}_{N}\left( P\right)}\left\| \omega-q_{0}\right\| _{\infty}}{\sqrt{d}}\Bigg\}\\
& =\min\Big\{\frac{c_{2}-c_{1}}{2},c_{2}\Big\}\\
& =\frac{c_{2}-c_{1}}{2}>0.
\end{align*}
This also contradicts $d_{B}\left(\bar{{\rm D}}_{N}\left(P\right), \bar{{\rm D}}_{N}\left( Q\right) \right)=o_{\mathbb{P}}(\sqrt{d})$ as $d\rightarrow \infty$. Consequently, the assumption does not hold, and the proof is completed.
\end{proof}

\subsection{Classification of the similarity between the original and the observed persistence diagrams}\label{subsec4}

From the results obtained so far, it seems that the observed persistence diagram is significantly different from the original one in both filtrations. This suggests that the use of observed persistence diagrams in the HDLSS setting is not reliable. Additionally, the results in the \v{C}ech filtration case are dissimilar to those in the Rips filtration case. In this subsection, we will classify the asymptotic similarity to describe the varying levels of consistency between the original and the observed persistence diagrams.

To measure the difference between the original and observed persistence diagrams, we primarily use the bottleneck distance. However, this distance alone cannot distinguish between two cases—one where the observed persistence diagram is getting closer to the original persistence diagram and another where it is moving away from the original persistence diagram but still close to the diagonal line $\setc*{(x,x)}{x\in \nnegR}$. In both cases, the bottleneck distance would eventually be bounded. To solve this problem, we use the Hausdorff distance as a complementary measurement to differentiate between these two cases.

\begin{definition}\label{classification}
\normalfont Let $P$ and $P'$ be the original and the observed point clouds in $\mathbb{R}^{d}$, respectively. The different levels of the asymptotic similarity of persistence diagrams of $P$ and $P'$ are provided as follows.

\vspace{1ex}

\noindent \textbf{Bottleneck consistency:} $d_{B}\left(\bar{{\rm D}}_{N}\left(P\right), \bar{{\rm D}}_{N}\left( P'\right) \right)$ converges to zero in probability as $d\rightarrow \infty$.

\vspace{1ex}

\noindent \textbf{Bottleneck inconsistency:} $d_{B}\left(\bar{{\rm D}}_{N}\left(P\right), \bar{{\rm D}}_{N}\left( P'\right) \right)$ is eventually bounded but does not converge to zero in probability as $d\rightarrow \infty$. Furthermore, there are three sub-classes:
\begin{enumerate}[I.]
  \item
  $d_{H}\left({\rm D}_{N}\left(P\right), {\rm D}_{N}\left( P'\right) \right)$ converges to zero in probability as $d\rightarrow \infty$.       
  \item
  $d_{H}\left({\rm D}_{N}\left(P\right), {\rm D}_{N}\left( P'\right) \right)$ is eventually bounded but does not converge to zero in probability as $d\rightarrow \infty$. 
  \item
  $d_{H}\left({\rm D}_{N}\left(P\right), {\rm D}_{N}\left( P'\right) \right)$ is eventually unbounded in probability as $d\rightarrow \infty$.
\end{enumerate}

\microtypesetup{protrusion=true,expansion=true}
\noindent \textbf{Strong bottleneck inconsistency:} $d_{B}\left(\bar{{\rm D}}_{N}\left(P\right), \bar{{\rm D}}_{N}\left( P'\right) \right)$ is eventually unbounded in probability as $d\rightarrow \infty$.
\end{definition}

\microtypesetup{protrusion=false,expansion=false}

Based on Theorems \ref{thm1b}, \ref{db will not converge to 0 in probability}, \ref{dHunboundedinprobability} for the Rips filtration, and Theorem \ref{cech strong inconsistency} for the \v{C}ech filtration, we can now formulate the following result.

\begin{theorem}\label{classification theorem}
Let $P$ and $P'$ be the original and the observed point clouds in $\mathbb{R}^{d}$, respectively. Then for the Rips filtration, the asymptotic similarity of the $N$-th persistence diagrams of $P$ and $P'$ is at the bottleneck inconsistency level (sub-class III) for every $N>0$, and at the strong bottleneck inconsistency level for $N=0$. For the \v{C}ech filtration, the asymptotic similarity of the $N$-th persistence diagrams of $P$ and $P'$ is at the strong bottleneck inconsistency level for every $N\in \mathbb{N}$.
\end{theorem}

In conclusion, we highlight the existence of the curse of dimensionality on persistence diagrams in the HDLSS setting. The curse of dimensionality affects persistence diagrams, making them unreliable for use.

\section{An attempt to mitigate the curse of dimensionality}\label{sec4}

The impact of the curse of dimensionality on persistence diagrams has been discussed in the previous section. The aim of this section is to mitigate the curse of dimensionality through the use of dimension-reduction techniques. One idea is using PCA, a widely used technique for various purposes such as dimensionality reduction, data compression, feature extraction, and data visualization. We will show that applying the so-called normalized PCA to the observed point cloud can improve the consistency level. However, our numerical experiment results indicate that the curse of dimensionality on persistence diagrams cannot be completely eliminated by using this technique.

Throughout this section, we will restrict our attention to the Rips filtration. $e_{1},e_{2},\ldots,e_{n}$ in the noise point cloud $E$ are supposed to be \iid\ random vectors following the standard Gaussian distribution $\mathcal{N}_{d}(0, I_{d})$.

\subsection{Normalized principal component analysis}\label{subsec5}

PCA is defined as the orthogonal projection of data onto a lower dimensional linear space, known as the principal subspace, such that the variance of the projected data is maximized \cite{Hotelling1933AnalysisOA}. In this subsection, we introduce the concept of normalized PCA and then apply this technique to the observed point cloud $P'$ by projecting $P'$ onto the principal subspace spanned by the first $s$ principal components, where $s$ is the essential dimension of $P$. A new point cloud called the compressed point cloud is obtained. Our objective is to study the difference between the persistence diagrams of the original and the compressed point clouds to see whether the curse of dimensionality still remains.

\begin{definition}\label{samplecovmatrix}
    \normalfont Let $\bfZ=\{\mathrm{z}_{1},\mathrm{z}_{2},\ldots,\mathrm{z}_{n}\}\subseteq \mathbb{R}^{d}$ be a point cloud. The (unbiased) sample covariance matrix of $\bfZ$, denoted by $S_{\bfZ}$, is defined by
    \begin{equation}
        S_{\bfZ}=\frac{1}{n-1}\sum_{i\in [n]}\left( \mathrm{z}_{i}-\overline{\mathrm{z}}\right) \left( \mathrm{z}_{i}-\overline{\mathrm{z}}\right) ^{T},\label{eqn: sample covariance matrix formula}
    \end{equation}
    where $\overline{\mathrm{z}}=n^{-1}\sum_{i\in [n]}\mathrm{z}_{i}$ is called the sample mean of $\bfZ$.
\end{definition}

\begin{remark}
    \normalfont Once we regard the point cloud as the data matrix, \eqref{eqn: sample covariance matrix formula} can be reformulated as
    \begin{equation*}
        S_{\bfZ}=\frac{1}{n-1}(\bfZ-\overline{\bfZ})(\bfZ-\overline{\bfZ})^{T},
    \end{equation*}
    where $\overline{\bfZ}$ denotes the $d\times n$ matrix $(\underbrace{\overline{\mathrm{z}}\ \overline{\mathrm{z}}\ \cdots\ \overline{\mathrm{z}}}_{n})$.
\end{remark}

In high-dimensional statistics, the following dual sample covariance matrix is frequently used because it can reduce the computation cost in the HDLSS framework.

\begin{definition}\label{dualsamplecovmatrix}
    \normalfont Let us denote by $\bfZ$ the point cloud with $n$ samples in $\mathbb{R}^{d}$, $\overline{\mathrm{z}}$ the sample mean of $\bfZ$, and $\overline{\bfZ}$ the $d\times n$ matrix $(\underbrace{\overline{\mathrm{z}}\ \overline{\mathrm{z}}\ \cdots\ \overline{\mathrm{z}}}_{n})$. The dual sample covariance matrix of $\bfZ$, written by $S_{\mathrm{D},\bfZ}$, is defined by
    \begin{equation*}
        S_{\mathrm{D},\bfZ}=\frac{1}{n-1}(\bfZ-\overline{\bfZ})^{T}(\bfZ-\overline{\bfZ}).
    \end{equation*}
\end{definition}

We note here that the size of $S_{{\rm D}, \bfZ}$ is much smaller than that of $S_{\bfZ}$ when $d$ tends to infinity but $n$ is fixed, and they share the same eigenvalues. From now on, we set $\bfP_{n}=I_{n}-n^{-1}\bfone_{n}^{}\bfone_{n}^{T}$, where $\bfone_{n}=(\underbrace{1,1,\ldots, 1}_{n})^{T}$. The following lemma provides an alternative way of forming the sample covariance matrix and its duality.

\begin{lemma}\label{samplecov123}
    Let $\bfZ=\{\mathrm{z}_{1},\mathrm{z}_{2},\ldots,\mathrm{z}_{n}\}\subseteq \mathbb{R}^{d}$ be a point cloud. Then we have
    \begin{align*}
        S_{\bfZ}=\frac{1}{n-1}(\bfZ \bfP_{n})(\bfZ \bfP_{n})^{T}\ \ \text{and}\ \ S_{\mathrm{D},\bfZ}=\frac{1}{n-1}(\bfZ \bfP_{n})^{T}(\bfZ \bfP_{n}).
    \end{align*}
\end{lemma}

\begin{proof}[Proof.\nopunct]
    The proof is straightforward if we notice that 
    \begin{align*}
        \bfZ \bfP_{n}=\bfZ(I_{n}-n^{-1}\bfone_{n}^{}\bfone_{n}^{T})=\bfZ-\overline{\bfZ}.\tag*{\qedhere}
    \end{align*}
\end{proof}

Applying PCA to the observed point cloud $P'=\{x_{1}',x_{2}',\ldots,x_{n}'\}$ involves evaluating the sample mean $\overline{x'}$, the sample covariance matrix $S_{P'}$, and then finding eigenvectors of $S_{P'}$ corresponding to the top few eigenvalues \cite{MR2247587}. From Lemma~\ref{samplecov123}, the sample covariance matrix $S_{P'} = \left( n-1\right) ^{-1}(P'\bfP_{n})(P'\bfP_{n})^{T}$. By singular value decomposition (SVD),
\begin{equation}
    \frac{P'\bfP_{n}}{\sqrt{n-1}}=\sum_{i\in [r]}\lambda_{i}^{1/2}h_{i}^{}u_{i}^{T}.\label{SVDnsdvjk}
\end{equation}
Here $r$ is the rank of $P'\bfP_{n}$, $\lambda_{1}\geq \lambda_{2}\geq \cdots \geq \lambda_{r}$ are the first $r$ largest positive eigenvalues of $S_{P'}$, $h_{1}, h_{2}, \ldots, h_{r}$ are the corresponding orthonormal eigenvectors of $S_{P'}$, and $u_{1}, u_{2}, \ldots, u_{r}$ are the corresponding orthonormal eigenvectors of $S_{{\rm D},P'}$. 

It can be shown that $h_{1}, h_{2},\ldots, h_{r}$ are the first $r$ principal components, and principal component scores (PC scores for short) are coordinates of the projected data \cite{MR2247587}. For each $i\in[n]$, the PCA projection of $x_{i}'$ onto the first $s$ components is denoted by $\dt{x}_{i}=(\dt{x}_{1i},\dt{x}_{2i},\ldots, \dt{x}_{si})^{T}$, where $\dt{x}_{ki}=h_{k}^{T}\left( x_{i}'-\overline{x'}\right)^{}$ is the $k$-th PC score of $x_{i}'$ ($k\in [s]$). We collect all the projected points into a new point cloud $\dt{P}=\{\dt{x}_{1},\dt{x}_{2},\ldots, \dt{x}_{n}\}$. This new point cloud $\dt{P}$ is called the compressed point cloud of $P'$ by PCA. 

The normalized PCA is another dimension-reduction technique defined by normalizing the PC scores. It can be used, for example, in the data clustering of mixture models \cite{yata2020geometric}. For $i\in[n]$, the normalized PCA projection of $x_{i}'$ onto the first $s$ components is denoted by $\hat{x}_{i}=(\hat{x}_{1i},\hat{x}_{2i},\ldots, \hat{x}_{si})^{T}$, where $\hat{x}_{ki}=\lambda_{k}^{-1/2}\dt{x}_{ki}^{}=\lambda_{k}^{-1/2}h_{k}^{T}\left( x_{i}'-\overline{x'}\right)^{}$ is the $k$-th normalized PC score of $x_{i}'$ ($k\in [s]$). We collect all the projected points into a new point cloud $\hat{P}=\{\hat{x}_{1},\hat{x}_{2},\ldots, \hat{x}_{n}\}$. This new point cloud $\hat{P}$ is called the compressed point cloud of $P'$ by normalized PCA. In what follows, the persistence diagram of $\hat{P}$ will be called the compressed persistence diagram.

In high-dimensional statistics, we prefer to utilize the eigenvectors of the dual sample covariance matrix to represent the PC scores. Notice that $h_{1}, h_{2}, \ldots, h_{r}$ are orthonormal by left multiplying both sides of~\eqref{SVDnsdvjk} by the row vector $h_{k}^{T}$, we have
\begin{align*}
     \frac{h_{k}^{T}P'\bfP_{n}^{}}{\sqrt{n-1}}=\sum_{i\in [r]}\lambda_{i}^{1/2}h_{k}^{T}h_{i}^{}u_{i}^{T}=\lambda_{k}^{1/2}u_{k}^{T}.
\end{align*}
Comparing $h_{k}^{T}P'\bfP_{n}^{} =(\dt{x}_{k1},\dt{x}_{k2},\ldots, \dt{x}_{kn})$ with $u_{k}^{T}=(u_{1k},u_{2k},\ldots, u_{nk})$, we have $\dt{x}_{ki}=\sqrt{(n-1)\lambda_{k}}\cdot u_{ik}$ and $\hat{x}_{ki}=\sqrt{n-1}\cdot u_{ik}$ $(i \in [n])$. Hence, $\dt{x}_{i}=\sqrt{n-1}\cdot(\lambda_{1}^{1/2}u_{i1}^{},\lambda_{2}^{1/2}u_{i2}^{},\ldots, \lambda_{s}^{1/2}u_{is})^{T}$ and $\hat{x}_{i}=\sqrt{n-1}\cdot (u_{i1},u_{i2},\ldots, u_{is})^{T}$. 

The pairwise distance is crucial in computing persistence diagrams in the Rips filtration case. The squared distance between $\dt{x}_{i}$ and $\dt{x}_{j}$ is given by
\begin{align}
    \left\| \dt{x}_{i}-\dt{x}_{j}\right\| ^{2}=(n-1)\sum_{k\in [s]}\lambda_{k}\left(u_{ik}-u_{jk}\right) ^{2}\label{sqdisthat}
\end{align}
and the squared distance between $\hat{x}_{i}$ and $\hat{x}_{j}$ is given by
\begin{align}
    \left\| \hat{x}_{i}-\hat{x}_{j}\right\| ^{2}=(n-1)\sum_{k\in [s]}\left(u_{ik}-u_{jk}\right) ^{2}.\label{sqdisthathat}
\end{align}
It is clear from~\eqref{sqdisthat}, \eqref{sqdisthathat} that the pairwise distances are determined by the eigenvalues or the eigenvectors of $S_{{\rm D}, P'}$. Furthermore, by a simple computation,
\begin{align}
    S_{{\rm D},P'} &= \frac{1}{n-1}(P'\bfP_{n})^{T}(P'\bfP_{n})\notag\\
    & = \frac{1}{n-1}\bfP_{n}(P+E)^{T}(P+E)\bfP_{n}\notag\\
    & = S_{{\rm D},E}+\frac{1}{n-1}\bfP_{n}(P^{T}P+P^{T}E+E^{T}P)\bfP_{n}.\label{equality of dual matrices}
\end{align}
Set $\bfMo = (n-1)^{-1}\bfP_{n}(P^{T}P+P^{T}E+E^{T}P)\bfP_{n}$. Since the last $(d-s)$ rows of $P$ are zero, $\bfMo$ depends only on the first $s$ rows of $E$ and $P$. Hence the matrix $\bfMo$ can be considered as a small noise component affecting $S_{{\rm D}, E}$ as $d$ tends to infinity. Through this observation, it is reasonable to claim that the eigenvectors of $S_{{\rm D}, P'}$ and $S_{{\rm D}, E}$ are asymptotically close.

\subsection{Matrix perturbation}\label{subsec6}

As discussed above, the matrix $S_{{\rm D}, P'}$ can be considered as a perturbation of the matrix $S_{{\rm D}, E}$ by the matrix $\bfMo$. Therefore, the matrix perturbation theory is applicable here. The theory of matrix perturbation has been widely studied, and several classical results exist, including Weyl’s theorem for eigenvalues and the Davis-Kahan theorem for eigenvectors. In recent years, several advancements and generalizations have been made \cite{vu2011singular}; \cite{o2018random}; \cite{fan2018eigenvector}; \cite{eldridge2018unperturbed}; \cite{yu2015useful}. In this paper, we primarily utilize a variant of the Davis–Kahan theorem presented in \cite{yu2015useful} as a lemma.

\begin{lemma}[\cite{yu2015useful}]\label{perturbation theo}
Let $\mathbf{\Sigma}$ and $\hat{\mathbf{\Sigma}}\in \mathbb{R}^{p\times p}$ be symmetric with eigenvalues $\lambda_{1}\geq \cdots \geq \lambda_{p}$ and $\hat{\lambda}_{1}\geq \cdots \geq \hat{\lambda}_{p}$, respectively. Fix $j\in [p]$, and assume that $\min\{\lambda_{j-1}-\lambda_{j},\lambda_{j}-\lambda_{j+1}\}>0$, where we define $\lambda_{0}=\infty$ and $\lambda_{p+1}=-\infty$. Let $v$ and $\hat{v}$ be unit vectors such that $\mathbf{\Sigma} v=\lambda_{j}v$ and $\hat{\mathbf{\Sigma}} \hat{v}=\hat{\lambda}_{j}\hat{v}$. Then
\begin{equation*}
    \left\| \hat{v}-v\right\| \leq \dfrac{2^{3/2} \| \hat{\mathbf{\Sigma} }-\mathbf{\Sigma} \|_{2} }{\min \left\{ \lambda _{j-1}-\lambda _{j},\lambda_{j}-\lambda _{j+1}\right\}},
\end{equation*}
unless the scalar product $\big\langle v, \hat{v}\big\rangle\geq 0$.
\end{lemma}

We claim that ${\rm rank}(S_{{\rm D},E})=n-1$ with probability $1$. Indeed,
\begin{align*}
    {\rm rank}(S_{{\rm D},E})={\rm rank}(E\bfP_{n})\leq {\rm rank}(\bfP_{n})=n-1.
\end{align*}
On the other hand, $e_{1},e_{2},\ldots,e_{n}$ are linearly independent with probability 1 by Proposition~\ref{LDwithprob0}, i.e., ${\rm rank}(E)=n$ with probability $1$. Then with probability 1,
\begin{align*}
    {\rm rank}(S_{{\rm D},E})&\geq {\rm rank}(E)+{\rm rank}(\bfP_{n})-n\\
    &= n+(n-1)-n\\
    &= n-1.
\end{align*}

Put $p=n$, $\mathbf{\Sigma}=S_{{\rm D},E}$, $\hat{\mathbf{\Sigma}}=S_{{\rm D},P'}$ in Lemma~\ref{perturbation theo}, and let $u_{1}^{E},u_{2}^{E},\ldots, u_{n-1}^{E}$ be the eigenvectors of $S_{{\rm D},E}$ corresponding to the first $(n-1)$ largest eigenvalues $\lambda_{1}^{E}\geq \cdots \geq \lambda_{n-1}^{E}$. It can be expected that $u_{k}^{E}$ and $u_{k}$ are getting close asymptotically whenever the difference $\lambda_{k-1}^{E}-\lambda _{k}^{E}$ between consecutive eigenvalues, called the eigengap, diverges to infinity in probability as $d\rightarrow \infty$ but $n$ is fixed. If this claim holds true, then we can conclude that the result of using the normalized PCA is strongly dominated by the noise part $E$.

However, proving the asymptotically infinite divergence of all the eigengaps of $S_{{\rm D}, E}$ is not a simple task. The remainder of this subsection will focus on this issue, and the core divergence theorem will be completely proved in Appendix~\ref{secA4}.

\begin{definition}\label{Wishartdistribution}
\normalfont Let $\bfG=(\bfG_{1}\ \bfG_{2}\ \cdots\ \bfG_{d})$ be an $n\times d$ matrix, where column vectors $\bfG_{1}, \bfG_{2},\ldots, \bfG_{d}$ are supposed to be \iid\ $n$-dimensional random vectors following the Gaussian distribution $\mathcal{N}_{n}(0,\Sigma)$. Set $\bfS=\bfG \bfG^{T}=\sum_{i\in [d]}\bfG_{i}^{}\bfG_{i}^{T}$. The Wishart distribution is the probability distribution of $\bfS$. We will write $\bfS\sim \mathcal{W}_{n}(\Sigma,d)$. In particular, $\bfS$ is said to follow the standard Wishart distribution if $\Sigma$ is an identity matrix.
\end{definition}

Let $\eta_{i}^{T}=(e_{i1},e_{i2},\ldots,e_{in})$ be the $i$-th row vector of $E$. A simple computation shows that ${\mathbb{E}}\eta_{i}=0$ and ${\mathbb{V}}(\eta_{i})={\mathbb{E}}\eta_{i}\eta_{i}^{T}=I_{n}$. It then follows that $\eta_{1},\ldots,\eta_{d}$ are \iid\ random vectors following the standard Gaussian distribution $\mathcal{N}_{n}(0,I_{n})$. Now, we let $\bfG=(\bfP_{n}\eta_{1}\ \bfP_{n}\eta_{2}\ \cdots\ \bfP_{n}\eta_{d})$. Then we have
\begin{align}
    S_{{\rm D},E}=\frac{1}{n-1}\bfP_{n}E^{T}E\bfP_{n}= \frac{1}{n-1}\bfG \bfG^{T}= \frac{1}{n-1}\sum_{i\in [d]}(\bfP_{n}\eta_{i})(\bfP_{n}\eta_{i})^{T}.\label{SDEformula}
\end{align}
Since $\bfP_{n}$ is non-random, $\bfP_{n}\eta_{1}, \bfP_{n}\eta_{2}, \ldots, \bfP_{n}\eta_{d}$ are \iid\ random vectors following the Gaussian distribution $\mathcal{N}_{n}(0,\bfP_{n})$. Hence the random matrix $\bfG \bfG^{T}=(n-1)S_{{\rm D},E}\sim \mathcal{W}_{n}(\bfP_{n},d)$ from Definition~\ref{Wishartdistribution}.

\begin{definition}
\normalfont Let ${\rm O}(n)$ be the orthogonal group in dimension $n$. A random matrix $Y\in {\rm O}(n)$ is said to have the Haar invariant distribution if $QY$ has the same distribution for every non-random orthogonal matrix $Q\in {\rm O}(n)$.
\end{definition}

\begin{lemma}[\cite{anderson1962introduction}]\label{Wishart}
Let $\bfS$ be an $n\times n$ random matrix following the standard Wishart distribution $\mathcal{W}_{n}(I_{n},d)$, and $\mathrm{h}_{1}, \mathrm{h}_{2}, \ldots, \mathrm{h}_{n}$ the unit eigenvectors of $\bfS$. Put $\bfH=(\mathrm{h}_{1}\ \mathrm{h}_{2}\ \cdots\ \mathrm{h}_{n})\in {\rm O}(n)$. Then $\bfH$ has the Haar invariant distribution.
\end{lemma}

Recall that $\bfP_{n}=I_{n}-n^{-1}\bfone_{n}^{}\bfone_{n}^{T}$, where $\bfone_{n}=(\underbrace{1,1,\ldots, 1}_{n})^{T}$. The $n$-th largest (i.e., the smallest) eigenvalues of $\bfP_{n}$ and $S_{{\rm D},E}$ are $0$, and the corresponding unit eigenvector is $n^{-1/2}\bfone_{n}$ since $\bfP_{n}\bfone_{n}=(I_{n}-n^{-1}\bfone_{n}^{}\bfone_{n}^{T})\bfone_{n}=0$. This indicates that the unit eigenvectors corresponding to the first $(n-1)$ largest eigenvalues of $S_{{\rm D},E}$ are orthogonal to $\bfone_{n}$, i.e., $\bfone_{n}^{T}u_{k}^{E}=0$ ($k\in [n-1]$). It follows that $\bfP_{n}u_{k}^{E}=(I_{n}-n^{-1}\bfone_{n}^{}\bfone_{n}^{T})u_{k}^{E}=u_{k}^{E}$. Based on these facts, we claim that the first $(n-1)$ largest eigenvalues of $S_{{\rm D}, E}$ multiplied by the scalar $(n-1)$ are exactly eigenvalues of an $(n-1)\times (n-1)$ standard Wishart distributed matrix.

\begin{lemma}\label{eigenvalueofstandardWishart}
    Let $E\subseteq \mathbb{R}^{d}$ be the noise point cloud, $S_{{\rm D},E}$ the dual sample covariance matrix of $E$, and $\lambda_{1}^{E}\geq \cdots \geq \lambda_{n-1}^{E}\geq \lambda_{n}^{E}=0$ the ordered eigenvalues of $S_{{\rm D},E}$. Then $(n-1)\lambda_{1}^{E}, \cdots, (n-1)\lambda_{n-1}^{E}$ are eigenvalues of an $(n-1)\times (n-1)$ random matrix following the standard Wishart distribution $\mathcal{W}_{n-1}(I_{n-1},d)$.
\end{lemma}

\begin{proof}[Proof.\nopunct]
    We note that $\bfP_{n}$ is a real symmetric matrix. The eigendecomposition of $\bfP_{n}$ is
    \begin{equation*}
        \bfP_{n} = \bfgma \Lambda_{0} \bfgma^{T}, 
    \end{equation*}
    where $\Lambda_{0} = {\rm diag}(\underbrace{1, 1, \ldots, 1}_{n-1}, 0)$ denotes the diagonal matrix whose diagonal elements are eigenvalues of $\bfP_{n}$, and $\bfgma=(\gamma_{1}\ \gamma_{2}\ \cdots\ \gamma_{n-1}\ n^{-1/2}\bfone_{n})$ is the orthogonal matrix whose column vectors are the corresponding unit eigenvectors. For the convenience, we denote the block matrix $(\gamma_{1}\ \gamma_{2}\ \cdots\ \gamma_{n-1})$ by $\bfgma_{1}$. Put $U_{1}^{E}=(u_{1}^{E}\ u_{2}^{E}\ \cdots\ u_{n-1}^{E})$ and $U^{E}=(U_{1}^{E}\ n^{-1/2}\bfone_{n})$, where $u_{1}^{E},u_{2}^{E},\ldots, u_{n-1}^{E}$ are the eigenvectors of $S_{{\rm D},E}$ corresponding to the eigenvalues $\lambda_{1}^{E}, \lambda_{2}^{E}, \cdots, \lambda_{n-1}^{E}$. Put $\Lambda_{1}^{E}={\rm diag}(\lambda_{1}^{E}, \lambda_{2}^{E}, \ldots, \lambda_{n-1}^{E})$ and $\Lambda^{E}={\rm diag}(\Lambda_{1}^{E},0)$. 
    
From~\eqref{SDEformula}, we have 
    \begin{equation*}
      \bfG \bfG^{T}=(n-1)S_{{\rm D},E}=U^{E}[(n-1)\Lambda^{E}](U^{E})^{T}.
    \end{equation*}
Then $\bfgma^{T}\bfG \bfG^{T}\bfgma=\bfgma^{T}U^{E}[(n-1)\Lambda^{E}](U^{E})^{T}\bfgma$. Noticing that
\begin{align}
\bfgma^{T}U^{E}=\begin{pmatrix}
\bfgma_{1}^{T} \\
\frac{1}{\sqrt{n}}\bfone_{n}^{T}
\end{pmatrix}(U_{1}^{E}\ \frac{1}{\sqrt{n}}\bfone_{n})=\begin{pmatrix}
\bfgma _{1}^{T}U_{1}^{E} & 0 \\
0 & 1
\end{pmatrix},\label{TDncmbxmb}
\end{align}
we have
\begin{align}
\bfgma^{T}\bfG \bfG^{T}\bfgma&=\begin{pmatrix}
\bfgma _{1}^{T}U_{1}^{E} & 0 \\
0 & 1
\end{pmatrix}
\begin{pmatrix}
(n-1)\Lambda_{1}^{E} & 0 \\
0 & 0
\end{pmatrix}
\begin{pmatrix}
(U_{1}^{E})^{T}\bfgma_{1} & 0 \\
0 & 1
\end{pmatrix}\notag\\
&=\begin{pmatrix}
\bfgma _{1}^{T}U_{1}^{E}[(n-1)\Lambda_{1}^{E}](U_{1}^{E})^{T}\bfgma_{1} & 0 \\
0 & 0
\end{pmatrix}.\label{equation3.1}
\end{align}

On the other hand, for each $k\in [d]$,
\begin{align}
    \bfgma^{T}\bfP_{n}\eta_{k}= \begin{pmatrix}
\bfgma_{1}^{T}\bfP_{n}\eta_{k} \\
0
\end{pmatrix}, \label{VGskdjjjd}
\end{align}
where $\eta_{k}^{T}$ is the $k$-th row vector of $E$. We let $v_{k}=\bfgma_{1}^{T}\bfP_{n}\eta_{k}$. Then $v_{1}, v_{2}, \ldots, v_{d}$ are \iid\ random vectors following $\mathcal{N}_{n-1}(0,I_{n-1})$ because $\bfgma_{1}^{T}\bfP_{n}$ is non-random and $\eta_{1}, \eta_{2}, \ldots, \eta_{d}$ are \iid\ random vectors following $\mathcal{N}_{n}(0,I_{n})$. Then from~\eqref{SDEformula}, \eqref{VGskdjjjd}, we have
\begin{align}
\bfgma^{T}\bfG \bfG^{T}\bfgma=\sum_{k\in [d]}(\bfgma^{T}\bfP_{n}\eta_{k})(\bfgma^{T}\bfP_{n}\eta_{k})^{T}=\sum_{k\in [d]}\begin{pmatrix}
v_{k} \\
0
\end{pmatrix}(v_{k}^{T}\ 0)=\begin{pmatrix}
\sum\limits_{k\in [d]}v_{k}v_{k}^{T} & 0 \\
0 & 0
\end{pmatrix}.\label{equation3.2}
\end{align}

By comparing \eqref{equation3.1} with \eqref{equation3.2}, we obtain
\begin{equation}
    \sum\limits_{k\in [d]}v_{k}^{}v_{k}^{T}=\bfgma_{1}^{T}U_{1}^{E}[(n-1)\Lambda_{1}^{E}](U_{1}^{E})^{T}\bfgma_{1}.\label{equation3.3}
\end{equation}
By noticing~\eqref{TDncmbxmb}, $\bfgma_{1}^{T}U_{1}^{E}$ is orthogonal since $\bfgma^{T}U^{E}$ is. Consequently, \eqref{equation3.3} is the eigendecomposition of the matrix $\sum_{k\in [d]}v_{k}^{}v_{k}^{T}$ following the standard Wishart distribution $\mathcal{W}_{n-1}(I_{n-1},d)$.
\end{proof}

\begin{remark}\label{gammaUHaardistri}
    From Lemma~\ref{Wishart} and~\eqref{equation3.3}, one can conclude that $\bfgma_{1}^{T}U_{1}^{E}$ has the Haar invariant distribution.
\end{remark}
    
Lemma~\ref{eigenvalueofstandardWishart} indicates that in order to analyze the eigengaps of $S_{{\rm D}, E}$, it suffices to study the eigengaps of the standard Wishart distributed matrix. The following core theorem shows the asymptotically infinite divergence of the minimum eigengap of the standard Wishart distributed matrix as $d\rightarrow \infty$ but $n$ is fixed.

\begin{theorem}[Asymptotically infinite divergence of eigengaps]\label{corethm} Let $W$ be an $m \times m$ ($m\geq 2$) random matrix following the standard Wishart distribution $\mathcal{W}_{m}(I_{m},d)$, and $\lambda_{1}\geq \lambda_{2}\geq \cdots \geq \lambda_{m}$ the ordered eigenvalues of $W$. Then for every constant $C>0$,
\begin{equation*}
\mathbb{P}\left(\min_{k\in [m-1]}\{\lambda_{k}-\lambda_{k+1}\}>C\right)\rightarrow 1
\end{equation*}
holds as $d\rightarrow \infty$ but $m$ is fixed.
\end{theorem}

For the full proof we refer the reader to Appendix~\ref{secA4}. In Theorem~\ref{corethm}, if we let $m=n-1$ and put $W=\sum_{k\in [d]}v_{k}^{}v_{k}^{T}$, where $v_{k}$ is defined in the proof of Lemma~\ref{eigenvalueofstandardWishart}, then the minimum eigengap of $S_{{\rm D}, E}$ diverges to infinity with high probability.

\begin{corollary}\label{absdjhwb}
    Let $E\subseteq \mathbb{R}^{d}$ be the noise point cloud, $S_{{\rm D}, E}$ the dual sample covariance matrix of $E$, and $\lambda_{1}^{E}\geq \cdots \geq \lambda_{n-1}^{E}\geq \lambda_{n}^{E}=0$ the ordered eigenvalues of $S_{{\rm D},E}$. Then
    \begin{equation*}
        \frac{1}{\min\limits_{k\in [n-1]}\{\lambda_{k}^{E}-\lambda_{k+1}^{E}\}} = o_{\mathbb{P}}(1)
    \end{equation*}
holds as $d\rightarrow \infty$.
\end{corollary}

\begin{proof}[Proof.\nopunct]
    We first note that 
    \begin{align*}
        \min_{k\in [n-1]}\{\lambda_{k}^{E}-\lambda_{k+1}^{E}\}=\min \Big \{\min_{k\in [n-2]}\{\lambda_{k}^{E}-\lambda_{k+1}^{E}\}, \lambda_{n-1}^{E} \Big \}.
    \end{align*}
    By Lemma~\ref{eigenvalueofstandardWishart} and Theorem~\ref{corethm}, for every constant $C>0$, 
    \begin{equation*}
    \mathbb{P}\left(\min_{k\in [n-2]}\{\lambda_{k}^{E}-\lambda_{k+1}^{E}\}>C\right)\rightarrow 1
    \end{equation*}
as $d\rightarrow +\infty$ but $n$ is fixed. By Proposition~\ref{orderofeigenvalue}, we have
\begin{equation*}
    \mathbb{P}\left(\lambda_{n-1}^{E}>C\right)\rightarrow 1.
    \end{equation*}
Consequently, 
    \begin{equation*}
    \mathbb{P}\left(\min_{k\in [n-1]}\{\lambda_{k}^{E}-\lambda_{k+1}^{E}\}>C\right)\rightarrow 1.
    \end{equation*}
We finally complete the proof by letting $C=1/\varepsilon$ for every $\varepsilon>0$.
\end{proof}

Now, we are able to show the asymptotic closeness of eigenvectors using the results obtained so far.

\begin{proposition}\label{eigenclose}
Let $P'$ and $E$ be the observed and the noise point clouds in $\mathbb{R}^{d}$, $S_{{\rm D}, P'}$ and $S_{{\rm D}, E}$ the dual sample covariance matrices of $P'$ and $E$ with eigenvalues $\lambda_{1}\geq \cdots \geq \lambda_{n}$ and $\lambda_{1}^{E}\geq \cdots \geq \lambda_{n}^{E}$, respectively. For every $k\in [n]$, we let $u_{k}$ and $u_{k}^{E}$ be the unit eigenvectors corresponding to $\lambda_{k}$ and $\lambda_{k}^{E}$, respectively. Then $\big\| u_{k}^{}-u_{k}^{E}\big\|=o_{\mathbb{P}}(1)$ as $d\rightarrow \infty$, unless the scalar product $\big\langle u_{k}^{}, u_{k}^{E}\big\rangle\geq 0$.
\end{proposition}

\begin{proof}[Proof.\nopunct]
We shall prove this proposition for every $k\in [n-1]$ since the case that $k=n$ is trivial. By noticing~\eqref{equality of dual matrices} and Lemma~\ref{perturbation theo}, we have
\begin{equation*}
    \left\| u_{k}^{}-u_{k}^{E}\right\| \leq \dfrac{2^{3/2} \| \bfMo \|_{2} }{\min \left\{ \lambda_{k-1}^{E}-\lambda _{k}^{E},\lambda_{k}^{E}-\lambda _{k+1}^{E}\right\}}\leq \dfrac{2^{3/2} \| \bfMo \|_{2} }{\min\limits_{k\in [n-1]}\{\lambda_{k}^{E}-\lambda_{k+1}^{E}\}}.
\end{equation*}
Since $\bfMo$ depends only on the first $s$ rows of $E$ and $P$, we have $\| \bfMo \|_{2}=O_{\mathbb{P}}(1)$ as $d\rightarrow \infty$. By Corollary~\ref{absdjhwb}, Propositions~\ref{prop0a}, \ref{prop0b}, we obtain $\big\| u_{k}^{}-u_{k}^{E}\big\|=o_{\mathbb{P}}(1)$ as $d\rightarrow \infty$.
\end{proof}

\begin{corollary}\label{coordinateclose}
For every $k\in [n]$, we let $u_{k}=(u_{1k},u_{2k},\ldots, u_{nk})^{T}$ and $u_{k}^{E}=(u_{1k}^{E},u_{2k}^{E},\ldots, u_{nk}^{E})^{T}$. Then $\big\lvert u_{ik}^{}-u_{ik}^{E}\big\rvert=o_{\mathbb{P}}(1)$ holds for every $i\in [n]$ as $d\rightarrow \infty$, unless the scalar product $\big\langle u_{k}^{}, u_{k}^{E}\big\rangle\geq 0$.
\end{corollary}

In what follows, we shall always assume that $\big\langle u_{k}^{}, u_{k}^{E}\big\rangle\geq 0$ since we can freely choose the sign of eigenvectors. Moreover, we write $\hat{E}=\{\hat{e}_{1},\hat{e}_{2},\ldots, \hat{e}_{n}\}$ for the compressed point cloud obtained by projecting the noise point cloud $E=\{e_{1},e_{2},\ldots, e_{n}\}$ onto the principal subspace by the first $s$ components with the normalized PCA technique.

\begin{theorem}\label{convergeresult}
Let $P'$ and $E$ be the observed and the noise point clouds in $\mathbb{R}^{d}$, $\hat{P}$ and $\hat{E}$ the compressed point clouds of $P'$ and $E$ by normalized PCA, respectively. Then
\begin{align*}
  d_{B}(\bar{{\rm D}}( \hat{P}), \bar{{\rm D}}( \hat{E}) ) = o_{\mathbb{P}}(1)
\end{align*}
holds as $d\rightarrow \infty$.
\end{theorem}

\begin{proof}[Proof.\nopunct]
Define a correspondence $\hat{\Corr}$$:\hat{P}\rightrightarrows \hat{E}$ by setting
\begin{equation*}
    \hat{\Corr} = \setc*{(\hat{x}_{i},\hat{e}_{i})\in \hat{P}\times \hat{E}}{i\in [n]}.
\end{equation*}
Then we have
\begin{align*}
    \dis(\hat{\Corr}) = \max_{\substack{i,j\in [n] \\ i<j}}  \Big\lvert  \| \hat{x}_{i}-\hat{x}_{j}\| -\| \hat{e}_{i}-\hat{e}_{j}\| \Big\rvert \leq \max_{\substack{i,j\in [n] \\ i<j}}  \sqrt{\Big\lvert  \| \hat{x}_{i}-\hat{x}_{j}\|^{2} -\| \hat{e}_{i}-\hat{e}_{j}\|^{2} \Big\rvert }.
\end{align*}
We recall that $\|\hat{x}_{i}-\hat{x}_{j}\| ^{2}=(n-1)\sum_{k\in [s]}\left(u_{ik}-u_{jk}\right) ^{2}$ from~\eqref{sqdisthathat}. The similar consideration of applying the normalized PC scores to the noise point cloud $E$ yields $\|\hat{e}_{i}-\hat{e}_{j}\| ^{2}=(n-1)\sum_{k\in [s]}\left(u_{ik}^{E}-u_{jk}^{E}\right) ^{2}$. It follows that
    \begin{align*}
    &\ \ \ \ \frac{1}{n-1} \Big\lvert  \| \hat{x}_{i}-\hat{x}_{j}\|^{2} -\| \hat{e}_{i}-\hat{e}_{j}\|^{2} \Big\rvert = \Bigg\lvert \sum\limits_{k\in [s]}\left(u_{ik}-u_{jk}\right) ^{2}-\left(u_{ik}^{E}-u_{jk}^{E}\right) ^{2}\Bigg\rvert \\
    &=\Bigg\lvert \sum\limits_{k\in [s]}\left[u_{ik}^{2}-(u_{ik}^{E})^{2}\right]+\left[u_{jk}^{2}-(u_{jk}^{E}) ^{2}\right]+2\left(u_{ik}^{E}u_{jk}^{E}-u_{ik}u_{jk}\right)\Bigg\rvert \\
    &\leq \sum\limits_{k\in [s]}\big\lvert u_{ik}^{2}-(u_{ik}^{E})^{2}\big\rvert +\big\lvert u_{jk}^{2}-(u_{jk}^{E})^{2}\big\rvert +2\big\lvert u_{ik}^{E}\big\rvert \big\lvert u_{jk}^{E}-u_{jk}\big\rvert +2\big\lvert u_{jk}\big\rvert \big\lvert u_{ik}^{E}-u_{ik}\big\rvert \\
    &=\sum\limits_{k\in [s]}\big\lvert u_{ik}-u_{ik}^{E}\big\rvert \cdot \left(\big\lvert u_{ik}+u_{ik}^{E}\big\rvert +2\big\lvert u_{jk}\big\rvert \right)+\big\lvert u_{jk}-u_{jk}^{E}\big\rvert \cdot \left(\big\lvert u_{jk}+u_{jk}^{E}\big\rvert +2\big\lvert u_{ik}^{E}\big\rvert \right)\\
    &\leq 4\sum\limits_{k\in [s]}\big\lvert u_{ik}-u_{ik}^{E}\big\rvert +\big\lvert u_{jk}-u_{jk}^{E}\big\rvert.
    \end{align*}
For every $k\in [s]$ and $i\in [n]$, $\big\lvert u_{ik}^{}-u_{ik}^{E}\big\rvert =o_{\mathbb{P}}(1)$ holds as $d\rightarrow \infty$ by Corollary~\ref{coordinateclose}. It follows that for every $i,j\in [n]$, $\Big\lvert  \| \hat{x}_{i}-\hat{x}_{j}\|^{2} -\| \hat{e}_{i}-\hat{e}_{j}\|^{2} \Big\rvert =o_{\mathbb{P}}(1)$ holds as $d\rightarrow \infty$, implying that 
\begin{equation}\label{eqn:asymp of reduced pairwise dist}
    \Big\lvert  \| \hat{x}_{i}-\hat{x}_{j}\|-\| \hat{e}_{i}-\hat{e}_{j}\|\Big\rvert =o_{\mathbb{P}}(1)
\end{equation}
and
\begin{equation*}
    \dis(\hat{\Corr})=o_{\mathbb{P}}(1)
\end{equation*}
as $d\rightarrow \infty$ by Propositions~\ref{prop0c}, \ref{prop0d}. Consequently, by~\eqref{GromovHausdorffdistdef}, Theorem~\ref{lem0b}, and Proposition~\ref{prop0b}, we obtain
\begin{equation*}
    d_{B}(\bar{{\rm D}}( \hat{P}), \bar{{\rm D}}( \hat{E}) ) = o_{\mathbb{P}}(1)
\end{equation*}
as $d\rightarrow \infty$.
\end{proof}

\begin{theorem}\label{PCAPDlimit}
Let $P'$ and $E$ be the observed and the noise point clouds in $\mathbb{R}^{d}$, $\hat{P}$ and $\hat{E}$ the compressed point clouds of $P'$ and $E$ by normalized PCA, respectively. Then
\begin{equation*}
    d_{B}(\bar{{\rm D}}(P), \bar{{\rm D}}( \hat{P}))-d_{B}(\bar{{\rm D}}(P), \bar{{\rm D}}( \hat{E})) = o_{\mathbb{P}}(1)
\end{equation*}
holds as $d\rightarrow \infty$.
\end{theorem}

\begin{proof}[Proof.\nopunct]
The proof will be finished by the triangle inequality
\begin{equation*}
    \Big\lvert d_{B}(\bar{{\rm D}}(P), \bar{{\rm D}}( \hat{P}))-d_{B}(\bar{{\rm D}}(P), \bar{{\rm D}}( \hat{E}))\Big\rvert \leq d_{B}(\bar{{\rm D}}( \hat{P}), \bar{{\rm D}}( \hat{E})),
\end{equation*}
together with Propositions~\ref{prop0b}, \ref{prop0c}, Theorem~\ref{convergeresult}.
\end{proof}

Theorem~\ref{PCAPDlimit} indicates that the bottleneck distance between the persistence diagrams of the original and the compressed point clouds is mainly dominated by the noise part. In order to find out the asymptotic behavior of $d_{B}(\bar{{\rm D}}(P), \bar{{\rm D}}( \hat{P}))$, we only need to focus on the asymptotic behavior of $d_{B}(\bar{{\rm D}}(P), \bar{{\rm D}}(\hat{E}))$.

\subsection{Moments of pairwise distances}\label{subsec7}

The focus of the discussion now shifts to the study of the asymptotic behavior of $d_{B}(\bar{{\rm D}}(P), \bar{{\rm D}}(\hat{E}))$. In the Rips filtration, the pairwise distances play a crucial role in determining the filtration. These distances are random variables, and obtaining their probability distributions is extremely hard. Fortunately, there is a method called Weingarten calculus that we could utilize to compute their moments.

Weingarten calculus is a systematic method for computing integrals on groups, more precisely, Haar integrals on some compact groups. This method was first considered by theoretical physicists in the 1970s in nonabelian gauge theories and then continued to develop in the early 2000s by mathematicians in random matrix theory. A comprehensive introduction to Weingarten calculus will be provided in Appendix~\ref{secA5}.

We will follow the notations given in the proof of Lemma~\ref{eigenvalueofstandardWishart}. We denote the $(n-1)\times (n-1)$ matrix $\bfgma _{1}^{T}U_{1}^{E}$ by $g$ and let $g = (g_{ij})_{i,j\in [n-1]}$. Suppose $\bfgma_{1}=(\gamma_{ij})_{i\in[n],j\in[n-1]}$. Then $\bfgma_{1}g = (\sum_{m\in [n-1]}\gamma_{im}g_{mj})_{i\in [n], j\in [n-1]}$.

Notice that
\begin{align*}
    U^{E}&=\bfgma(\bfgma^{T}U^{E})=\left(\bfgma_{1}\ \frac{1}{\sqrt{n}}\bfone_{n}\right)\begin{pmatrix}
\bfgma _{1}^{T}U_{1}^{E} & 0 \\
0 & 1
\end{pmatrix}=\left(\bfgma_{1}g\ \  \frac{1}{\sqrt{n}}\bfone_{n}\right).
\end{align*}
Then for all $i,j\in [n]$ with $i<j$ and for every $k\in [n-1]$, we have $u_{ik}^{E}=\sum_{m\in [n-1]}\gamma_{im}g_{mk}$ and $u_{jk}^{E}=\sum_{m\in [n-1]}\gamma_{jm}g_{mk}$. Thus,
\begin{align}
    \frac{1}{n-1}\|\hat{e}_{i}-\hat{e}_{j}\|^{2} = \sum\limits_{k\in [s]}\left(u_{ik}^{E}-u_{jk}^{E}\right) ^{2} = \sum\limits_{k\in [s]}\left[\sum_{m\in [n-1]}(\gamma_{im}-\gamma_{jm})g_{mk}\right]^{2}.\label{dahsj}
\end{align}
Denote $\gamma_{im}-\gamma_{jm}$ by $C_{m}^{ij}$ in~\eqref{dahsj}. By expanding brackets, we obtain
\begin{align}
    \frac{1}{n-1}\|\hat{e}_{i}-\hat{e}_{j}\|^{2} & = \sum\limits_{k\in [s]}\left(\sum_{m\in [n-1]}C_{m}^{ij}g_{mk}\right)^{2}\notag\\
    &= \sum\limits_{k\in [s]}\sum\limits_{m\in [n-1]}\left(C_{m}^{ij}\right)^{2}g_{mk}^{2}+\sum\limits_{k\in [s]}\sum\limits_{\substack{m,t\in [n-1] \\ m\neq t}}C_{m}^{ij}C_{t}^{ij}g_{mk}g_{tk}\notag\\
    & = \sum\limits_{m\in [n-1]}\left(C_{m}^{ij}\right)^{2}\sum\limits_{k\in [s]}g_{mk}^{2}+\sum\limits_{\substack{m,t\in [n-1] \\ m\neq t}}C_{m}^{ij}C_{t}^{ij}\sum\limits_{k\in [s]}g_{mk}g_{tk}.\label{eqofpairwisedistance}
\end{align}
It is now clear from~\eqref{eqofpairwisedistance} that in order to compute the expected value of the squared pairwise distance $\|\hat{e}_{i}-\hat{e}_{j}\|^{2}$, it suffices to compute the first moments of $g_{mk}^{2}$ and $g_{mk}g_{tk}$ ($m,t\in [n]$, $m\neq t$).

\begin{proposition}\label{1st and 2nd moments of pairwise distance}
Let $E \subseteq \mathbb{R}^{d}$ be the noise point cloud, and $\hat{E}$ the compressed point cloud of $E$ by normalized PCA. Then
    \begin{equation*}
        {\mathbb{E}}\|\hat{e}_{i}-\hat{e}_{j}\|^{2}=2s\ \ \text{and}\ \ {\mathbb{V}}\|\hat{e}_{i}-\hat{e}_{j}\|^{2}=\frac{-8s^{2}+8(n-1)s}{n+1}
    \end{equation*}
hold for all $i,j\in [n]$ with $i<j$.
\end{proposition}

\begin{proof}[Proof.\nopunct]
    From~\eqref{eqofpairwisedistance} we have
    \begin{align*}
        \frac{1}{n-1}\mathbb{E}\|\hat{e}_{i}-\hat{e}_{j}\|^{2} = \sum\limits_{m\in [n-1]}\left(C_{m}^{ij}\right)^{2}\sum\limits_{k\in [s]}\mathbb{E}g_{mk}^{2}+\sum\limits_{\substack{m,t\in [n-1] \\ m\neq t}}C_{m}^{ij}C_{t}^{ij}\sum\limits_{k\in [s]}\mathbb{E}g_{mk}g_{tk}.
    \end{align*}
    According to Remark \ref{gammaUHaardistri}, $g=\bfgma _{1}^{T}U_{1}^{E}$ has the Haar probability distribution. By Proposition \ref{k=2}, we have $\mathbb{E}g_{mk}^{2}=1/(n-1)$ and $\mathbb{E}g_{mk}g_{tk}=0$ ($m,t\in [n]$, $m\neq t$) if $N=n-1$ in Proposition~\ref{k=2}. Hence
    \begin{align}
        \frac{1}{n-1}\mathbb{E}\|\hat{e}_{i}-\hat{e}_{j}\|^{2} & = \sum\limits_{m\in [n-1]}\left(C_{m}^{ij}\right)^{2}\sum\limits_{k\in [s]}\frac{1}{n-1}=\frac{s}{n-1}\sum\limits_{m\in [n-1]}\left(\gamma_{im}-\gamma_{jm}\right)^{2}\nonumber\\
        & = \frac{s}{n-1}\left(\sum_{m\in [n-1]}\gamma _{im} ^{2}+\sum_{m\in [n-1]}\gamma _{jm} ^{2}-2\sum_{m\in [n-1]}\gamma _{im}\gamma _{jm}\right)\nonumber\\
        & = \frac{s}{n-1}\left[\left(1-\frac{1}{n}\right)+\left(1-\frac{1}{n}\right)-2\cdot \left(-\frac{1}{n}\right)\right]\nonumber\\
        & = \frac{2s}{n-1}.\label{eqn: the explicit form of sum_square_C^ij_m}
    \end{align}
    This gives ${\mathbb{E}}\|\hat{e}_{i}-\hat{e}_{j}\|^{2}=2s$. Regarding the variance of $\|\hat{e}_{i}-\hat{e}_{j}\|^{2}$, we refer the reader to Appendix~\ref{secA53} for the detailed computation. 
\end{proof}

The following is immediate from Proposition~\ref{1st and 2nd moments of pairwise distance} combining with Chebyshev's inequality.

\begin{proposition}\label{asymptomatic behavior of pairwise distance of reduced noise}
    Let $E \subseteq \mathbb{R}^{d}$ be the noise point cloud, and $\hat{E}$ the compressed point cloud of $E$ by normalized PCA. Then $\|\hat{e}_{i}-\hat{e}_{j}\| = O_{\mathbb{P}}(1)$ holds as $d\rightarrow \infty$.
\end{proposition}

\subsection{Results of the normalized PCA projection}\label{subsec8}

The above computations and analysis show that for the low-sample-size original point clouds affected by high-dimensional noise, the geometric information of the original point clouds will be lost after applying the normalized PCA. This is because the moments of the pairwise distance after projection do not retain information about the original pairwise distance. Nevertheless, the normalized PCA does mitigate the curse of dimensionality on persistence diagrams from Proposition~\ref{asymptomatic behavior of pairwise distance of reduced noise}. We can now formulate our main results.

\begin{theorem}\label{bottleneckdistofreducedPD}
Let $P$ and $P'$ be the original and the observed point clouds in $\mathbb{R}^{d}$, respectively. Let $\hat{P}$ denote the compressed point cloud of $P'$ by normalized PCA. Then 
$$d_{B}(\bar{{\rm D}}(P), \bar{{\rm D}}(\hat{P})) = O_{\mathbb{P}}(1)$$ 
holds as $d\rightarrow \infty$.
\end{theorem}
\begin{proof}[Proof.\nopunct]
Define a correspondence $\Corr''$$:P\rightrightarrows \hat{E}$ by setting
\begin{equation*}
    \Corr'' = \setc*{(x_{i},\hat{e}_{i})\in P\times \hat{E}}{i\in [n]}.
\end{equation*}
From~\eqref{defofdistortion}, \eqref{GromovHausdorffdistdef}, \eqref{stabilityforRipsfiltration}, we have
\begin{align*}
    d_{B}(\bar{{\rm D}}(P), \bar{{\rm D}}( \hat{E}))\leq 2d_{GH}(P,\hat{E})\leq \dis(\Corr'')= \max_{\substack{i,j\in [n] \\ i<j}} \Big \lvert  \| x_{i}-x_{j}\| -\| \hat{e}_{i}-\hat{e}_{j}\| \Big\rvert.
\end{align*}
Combining Propositions~\ref{prop0b}, \ref{prop0c}, \ref{prop0d}, and \ref{asymptomatic behavior of pairwise distance of reduced noise} yields $d_{B}(\bar{{\rm D}}(P), \bar{{\rm D}}( \hat{E}))=O_{\mathbb{P}}(1)$ as $d\rightarrow \infty$. Finally, $d_{B}(\bar{{\rm D}}(P), \bar{{\rm D}}( \hat{P}))=O_{\mathbb{P}}(1)$ holds as $d\rightarrow \infty$ by Theorem~\ref{PCAPDlimit}. 
\end{proof}

\begin{theorem}\label{hausdorffdistofreducedPD}
Let $P$ and $P'$ be the original and the observed point clouds in $\mathbb{R}^{d}$, respectively. Let $\hat{P}$ denote the compressed point cloud of $P'$ by normalized PCA. Suppose that ${\rm D}_{N}(P)$ is non-empty and that ${\rm D}_{N}(\hat{P})$ is non-empty almost surely. Then
$$d_{H}({\rm D}_{N}(P), {\rm D}_{N}( \hat{P})) = O_{\mathbb{P}}(1)$$ 
holds as $d\rightarrow \infty$.
\end{theorem}

\begin{proof}[Proof.\nopunct]
By the definition of Hausdorff distance, we have
\begin{align*}
    d_{H}({\rm D}_{N}\left(P\right), {\rm D}_{N}( \hat{P}) )& \leq \sup _{\omega\in {\rm D}_{N}\left( P\right)}\sup _{\hat{\omega}\in {\rm D}_{N}( \hat{P})}\| \omega-\hat{\omega}\| _{\infty }\\
    & = \max _{\omega\in {\rm D}_{N}\left( P\right) }\max _{\hat{\omega}\in {\rm D}_{N}( \hat{P})}\max \big\{\big\lvert b_{\omega}-b_{\hat{\omega}}\big\rvert ,\big\lvert d_{\omega}-d_{\hat{\omega}}\big\rvert \big\}\\
    &\leq \max _{\omega\in {\rm D}_{N}\left( P\right) }\max _{\substack{i,j\in [n] \\ i<j}}\max\left\{\bigg\lvert b_{\omega}-\frac{\| \hat{x}_{i}-\hat{x}_{j}\|}{2}\bigg\rvert ,\bigg\lvert d_{\omega}-\frac{\| \hat{x}_{i}-\hat{x}_{j}\|}{2}\bigg\rvert \right\}.
\end{align*}
Combining Proposition~\ref{asymptomatic behavior of pairwise distance of reduced noise} with~\eqref{eqn:asymp of reduced pairwise dist} yields $\|\hat{x}_{i}-\hat{x}_{j}\|=O_{\mathbb{P}}(1)$. Consequently, $d_{H}({\rm D}_{N}\left(P\right), {\rm D}_{N}(\hat{P}) )=O_{\mathbb{P}}(1)$ holds as $d\rightarrow \infty$.
\end{proof}

In the Rips filtration case, Theorem~\ref{classification theorem} shows that the bottleneck distance between the original and the observed persistence diagrams is eventually bounded, while the Hausdorff distance is eventually unbounded. As a comparison, Theorems~\ref{bottleneckdistofreducedPD}, \ref{hausdorffdistofreducedPD} show that both the bottleneck distance and the Hausdorff distance between the original and the compressed persistence diagrams are eventually bounded, indicating that the use of normalized PCA can mitigate the curse of dimensionality on persistence diagrams.

To validate Theorems \ref{bottleneckdistofreducedPD}, \ref{hausdorffdistofreducedPD}, a numerical experiment is performed using the same data as in Section~\ref{subsubsec4}. The $1$st persistence diagrams of $P$ (on the left-hand side) and $\hat{P}$ (on the right-hand side) in different dimensions are shown side by side in Fig.~\ref{fig:redPD}. The figure illustrates that the compressed persistence diagram does not move far away from the original one as $d$ tends to infinity. Moreover, all birth-death pairs in the compressed persistence diagram randomly remain within a bounded area $\setc*{\left(x,y\right)\in [0.08,0.39]\times [0.08,0.54]}{x<y}$ in the plane, though $d$ tends to infinity. Consequently, Fig.~\ref{fig:redPD} confirms that both the bottleneck distance and the Hausdorff distance are eventually bounded in probability.

\section{Discussion}\label{sec9}

In this work, we have rigorously shown the unreliability of using the observed persistence diagram in the HDLSS data setting. The observed persistence diagram is inconsistent with the original one in the Rips filtration case and strongly inconsistent in the \v{C}ech filtration case as $d$ tends to infinity but $n$ is fixed. In the sequel, we have proposed that the use of normalized PCA can enhance the consistency level in the Rips filtration case. The following Table~\ref{tab2} summarizes the classification of the similarity before and after applying the normalized PCA.

\begin{table}[htbp]
\centering
\caption{Classification of the asymptotic similarity}\label{tab2}
\begin{tblr}{
  width = \linewidth,
  colspec = {Q[125]Q[296]Q[213]Q[304]},
  row{2} = {c},
  row{3} = {c},
  cell{1}{2} = {c},
  cell{1}{3} = {c},
  cell{1}{4} = {c},
  cell{2}{3} = {r=2}{},
  cell{2}{4} = {r=2}{},
  cell{3}{2} = {},
  vlines,
  hline{1-2,4} = {-}{},
  hline{3} = {1-2}{},
}
& Rips & \v{C}ech & Rips+normalized PCA \\ {\\ \vspace{-1ex} $N=0$} & {strong bottleneck \\ inconsistency} & {strong bottleneck \\ inconsistency} & {(at worst) bottleneck inconsistency \\ (sub-class II)} \\
{\\ \vspace{-1ex} $N>0$} & {bottleneck inconsistency \\(sub-class III)} & &
\end{tblr}
\end{table}

The following aspects are worth investigating in the future:

\begin{enumerate}[{(1)}]
    \item We apply the normalized PCA to mitigate the curse of dimensionality on persistence diagrams. Unfortunately, the precise consistency level after applying the normalized PCA is not yet known. Moreover, it is evident from Fig.~\ref{fig:redPD} that the use of normalized PCA still can not eliminate the curse of dimensionality on persistence diagrams completely, meaning the use of normalized PCA can not enhance the asymptotic similarity level to the bottleneck consistency level.
    \item Analyzing persistence diagrams after applying classical PCA can be challenging due to the presence of both eigenvalues and eigenvectors in~\eqref{sqdisthat} for calculating the squared distance between the projected points.
    \item In Section~\ref{sec4}, we have assumed that the noise vectors follow the standard Gaussian distribution. This assumption is made because the sample covariance matrix of the noise point cloud $E$ has the Wishart distribution, which has been well understood in the literature. Replacing the standard Gaussian assumption with other continuous distributions could be challenging.
    \item It is important to note that addressing the curse of dimensionality while discussing the \v{C}ech filtration case would involve the scalar product of projected points, which can lead to complicated computation. For this reason, we have confined our analysis to the Rips filtration in Section~\ref{sec4}.
    \item Our work treats the setting where $d$ tends to infinity but $n$ is fixed in the high-dimensional statistics framework. This framework also involves the setting where both $d$ and $n = n(d)$ diverge with the ratio $d/n\rightarrow \infty$. Since the sample size $n$ varies with $d$, discussions in this scenario become even more complex.
    \item We hope that future research can develop a new method of dimensionality reduction and rigorously show its ability to eliminate the curse of dimensionality on persistence diagrams in the asymptotic similarity sense. One possible candidate is to combine the nearest neighborhood information of each point in the HDLSS data, suggested by \cite{damrich2023persistent}. They have proposed using graph-based distances like effective resistance or diffusion distance in persistent homology for this purpose.
\end{enumerate}

\backmatter

\bmhead*{Acknowledgments}
The authors would like to thank Prof. Beno\^{i}t Collins for his helpful suggestions on proving the eigengap issue of the real Wishart matrix and his guidance on the Weingarten calculus. We would also like to thank Drs. Kazuyoshi Yata, Th\'{e}o Lacombe and Zeming Sun for the fruitful discussions.

\bmhead*{Funding}
This work was supported by JST SPRING, Grant Number JPMJSP2110, and JST PRESTO, Grant Number JPMJPR2021. We also acknowledge the support received from JST, the establishment of university fellowships towards the creation of science technology innovation, Grant Number JPMJFS2123.

\begin{appendices}

\section{Asymptotically infinite divergence of the minimum eigengap of real Wishart matrices}\label{secA4}
Studying gaps between consecutive eigenvalues is one of the central topics in random matrix theory. In literature, many works on the limiting distribution of the eigenvalue gap have been done so far. For example, the standard work on the limiting gap distribution for Gaussian matrices (GUE and GOE) is \cite{MR0220494}, and the work on the limiting gap distribution for Wigner matrices and adjacency matrices of random graphs is \cite{MR3627428}. Unfortunately, there is currently less work on the limiting gap distribution for the sample covariance matrices, particularly for real sample covariance matrices. In this section, we estimate the minimum eigenvalue gap of real Wishart matrices to complete the proof omitted in Section~\ref{subsec6}.

Let $v_{1}, v_{2}, \ldots, v_{d}\in \mathbb{R}^{n}$ ($ n\in \mathbb{N}_{+}$) be \iid\ random vectors following the standard Gaussian distribution $\mathcal{N}_{n}(0, I_{n})$. Set $W=\sum_{k\in [d]}v_{k}^{}v_{k}^{T}$. In random matrix theory, $W$ is said to follow the (standard) real Wishart distribution, denoted by $W\sim \mathcal{W}_{n}(I_{n},d)$. Clearly, $W$ has $n$ non-negative eigenvalues because $W$ is an $n\times n$ positive semi-definite symmetric real matrix. We denote all eigenvalues of $W$ in order by $\lambda_{1}\geq \lambda_{2}\geq \cdots \geq \lambda_{n}$. Spacings between consecutive eigenvalues $\lambda_{i}-\lambda_{i+1}$ are called eigengaps later. We will then consider the limiting behavior of the minimum eigengap $\min_{k\in [n-1]}\{\lambda_{k}-\lambda_{k+1}\}$ as $d$ tends to infinity but $n$ is fixed. 

\subsection{Preliminary}\label{secA41}
In this subsection, we introduce some fundamental knowledge and properties that will be used in the main proof. We begin by providing rough estimates for all eigenvalues of the standard Wishart distributed matrix and indicating rough estimates for all eigengaps.

\begin{proposition}\label{orderofeigenvalue}
Let $W\sim \mathcal{W}(I_{n},d)$, and $\lambda_{k}$ the $k$-th largest eigenvalue of $W$ ($k\in [n]$). Then
\begin{equation}
    \lambda_{k}=d+O_{\mathbb{P}}(d^{1/2})
\end{equation}
holds as $d\rightarrow \infty$ for every $k\in [n]$. Accordingly, 
\begin{align}
    \lambda_{k}-\lambda_{k+1} = O_{\mathbb{P}}(d^{1/2})\label{roughestimateofgaps}
\end{align}
holds as d $\rightarrow \infty$ for every $k\in [n-1]$.
\end{proposition}

\begin{proof}[Proof.\nopunct]
We set $W=\sum_{k\in [d]}v_{k}^{}v_{k}^{T}$, where $v_{k} = (v_{1k}, \ldots, v_{nk})^{T}$. Then $W=\left(\sum_{k\in [d]}v_{jk}v_{lk}\right)_{j,l\in[n]}$. For the convenience we denote the $(j,l)$-entry $\sum_{k\in [d]}v_{jk}v_{lk}$ of $W$ by $w_{jl}$. The first moment of $w_{jl}$ can be computed as follows.
\begin{align*}
    \mathbb{E}w_{jl} =  \sum_{k\in [d]}\mathbb{E}v_{jk}v_{lk}=\begin{dcases}
    \sum_{k\in [d]}\mathbb{E}v_{jk}^{2}=\sum_{k\in [d]}\mathbb{V}(v_{jk})=d,\  j=l,\\
    \sum_{k\in [d]}\mathbb{E}v_{jk}\mathbb{E}v_{lk}=0,\ j\neq l.
\end{dcases}
\end{align*}
The second moment of $w_{jl}$ can be computed as follows.
\begin{align*}
    \mathbb{V}(w_{jl}) =  \sum_{k\in [d]}\mathbb{V}(v_{jk}v_{lk})  =\begin{dcases}
    \sum_{k\in [d]}\mathbb{E}v_{jk}^{4}-(\mathbb{E}v_{jk}^{2})^{2}=2d,\  j=l,\\
    \sum_{k\in [d]}\mathbb{E}v_{jk}^{2}\mathbb{E}v_{lk}^{2}=d,\ j\neq l.
\end{dcases}
\end{align*}
It then follows that
\begin{align*}
    w_{jl} =\begin{dcases}
    d+O_{\mathbb{P}}(d^{1/2}),\  j=l,\\
    O_{\mathbb{P}}(d^{1/2}),\ j\neq l,
    \end{dcases}
\end{align*}
by Definition~\ref{def0b} and Chebyshev's inequality. Set $w_{jj}'=w_{jj}-d$ whenever $j=l$. Then $w_{jj}'=O_{\mathbb{P}}(d^{1/2})$. Let $u_{k}$ be the unit eigenvector of $W$ corresponding to $\lambda_{k}$. Then for each $k\in [n]$, we have
\begin{align*}
    \lvert\lambda_{k}-d\rvert & = \lvert u_{k}^{T}Wu_{k}^{}-d\rvert   = \Bigg\lvert\sum_{j\in [n]}(w_{jj}-d)u_{jk}^{2}+\sum_{\substack{j,l\in [n] \\ j\neq l}}w_{jl}u_{jk}u_{lk}\Bigg\rvert\\
    & = \Bigg\lvert\sum_{j\in [n]}w_{jj}'u_{jk}^{2}+\sum_{\substack{j,l\in [n] \\ j\neq l}}w_{jl}u_{jk}u_{lk}\Bigg\rvert \leq \sum_{j\in [n]}\lvert w_{jj}'\rvert+\sum_{\substack{j,l\in [n] \\ j\neq l}}\lvert w_{jl}\rvert.
\end{align*}
By Propositions~\ref{prop0a}, \ref{prop0b}, we have $\lvert\lambda_{k}-d\rvert = O_{\mathbb{P}}(d^{1/2})$. We finally obtain $\lambda_{k}=d+O_{\mathbb{P}}(d^{1/2})$ as $d\rightarrow \infty$ from Remark~\ref{Gtsdc}. \eqref {roughestimateofgaps} holds by the rules of the calculus of asymptotic notations.
\end{proof}

We turn to review the gamma function and the multivariate gamma function. Let $b$ be a positive real number. The convergent improper integral
\begin{equation}
    \Gamma(b) = \int_{0}^{+\infty}x^{b-1}\euler^{-x}dx
\end{equation}
is called the gamma function. Here and subsequently, $\euler$ represents the Euler's number. Let $m\in \mathbb{N}_{+}$, and $b$ a real number. The multivariate gamma function, denoted by $\Gamma_{m}(b)$, is defined to be
\begin{equation}
    \Gamma_{m}(b) = \pi^{m(m-1)/4}\prod_{k\in [m]}\Gamma(b-\frac{1}{2}(k-1)).
\end{equation}

Now, let us mention two important formulas related to the gamma function without proof.

\begin{proposition}[Recurrence formula, \cite{MR0167642}]\label{recurrence formula}
Let $m\in \mathbb{N}$, and $b$ a positive real number. Then
\begin{equation}
    \Gamma(b)=\frac{\Gamma(b+m+1)}{b(b+1)\cdots(b+m)}.
\end{equation}
\end{proposition}

\begin{proposition}[Legendre duplication formula, \cite{MR0167642}]\label{Legendre duplication formula}
Let $b$ be a positive real number. Then
\begin{equation}
    \Gamma(b)\Gamma(b+\frac{1}{2})=2^{1-2b}\sqrt{\pi}\Gamma(2b).
\end{equation}
\end{proposition}

The joint probability distribution of all eigenvalues of the standard Wishart distributed matrix is well studied. For thorough treatment we refer the reader to \cite{anderson1962introduction}.

\begin{lemma}[\cite{anderson1962introduction}]
Let $W\sim \mathcal{W}(I_{n},d)$, and $\lambda_{1}\geq \lambda_{2}\geq \cdots \geq \lambda_{n}$ the ordered eigenvalues of $W$. Then the joint probability density function of the ordered eigenvalues of $W$, denoted by $f(\lambda_{1},\ldots,\lambda_{n})$, is given by
\begin{equation}
    f(\lambda_{1},\ldots,\lambda_{n}) = \dfrac{\pi^{\frac{1}{2}n^{2}}\euler^{-\frac{1}{2}\sum\limits_{k\in [n]}\lambda_{k}}\cdot \prod\limits_{k\in [n]}\lambda_{k}^{\frac{d-n-1}{2}}\cdot \prod\limits_{\substack{k,j\in [n] \\ k<j}}(\lambda_{k}-\lambda_{j})}{2^{\frac{1}{2}nd}\Gamma_{n}(\frac{d}{2})\Gamma_{n}(\frac{n}{2})}.\label{jointdensityfunction}
\end{equation}
\end{lemma}

\subsection{Limiting behavior of the minimum eigengap}\label{secA42}
With all the necessary preliminary knowledge, we are now in a position to show the main theorem of this appendix.

\begin{theorem}\label{maintheorem02}
    Let $W\sim \mathcal{W}_{n}(I_{n},d)$ ($n\geq 2$), and $\lambda_{1}\geq \lambda_{2}\geq \cdots \geq \lambda_{n}$ the ordered eigenvalues of $W$. Then the following two statements hold true.
\begin{itemize}
    \item For every fixed even $n$,
    \begin{equation*}
    \mathbb{P}\left(d^{\frac{1}{2}-\gamma}\leq \min_{k\in [n-1]}\{\lambda_{k}-\lambda_{k+1}\}\leq d^{\frac{1}{2}+\eta}\right)\rightarrow 1
    \end{equation*}
    holds as $d\rightarrow \infty$ for arbitrary $\eta>0$ and $\gamma>\frac{n^2+n-6}{4}\eta$.

\item For every fixed odd $n$,
    \begin{equation*}
    \mathbb{P}\left(d^{\frac{1}{2}-\gamma}\leq \min_{k\in [n-1]}\{\lambda_{k}-\lambda_{k+1}\}\leq d^{\frac{1}{2}+\eta}\right)\rightarrow 1
    \end{equation*}
    holds as $d\rightarrow \infty$ for arbitrary $\eta, \gamma>0$ with $\gamma>\frac{n^2+n-6}{4}\eta-\frac{n^2+n-2}{8}$.
\end{itemize}
\end{theorem}

\begin{proof}[Proof.\nopunct]
We first observe that for every $k\in [n-1]$ and any $\eta>0$,
\begin{equation*}
    \frac{\lambda_{k}-\lambda_{k+1}}{d^{\frac{1}{2}+\eta}} = O_{\mathbb{P}}(d^{-\eta}) = o_{\mathbb{P}}(1)
\end{equation*}
holds as $d\rightarrow \infty$ from Propositions~\ref{prop0a}, \ref{orderofeigenvalue}. By definition,
\begin{equation*}
    \mathbb{P}\left(\lambda_{k}-\lambda_{k+1}\leq d^{\frac{1}{2}+\eta}\right)\rightarrow 1
\end{equation*}
holds as $d\rightarrow \infty$. Fix $i\in [n-1]$. Now it is sufficient to show that
\begin{equation*}
    \mathbb{P}\left(\lambda_{i}-\lambda_{i+1}\leq d^{\frac{1}{2}-\gamma},\ \max_{\substack{j\in [n-1] \\ j\neq i}}\{\lambda_{j}-\lambda_{j+1}\}\leq d^{\frac{1}{2}+\eta}\right)\rightarrow 0
\end{equation*}
as $d\rightarrow \infty$, then $\mathbb{P}\left(d^{\frac{1}{2}-\gamma}\leq \lambda_{i}-\lambda_{i+1}\right)\rightarrow 1$ immediately follows.

Let
\begin{align*}
    \scalebox{0.94}{$\Delta = \setc*{(\lambda_{1},\ldots,\lambda_{n})\in\mathbb{R}^{n}}{0\leq \lambda_{i}-\lambda_{i+1}\leq d^{\frac{1}{2}-\gamma},0\leq \lambda_{j}-\lambda_{j+1}\leq d^{\frac{1}{2}+\eta} \left(j\in [n]\setminus\{i\}\right)}.$}
\end{align*}
Noticing the joint probability density function of eigenvalues in~\eqref{jointdensityfunction}, we have
{\small \begin{align}
    &\ \ \ \  \mathbb{P}\left(\lambda_{i}-\lambda_{i+1}\leq d^{\frac{1}{2}-\gamma},\ \max_{\substack{j\in [n-1] \\ j\neq i}}\{\lambda_{j}-\lambda_{j+1}\}\leq d^{\frac{1}{2}+\eta}\right)\nonumber\\
    &= \frac{\pi^{\frac{1}{2}n^{2}}}{2^{\frac{1}{2}nd}\Gamma_{n}(\frac{d}{2})\Gamma_{n}(\frac{n}{2})}\int_{\Delta} \euler^{-\frac{1}{2}\sum\limits_{k\in [n]}\lambda_{k}}\prod_{k\in [n]}\lambda_{k}^{\frac{d-n-1}{2}}\prod_{\substack{k,j\in [n] \\ k<j}}(\lambda_{k}-\lambda_{j}) \textrm{d}\lambda_{1} \cdots \textrm{d}\lambda_{n}\nonumber\\
    & \leq \frac{\pi^{\frac{1}{2}n^{2}}}{2^{\frac{1}{2}nd}\Gamma_{n}(\frac{d}{2})\Gamma_{n}(\frac{n}{2})}\int_{\Delta} \euler^{-\frac{1}{2}\sum\limits_{k\in [n]}\lambda_{k}}\left(\frac{1}{n}\sum_{k\in [n]}\lambda_{k}\right)^{\frac{n(d-n-1)}{2}}\hspace{-15pt}\prod_{\substack{k,j\in [n] \\ k<j}}(\lambda_{k}-\lambda_{j}) \textrm{d}\lambda_{1} \cdots \textrm{d}\lambda_{n}. \label{Yjkxfjk}
\end{align}}
The last inequality holds due to the inequality of arithmetic and geometric means. We change variables in \eqref{Yjkxfjk} by letting $s = \sum_{k\in [n]}\lambda_{k}$ and $t_{j} = \lambda_{j}-\lambda_{j+1}$ for every $j\in [n-1]$. Then \eqref{Yjkxfjk} becomes
\begin{align*}
     &\quad\ \mathbb{P}\left(\lambda_{i}-\lambda_{i+1}\leq d^{\frac{1}{2}-\gamma},\ \max_{\substack{j\in [n-1] \\ j\neq i}}\{\lambda_{j}-\lambda_{j+1}\}\leq d^{\frac{1}{2}+\eta}\right)\\
     &\leq \frac{\pi^{\frac{1}{2}n^{2}}\cdot \lvert \textrm{det} J\rvert }{n^{\frac{n(d-n-1)}{2}}2^{\frac{1}{2}nd}\Gamma_{n}(\frac{d}{2})\Gamma_{n}(\frac{n}{2})}\int_{\Delta'} \euler^{-\frac{s}{2}}s^{\frac{n(d-n-1)}{2}}p(t_{1},\ldots,t_{n-1}) \textrm{d}s\textrm{d}t_{1} \cdots \textrm{d}t_{n-1}.
\end{align*}
Here $J=\frac{\partial(\lambda_{1},\lambda_{2},\ldots,\lambda_{n})}{\partial(s,t_{1},\ldots,t_{n-1})}$ is the Jacobian matrix of the change of variables, $p(t_{1},\ldots,t_{n-1})$ denotes the polynomial which is the finite sum of monomials $t_{1}^{\alpha_{1}}\cdots t_{n-1}^{\alpha_{n-1}}$ ($\alpha_{1},\ldots, \alpha_{n-1}\geq 1$, $\alpha_{1}+\cdots+\alpha_{n-1}=\binom{n}{2}$), and the integral domain $\Delta$ becomes 
\begin{equation*}
    \scalebox{0.88}{$\Delta' = \setc*{(s,t_{1},\ldots,t_{n-1})\in\mathbb{R}^{n}}{0\leq t_{i}\leq d^{\frac{1}{2}-\gamma},\ 0\leq t_{j}\leq d^{\frac{1}{2}+\eta} \left(j\in [n]\setminus\{i\}\right),s\geq \sum\limits_{k\in [n-1]}kt_{k}}$}.
\end{equation*}
In order to show the probability tends to $0$, we are left with the task of showing
\begin{equation*}
    \frac{\pi^{\frac{1}{2}n^{2}}\cdot \lvert \textrm{det} J\rvert }{n^{\frac{n(d-n-1)}{2}}2^{\frac{1}{2}nd}\Gamma_{n}(\frac{d}{2})\Gamma_{n}(\frac{n}{2})}\int_{\Delta'} \euler^{-\frac{s}{2}}s^{\frac{n(d-n-1)}{2}}t_{1}^{\alpha_{1}}\cdots t_{n-1}^{\alpha_{n-1}} \textrm{d}s\textrm{d}t_{1} \cdots \textrm{d}t_{n-1}\rightarrow 0
\end{equation*}
as $d\rightarrow \infty$ for every summand $t_{1}^{\alpha_{1}}\cdots t_{n-1}^{\alpha_{n-1}}$ in $p(t_{1},\ldots,t_{n-1})$.

By Tonelli's theorem, we have
{\small \begin{align*}
    &\quad\ \frac{\pi^{\frac{1}{2}n^{2}}\lvert \textrm{det} J\rvert }{n^{\frac{n(d-n-1)}{2}}2^{\frac{1}{2}nd}\Gamma_{n}(\frac{d}{2})\Gamma_{n}(\frac{n}{2})}\int_{\Delta'} \euler^{-\frac{s}{2}}s^{\frac{n(d-n-1)}{2}}t_{1}^{\alpha_{1}}\cdots t_{n-1}^{\alpha_{n-1}} \textrm{d}s\textrm{d}t_{1} \cdots \textrm{d}t_{n-1} \notag \\
    & \leq \frac{\pi^{\frac{1}{2}n^{2}}\lvert \textrm{det} J\rvert }{n^{\frac{n(d-n-1)}{2}}2^{\frac{1}{2}nd}\Gamma_{n}(\frac{d}{2})\Gamma_{n}(\frac{n}{2})}\int_{0}^{d^{\frac{1}{2}-\gamma}}\hspace{-15pt}t_{i}^{\alpha_{i}}\textrm{d}t_{i}\left(\prod_{\substack{j\in [n-1] \\ j\neq i}}\hspace{-1pt}\int_{0}^{d^{\frac{1}{2}+\eta}}\hspace{-15pt}t_{j}^{\alpha_{j}}\textrm{d}t_{j}\right)\int_{0}^{+\infty}\hspace{-10pt} \euler^{-\frac{s}{2}}s^{\frac{n(d-n-1)}{2}}\textrm{d}s\notag\\
    & = \frac{\pi^{\frac{1}{2}n^{2}} \lvert \textrm{det} J\rvert \cdot 2^{\frac{n(d-n-1)}{2}+1}\cdot \Gamma(\frac{n(d-n-1)}{2}+1)}{n^{\frac{n(d-n-1)}{2}}2^{\frac{1}{2}nd}\Gamma_{n}(\frac{d}{2})\Gamma_{n}(\frac{n}{2})}\cdot \frac{(d^{\frac{1}{2}-\gamma})^{\alpha_{i}+1}\cdot (d^{\frac{1}{2}+\eta})^{\sum_{j\neq i}(\alpha_{j}+1)}}{\prod\limits_{k\in [n-1]}(\alpha_{k}+1)}\notag\\
    & = \frac{\pi^{\frac{1}{2}n^{2}}\lvert \textrm{det} J\rvert d^{(\frac{1}{2}+\eta)\sum_{j\neq i}(\alpha_{j}+1)+(\frac{1}{2}-\gamma)(\alpha_{i}+1)} \Gamma(\frac{n(d-n-1)}{2}+1)}{2^{\frac{(n-1)(n+2)}{2}}n^{\frac{n(d-n-1)}{2}}\Gamma_{n}(\frac{n}{2})\Gamma_{n}(\frac{d}{2})\prod\limits_{k\in [n-1]}(\alpha_{k}+1)}\notag\\
    & \leq \frac{\pi^{\frac{1}{2}n^{2}}\lvert \textrm{det} J\rvert }{2^{\frac{(n-1)(n+2)}{2}}\Gamma_{n}(\frac{n}{2})\prod\limits_{k\in [n-1]}(\alpha_{k}+1)}\cdot \frac{d^{(\frac{1}{2}+\eta)\cdot \frac{n^2+n-2}{2}+2\left(-\gamma-\eta\right)}\Gamma(\frac{n(d-n-1)}{2}+1)}{n^{\frac{n(d-n-1)}{2}}\Gamma_{n}(\frac{d}{2})}.
\end{align*}
}
The last inequality holds because
\begin{align*}
    &\ \ \ \ (\frac{1}{2}+\eta)\sum_{\substack{j\in [n-1] \\ j\neq i}}(\alpha_{j}+1)+(\frac{1}{2}-\gamma)(\alpha_{i}+1)\\
    & = (\frac{1}{2}+\eta)\sum_{j\in [n-1]}(\alpha_{j}+1)+\left(-\gamma-\eta\right)(\alpha_{i}+1)\\
    & \leq (\frac{1}{2}+\eta)\cdot \frac{n^2+n-2}{2}+2\cdot \left(-\gamma-\eta\right).
\end{align*}
What is left is to show that
\begin{equation*}
    \frac{d^{(\frac{1}{2}+\eta)\cdot \frac{n^2+n-2}{2}+2\left(-\gamma-\eta\right)}\Gamma(\frac{n(d-n-1)}{2}+1)}{n^{\frac{n(d-n-1)}{2}}\Gamma_{n}(\frac{d}{2})}\rightarrow 0
\end{equation*}
as $d\rightarrow \infty$. For simplicity, we write $\beta = (\frac{1}{2}+\eta)\cdot \frac{n^2+n-2}{2}+2\cdot \left(-\gamma-\eta\right)$ and
\begin{equation*}
    a_{d} = \frac{d^{\beta}\Gamma(\frac{n(d-n-1)}{2}+1)}{n^{\frac{n(d-n-1)}{2}}\Gamma_{n}(\frac{d}{2})}.
\end{equation*}
We proceed to show $\lim_{d\rightarrow \infty}a_{d}=0$ by cases.

For even $n$, we evaluate $a_{d}/a_{d-1}$. By Proposition~\ref{recurrence formula}, we have
\begin{align*}
    \frac{a_{d}}{a_{d-1}} & = \left(\frac{d}{d-1}\right)^{\beta}\cdot \frac{\Gamma(\frac{n(d-n-1)}{2}+1)}{\Gamma(\frac{n(d-n-2)}{2}+1)}\cdot \frac{n^{\frac{n(d-n-2)}{2}}}{n^{\frac{n(d-n-1)}{2}}}\cdot \frac{\Gamma_{n}(\frac{d-1}{2})}{\Gamma_{n}(\frac{d}{2})}\\
    & = \left(\frac{d}{d-1}\right)^{\beta}\cdot\prod_{k=0}^{\frac{n}{2}-1}\left(\frac{n(d-n-2)}{2}+1+k\right)\cdot n^{-\frac{n}{2}}\cdot \frac{\prod\limits_{k\in [n]}\Gamma(\frac{d-1}{2}-\frac{1}{2}(k-1))}{\prod\limits_{k\in [n]}\Gamma(\frac{d}{2}-\frac{1}{2}(k-1))}\\
    & = \left(\frac{d}{d-1}\right)^{\beta}\cdot\frac{\prod\limits_{k=0}^{\frac{n}{2}-1}\left(\frac{n(d-n-2)}{2}+1+k\right)}{n^{\frac{n}{2}}}\cdot \frac{\Gamma(\frac{d}{2}-\frac{n}{2})}{\Gamma(\frac{d}{2})}\\
    & = \left(\frac{d}{d-1}\right)^{\beta}\cdot\frac{\prod\limits_{k=0}^{\frac{n}{2}-1}\left(\frac{n(d-n-2)}{2}+1+k\right)}{\prod\limits_{k=0}^{\frac{n}{2}-1} \left(\frac{nd}{2}-\frac{n^2}{2}+nk\right)}.
\end{align*}
For each $0\leq k\leq \frac{n}{2}-1$, we write $A_{k}=\frac{nd}{2}-\frac{n^2}{2}+nk$ and $B_{k}=\frac{n(d-n-2)}{2}+1+k$.  Notice that $B_{0}<\cdots<B_{\frac{n}{2}-1}<A_{0}<\cdots<A_{\frac{n}{2}-1}$. Taking the logarithm of $a_{d}/a_{d-1}$ yields
\begin{align}
    \ln{\frac{a_{d}}{a_{d-1}}} & = \beta\ln{\left(1+\frac{1}{d-1}\right)}+\sum_{k=0}^{\frac{n}{2}-1}\ln{\left(1-\frac{A_{k}-B_{k}}{A_{k}}\right)}\notag\\
    & \leq \beta\cdot \frac{1}{d-1}-\sum_{k=0}^{\frac{n}{2}-1}\frac{A_{k}-B_{k}}{A_{k}}\label{firstineq3.2}\\
    & \leq \beta\cdot\frac{1}{d-1}-\sum_{k=0}^{\frac{n}{2}-1}\frac{A_{k}-B_{k}}{A_{\frac{n}{2}-1}}\notag\\
    & = \beta\cdot\frac{1}{d-1}-\frac{(n-1)\cdot(1+\frac{n}{2})\cdot \frac{n}{2}\cdot \frac{1}{2}}{\frac{n}{2}}\frac{1}{d-2}\notag\\
    & \leq  \left[\beta-\frac{(n-1)(n+2)}{4}\right]\frac{1}{d-2}.\label{finalineq3.3}
\end{align}
The first inequality~\eqref{firstineq3.2} comes from the inequality: $\ln{(1+x)}\leq x$ whenever $x\geq -1$. Let us denote the coefficient
\begin{equation*}
    \beta-\frac{(n-1)(n+2)}{4}
\end{equation*}
in \eqref{finalineq3.3} by $\Theta$. We wish $\Theta<0$, i.e.,
\begin{equation*}
    (\frac{1}{2}+\eta)\cdot \frac{n^2+n-2}{2}+2\cdot \left(-\gamma-\eta\right)-\frac{(n-1)(n+2)}{4}<0.
\end{equation*}
Consequently, we obtain
\begin{align*}
    \gamma>\frac{n^2+n-6}{4}\eta.
\end{align*}
It then follows that
\begin{align}
    a_{d} & = \euler^{\ln{\frac{a_{d}}{a_{d-1}}\cdot \frac{a_{d-1}}{a_{d-2}}\cdot\ \cdots\ \cdot \frac{a_{n+2}}{a_{n+1}}\cdot a_{n+1}}}\notag\\
    & = \euler^{\ln{\frac{a_{d}}{a_{d-1}}}+\ln{\frac{a_{d-1}}{a_{d-2}}}+\cdots+\ln{ \frac{a_{n+2}}{a_{n+1}}}+\ln{a_{n+1}}}\notag\\
    &\leq a_{n+1}\euler^{\Theta\cdot (\frac{1}{d-2}+\frac{1}{d-3}+\cdots+\frac{1}{n})}.\label{harmonicseries}
\end{align}
Hence $\lim_{d\rightarrow \infty}a_{d}=0$. Therefore, $\lim_{d\rightarrow \infty}a_{d}=0$ holds for even $n$.

For odd $n$, we evaluate $a_{d}/a_{d-2}$. By Propositions~\ref{recurrence formula}, \ref{Legendre duplication formula}, we have
\begin{align*}
    \frac{a_{d}}{a_{d-2}} & = \left(\frac{d}{d-2}\right)^{\beta}\cdot \frac{\Gamma(\frac{n(d-n-1)}{2}+1)}{\Gamma(\frac{n(d-n-3)}{2}+1)}\cdot \frac{n^{\frac{n(d-n-3)}{2}}}{n^{\frac{n(d-n-1)}{2}}}\cdot \frac{\Gamma_{n}(\frac{d-2}{2})}{\Gamma_{n}(\frac{d}{2})}\\
    & = \left(\frac{d}{d-2}\right)^{\beta}\cdot\prod\limits_{k=0}^{n-1}\left(\frac{n(d-n-3)}{2}+1+k\right)\cdot n^{-n}\cdot \frac{\prod\limits_{k\in [n]}\Gamma(\frac{d-2}{2}-\frac{1}{2}(k-1))}{\prod\limits_{k\in [n]}\Gamma(\frac{d}{2}-\frac{1}{2}(k-1))}\\
    & = \left(\frac{d}{d-2}\right)^{\beta}\cdot\frac{\prod\limits_{k=0}^{n-1}\left(\frac{n(d-n-3)}{2}+1+k\right)}{n^{n}}\cdot \frac{2^{(n+2)-d}\cdot \Gamma(d-(n+1))}{2^{2-d}\cdot \Gamma(d-1)}\\
    & = \left(\frac{d}{d-2}\right)^{\beta}\cdot\frac{\prod\limits_{k=0}^{n-1}\left(\frac{n(d-n-3)}{2}+1+k\right)}{\prod\limits_{k=0}^{n-1} \left(\frac{nd}{2}-\frac{n^2}{2}+\frac{n(k-1)}{2}\right)}.
\end{align*}
For each $0\leq k\leq n-1$, we write $A_{k}=\frac{nd}{2}-\frac{n^2}{2}+\frac{n(k-1)}{2}$ and $B_{k}=\frac{n(d-n-3)}{2}+1+k$.  Notice that $B_{0}<\cdots<B_{n-1}=A_{0}<\cdots<A_{n-1}$. Taking the logarithm of $a_{d}/a_{d-2}$ yields
\begin{align}
    \ln{\frac{a_{d}}{a_{d-2}}} & = \beta\ln{\left(1+\frac{1}{d-2}\right)}+\sum_{k=0}^{n-1}\ln{\left(1-\frac{A_{k}-B_{k}}{A_{k}}\right)}\notag\\
    & \leq \beta\cdot\frac{1}{d-2}-\sum_{k=0}^{n-1}\frac{A_{k}-B_{k}}{A_{k}}\notag\\
    & \leq \beta\cdot\frac{1}{d-2}-\sum_{k=0}^{n-1}\frac{A_{k}-B_{k}}{A_{n-1}}\notag\\
    & = \left[\beta-\frac{(n-1)(n+2)}{2}\right]\frac{1}{d-2}.\label{finalineq3.6}
\end{align}
Let us denote the coefficient
\begin{equation*}
    \beta-\frac{(n-1)(n+2)}{2}
\end{equation*}
in \eqref{finalineq3.6} by $\Theta'$. We wish $\Theta'<0$, i.e.,
\begin{equation*}
    (\frac{1}{2}+\eta)\cdot \frac{n^2+n-2}{2}+2\cdot \left(-\gamma-\eta\right)-\frac{(n-1)(n+2)}{2}<0.
\end{equation*}
Consequently, we obtain
\begin{align*}
    \gamma>\frac{n^2+n-6}{4}\eta-\frac{n^2+n-2}{8}.
\end{align*}
Analysis similar to that in the even $n$ case shows that $\lim_{d\rightarrow \infty}a_{d}=0$. Consequently, $\lim_{d\rightarrow\infty}a_{d}=0$ holds for odd $n$, and the theorem follows.
\end{proof}

The following theorem is a straightforward consequence of Theorem~\ref{maintheorem02}, indicating that the minimum eigengap diverges to infinity with high probability.

\begin{theorem}[Asymptotically infinite divergence of eigengaps]\label{maintheorem01}
\microtypesetup{protrusion=true,expansion=true}
Let $W\sim \mathcal{W}_{n}(I_{n},d)$ ($n\geq 2$), and $\lambda_{1}\geq \lambda_{2}\geq \cdots \geq \lambda_{n}$ the ordered eigenvalues of $W$. Then for every constant $C>0$,
\begin{equation*}
\mathbb{P}\left(\min_{k\in [n-1]}\{\lambda_{k}-\lambda_{k+1}\}>C\right)\rightarrow 1
\end{equation*}
holds as $d\rightarrow \infty$ but $n$ is fixed.
\end{theorem}

\microtypesetup{protrusion=false,expansion=false}

To validate our results, we conduct the numerical experiment of computing the minimum eigengap of the real Wishart matrix $W\sim \mathcal{W}(I_{n},d)$ with $d$ ranges in $[150,2000]$ and $n=100$. At each $d$, we repeat the process of randomly generating a real Wishart matrix $20{,}000$ times and record the minimum eigengap for each generated matrix, and then take the average value. The final result is shown in Fig.~\ref{fig:Eigengap}, which confirms that the minimum eigengap has a high probability of tending to infinity as the degree of freedom $d$ increases. Fig.~\ref{fig:Eigengap_curve_fitting} is the curve fitting result. The red curve in Fig.~\ref{fig:Eigengap_curve_fitting} has the approximate expression
\begin{equation*}
    0.0353324372\sqrt{d-100.690024}-0.000162906231.
\end{equation*}
This indicates that the minimum eigengap of a real Wishart matrix has a high probability of increasing with the order around $\sqrt{d}$.

\renewcommand{\thefigure}{\arabic{figure}}
\setcounter{figure}{2}

\begin{figure}[ht]
\centering
\includegraphics[width=1\textwidth]{"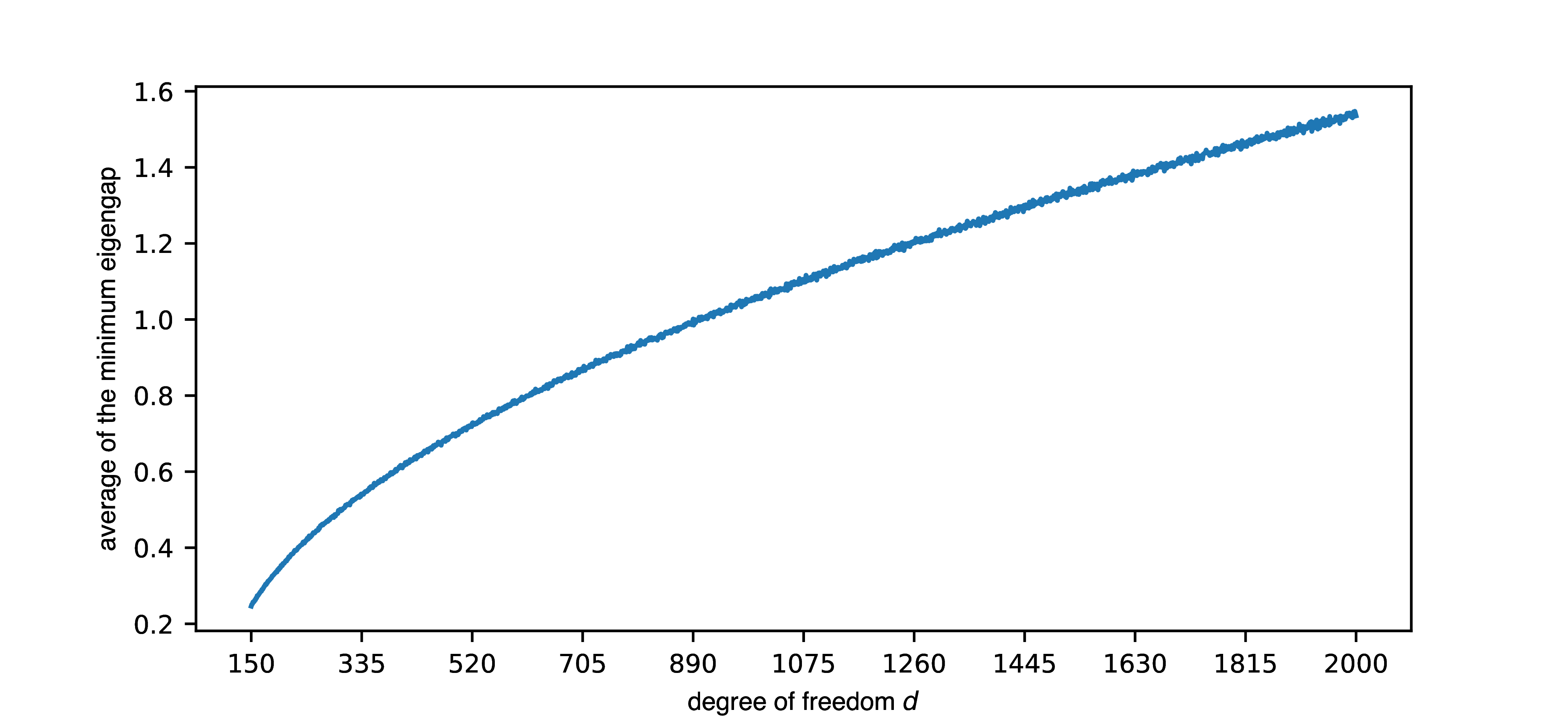"}
\caption{The minimum eigengap of the real Wishart matrix $W\sim \mathcal{W}(I_{100},d)$ with $d$ ranges in $[150,2000]$.}\label{fig:Eigengap}
\end{figure}

\begin{figure}[ht]
\centering
\includegraphics[width=1\textwidth]{"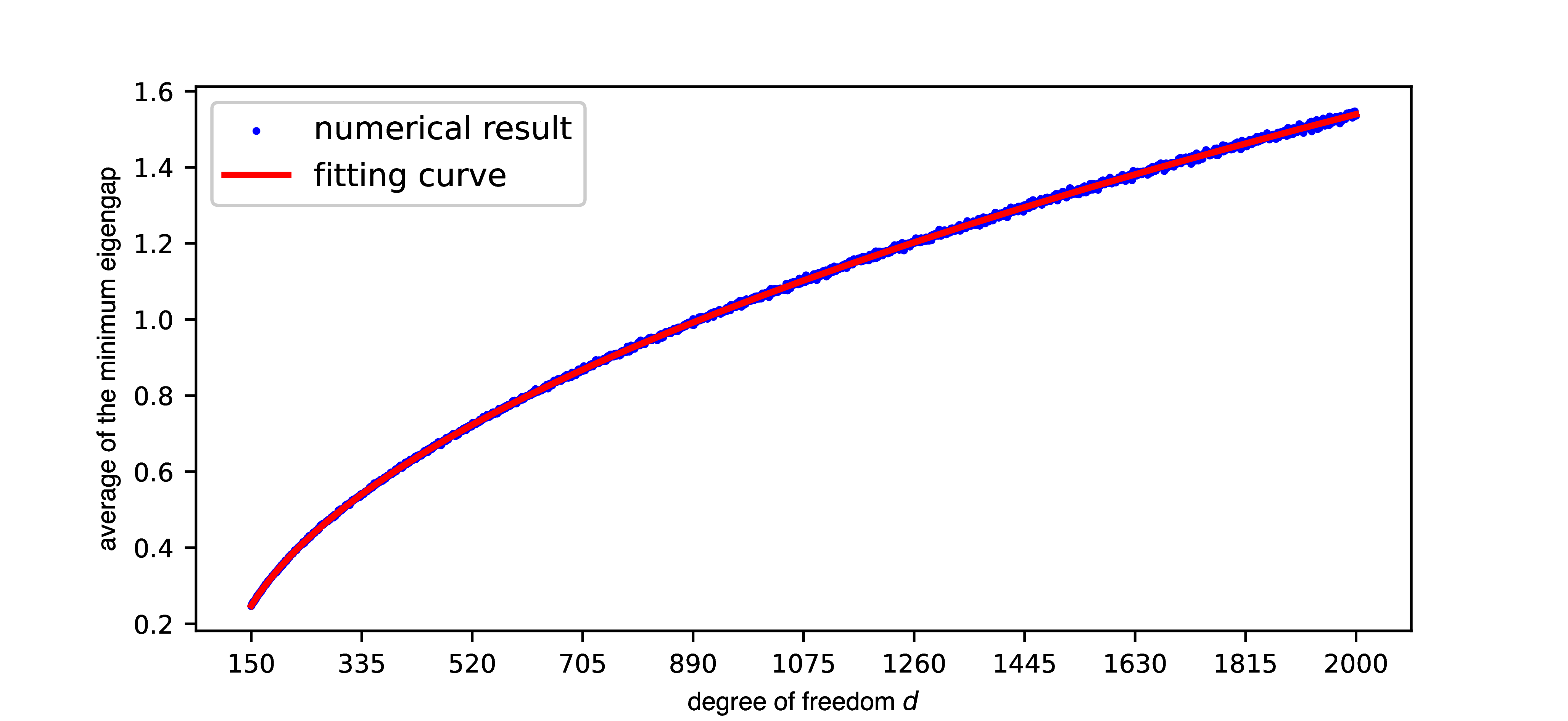"}
\caption{The curve fitting result. The blue dots represent the minimum eigengap averaged over 20{,}000 randomly generated Wishart matrices and the red curve is the fitting curve obtained by the fitting model $x\sqrt{d-y}+z$. Here $x= 0.0353324372$, $y=100.690024$, and $z=-0.000162906231$, respectively.}\label{fig:Eigengap_curve_fitting}
\end{figure}

\section{Weingarten calculus}\label{secA5}
In this section, we first introduce the general setup of Weingarten calculus and some related results. Next, we will restrict our attention to the orthogonal groups used in the main text. Finally, we will provide a complete proof of Proposition~\ref{1st and 2nd moments of pairwise distance}. For a fuller discussion of Weingarten calculus, we refer the reader to the survey~\cite{MR4415894}. Throughout this section, we will denote by $\mathbb{C}$ the set of all complex numbers, $\overline{z}$ the complex conjugate of a given complex number $z$, and $A^{*}$ the conjugate transpose of a given matrix 
$A$.

\subsection{General setup and results}\label{secA51}

\begin{definition}\label{Weingarten integrals}
    Let $G$ be a compact group, and $\mathcal{H}$ a finite-dimensional $\mathbb{C}$-Hilbert space with the orthonormal basis $\varepsilon_{1}, \ldots, \varepsilon_{N}$. Given a unitary representation $(\mathcal{H},U)$ of $G$, i.e., a continuous group homomorphism $U: G\to \textrm{U}(\mathcal{H})$ from $G$ to the unitary group of $\mathcal{H}$, let $U_{ij}: G\to \mathbb{C}$ be the corresponding matrix element functionals by $U_{ij}(\textsl{g})=\langle \varepsilon_{i}, U(\textsl{g})\varepsilon_{j}\rangle$ ($i,j\in [N]$). The integrals
    \begin{equation*}
        I_{\textbf{i}\textbf{j}} = \int_{G}\prod_{x\in [\textsl{k}]}U_{\textbf{i}(x)\textbf{j}(x)}(\textsl{g})\textrm{d}\textsl{g}
    \end{equation*}
    with respect to the Haar probability measure are called the Weingarten integrals of the unitary representation $(\mathcal{H}, U)$. Here $\mathbf{i}, \mathbf{j}$ are functions from $[\textsl{k}]$ to $[N]$.
\end{definition}

By using the tensor product notation, we can reformulate the Weingarten integrals $I_{\mathbf{i}\mathbf{j}}$. Let $\mathbf{i}, \mathbf{j}$ be two functions from $[\textsl{k}]$ to $[N]$. They can also be written as the multi-index forms: $\mathbf{i}=(\mathbf{i}(1),\ldots,\mathbf{i}(\textsl{k}))$, $\mathbf{j}=(\mathbf{j}(1),\ldots,\mathbf{j}(\textsl{k}))$. In what follows, we will use the terminology $(\mathbf{i},\mathbf{j})$-entry of $U^{\otimes \textsl{k}}(\textsl{g})$, denoted by $U_{\mathbf{i}\mathbf{j}}^{\otimes \textsl{k}}(\textsl{g})$, to represent the entry
\begin{align*}
    U_{\mathbf{i}(1)\mathbf{j}(1)}(\textsl{g})\cdots U_{\mathbf{i}(\textsl{k})\mathbf{j}(\textsl{k})}(\textsl{g}) =  \big\langle \varepsilon_{\mathbf{i}}, U^{\otimes \textsl{k}}(\textsl{g})\varepsilon_{\mathbf{j}}\big\rangle
\end{align*}
in the matrix $U^{\otimes \textsl{k}}(\textsl{g})$. Here $\varepsilon_{\mathbf{i}} = \varepsilon_{\mathbf{i}(1)}\otimes \cdots \otimes \varepsilon_{\mathbf{i}(\textsl{k})}$ is the orthonormal basis of $\mathcal{H}^{\otimes \textsl{k}}$. Then $I_{\mathbf{i}\mathbf{j}}$ can be written as
\begin{align*}
    I_{\mathbf{i}\mathbf{j}} = \int_{G}U_{\mathbf{i}\mathbf{j}}^{\otimes \textsl{k}}(\textsl{g})\textrm{d}\textsl{g} = \int_{G}\big\langle \varepsilon_{\mathbf{i}}, U^{\otimes \textsl{k}}(\textsl{g})\varepsilon_{\mathbf{j}}\big\rangle\textrm{d}\textsl{g} = \bigg\langle \varepsilon_{\mathbf{i}}, \int_{G} U^{\otimes \textsl{k}}(\textsl{g})\textrm{d}\textsl{g}\cdot \varepsilon_{\mathbf{j}}\bigg\rangle.
\end{align*}
Denote $\int_{G} U^{\otimes \textsl{k}}(\textsl{g})\textrm{d}\textsl{g}$ by $\mathscr{P}$. It is easily seen that the basic problem of Weingarten calculus is thus equivalent to computing the matrix elements of $\mathscr{P}\in \End \mathcal{H}^{\otimes \textsl{k}}$.

By the invariance of the Haar measure, we see at once that the operator $\mathscr{P}$ is self-adjoint and idempotent. Then $\mathscr{P}$ orthogonally projects $\mathcal{H}^{\otimes \textsl{k}}$ onto its image $\Img \mathscr{P}$. Furthermore, one can show that $\Img \mathscr{P}$ is the space of $G$-invariant tensors in $\mathcal{H}^{\otimes \textsl{k}}$, i.e.,
\begin{equation*}
    \Img \mathscr{P} = \setc*{t\in \mathcal{H}^{\otimes \textsl{k}}}{\forall \textsl{g} \in G, U^{\otimes \textsl{k}}(\textsl{g})t=t}.
\end{equation*}
From this, we see that evaluating the Weingarten integral
is determined by a basis (not necessarily orthonormal) of $\Img \mathscr{P}$.

Assume that $a_{1}, \ldots, a_{m}$ 
is a basis (not necessarily orthonormal) of $\Img \mathscr{P}$. Let
\begin{align}
\mathscr{P}\varepsilon_{\mathbf{j}} = \sum_{t\in [m]}\kappa_{t}a_{t}.\label{werdf}
\end{align}
Then
\begin{align}
    \begin{pmatrix}
\big\langle a_{1},\mathscr{P}\varepsilon_{\mathbf{j}}\big\rangle \\ \vdots \\
\big\langle a_{m},\mathscr{P}\varepsilon_{\mathbf{j}}\big\rangle\\
\end{pmatrix}
= \begin{pmatrix}
\langle a_{1},a_{1}\rangle & \langle a_{1},a_{2} \rangle & \ldots & \langle a_{1},a_{m}\rangle\\ \langle a_{2},a_{1}\rangle & \langle a_{2},a_{2} \rangle & \ldots & \langle a_{2},a_{m}\rangle \\ \vdots & \vdots & & \vdots\\
\langle a_{m},a_{1}\rangle & \langle a_{m},a_{2} \rangle & \ldots & \langle a_{m},a_{m}\rangle\\
\end{pmatrix}
\begin{pmatrix}
\kappa_{1} \\ \vdots \\
\kappa_{m}\\
\end{pmatrix}.\label{asbjkdv1}
\end{align}
Denote the matrix $(\langle a_{x},a_{y}\rangle)_{x,y\in [m]}$ by $\textrm{W}$ in \eqref{asbjkdv1}. $\textrm{W}$ is actually the Gram matrix of $a_{1}, \ldots, a_{m}$. Since $\mathscr{P}$ is self-adjoint and idempotent, $\big\langle a_{t},\mathscr{P}\varepsilon_{\mathbf{j}}\big\rangle = \big\langle \mathscr{P}a_{t},\varepsilon_{\mathbf{j}}\big\rangle = \langle a_{t},\varepsilon_{\mathbf{j}}\rangle = \overline{\langle \varepsilon_{\mathbf{j}},a_{t}\rangle}$ for every $t\in [m]$. If $\textrm{W}$ is invertable, then
\begin{align}
\begin{pmatrix}
\kappa_{1} \\ \vdots \\
\kappa_{m}\\
\end{pmatrix} = \textrm{W}^{-1}
    \begin{pmatrix}
\overline{\langle \varepsilon_{\mathbf{j}},a_{1}\rangle} \\ \vdots \\
\overline{\langle \varepsilon_{\mathbf{j}},a_{m}\rangle}\\
\end{pmatrix}.\label{asbjkdv2}
\end{align}
On the other hand, $a_{t} = \sum_{\mathbf{i}: [\textsl{k}]\to [N]}\langle \varepsilon_{\mathbf{i}},a_{t}\rangle \varepsilon_{\mathbf{i}}$ for every $t\in [m]$. Set $A = (\langle \varepsilon_{\mathbf{l}},a_{t}\rangle)_{t\in [m], \mathbf{l}: [\textsl{k}]\to [N]}$. Obviously, $ \textrm{W} = A^{*}A$. By~\eqref{werdf}, \eqref{asbjkdv2}, we have
\begin{align*}
    I_{\mathbf{i}\mathbf{j}} = \big\langle \varepsilon_{\mathbf{i}}, \mathscr{P} \varepsilon_{\mathbf{j}}\big\rangle = \bigg\langle \varepsilon_{\mathbf{i}}, \sum_{t\in [m]}\kappa_{t}a_{t}\bigg\rangle = \sum_{t\in [m]}\kappa_{t}\big\langle \varepsilon_{\mathbf{i}}, a_{t}\big\rangle = \prescript{}{\mathbf{i}}{A}\cdot \textrm{W}^{-1}\cdot A^{*}_{\mathbf{j}}.
\end{align*}
Here $\prescript{}{\mathbf{i}}{A}$ denotes the $\mathbf{i}$-th row of $A$ and $A^{*}_{\mathbf{j}}$ denotes the $\mathbf{j}$-th column of $A^{*}$. Consequently,
\begin{align*}
     \mathscr{P} = A\textrm{W}^{-1}A^{*} = A(A^{*}A)^{-1} A^{*}.
\end{align*}
Summarizing, we are able to present the following fundamental theorem of the Weingarten calculus.

\begin{theorem}[\cite{MR4415894}]\label{the fundamental theorem of Weingarten calculus}
    Given a unitary representation $(\mathcal{H},U)$ of $G$. Let $\textsl{k}$, $N\in \mathbb{N}_{+}$. Then
\begin{align*}
\int_{G}\prod_{x\in [\textsl{k}]}U_{\mathbf{i}(x)\mathbf{j}(x)}(\textsl{g})\textrm{d}\textsl{g} = \sum\limits_{x,y\in [m]}A^{}_{\mathbf{i}x}(\textrm{W}^{-1})^{}_{xy}A^{*}_{y\mathbf{j}}
\end{align*}
holds for functions $\mathbf{i}, \mathbf{j}: [\textsl{k}]\to [N]$. Here $m$ is the dimension of the space of $G$-invariant tensors in $\mathcal{H}^{\otimes \textsl{k}}$.
\end{theorem}

\subsection{Orthogonal group case}\label{orthogroupcase}

From now on, we let the compact group $G$ be the orthogonal group ${\rm O}\left( N\right)$. In this case, we take $\mathcal{H} = \mathbb{C}^{N}$ with the canonical basis $\varepsilon_{1}, \ldots, \varepsilon_{N}$, and $U(\textsl{g}) = \textsl{g}$, namely the matrix multiplication operator on $\mathcal{H}$. We first compute the Weingarten integrals when $\textsl{k}=2$, i.e., computing the integrals
\begin{align*}
\int_{{\rm O}\left( N\right)}\textsl{g}_{\textbf{i}(1)\textbf{j}(1)} \textsl{g}_{\textbf{i}(2)\textbf{j}(2)}\textrm{d}\textsl{g}.
\end{align*}

By Schur-Weyl duality, one can show that $\Img\mathscr{P}$ is a $1$-dimensional subspace with the basis $a_{1} = \sum_{i\in [N]}\varepsilon_{i}\otimes \varepsilon_{i}$. Hence the Gram matrix $\textrm{W} = (\langle a_{1},a_{1}\rangle) = (N)$ and its inverse $\textrm{W}^{-1} = (1/N)$.

\begin{proposition}\label{k=2}
    Let $N\in \mathbb{N}_{+}$ and $m,t,k\in [N]$. Then
\begin{align*}
\int_{{\rm O}\left( N\right)}\textsl{g}_{mk} \textsl{g}_{tk}\textrm{d}\textsl{g} = \begin{dcases}
\frac{1}{N},\ m=t,\\
0,\ m\neq t.
\end{dcases}
\end{align*}
\end{proposition}

Now, we turn to compute the Weingarten integrals when $\textsl{k}=4$, i.e., computing the integrals
\begin{align*}
\int_{{\rm O}\left( N\right)}\textsl{g}_{\textbf{i}(1)\textbf{j}(1)} \textsl{g}_{\textbf{i}(2)\textbf{j}(2)}\textsl{g}_{\textbf{i}(3)\textbf{j}(3)} \textsl{g}_{\textbf{i}(4)\textbf{j}(4)}\textrm{d}\textsl{g}.
\end{align*}
By Schur-Weyl duality, one can show that $\Img\mathscr{P}$ is a $3$-dimensional subspace with the basis $a_{1} = \sum_{i,i'\in [N]}\varepsilon_{i}\otimes \varepsilon_{i}\otimes \varepsilon_{i'}\otimes \varepsilon_{i'}$, $a_{2} = \sum_{i,i'\in [N]}\varepsilon_{i}\otimes \varepsilon_{i'}\otimes \varepsilon_{i}\otimes \varepsilon_{i'}$, and $a_{3} = \sum_{i,i'\in [N]}\varepsilon_{i}\otimes \varepsilon_{i'}\otimes \varepsilon_{i'}\otimes \varepsilon_{i}$. Then the Gram matrix 
\begin{align*}
    \textrm{W} = \begin{pmatrix}
\langle a_{1},a_{1}\rangle & \langle a_{1},a_{2} \rangle & \langle a_{1},a_{3}\rangle\\ \langle a_{2},a_{1}\rangle & \langle a_{2},a_{2} \rangle & \langle a_{2},a_{3}\rangle \\
\langle a_{3},a_{1}\rangle & \langle a_{3},a_{2} \rangle & \langle a_{3},a_{3}\rangle\\
\end{pmatrix} = 
\begin{pmatrix}
N^{2} & N & N\\ N & N^{2} & N \\
N & N & N^{2}\\
\end{pmatrix}
\end{align*}
and its inverse
\begin{align*}
    \textrm{W}^{-1} = \begin{pmatrix}
\frac{N+1}{N(N-1)(N+2)} & \frac{-1}{N(N-1)(N+2)} & \frac{-1}{N(N-1)(N+2)}\\ \frac{-1}{N(N-1)(N+2)} & \frac{N+1}{N(N-1)(N+2)} & \frac{-1}{N(N-1)(N+2)} \\
\frac{-1}{N(N-1)(N+2)} & \frac{-1}{N(N-1)(N+2)} & \frac{N+1}{N(N-1)(N+2)}\\
\end{pmatrix}
\end{align*}
exist whenever $N\geq 2$.

\begin{proposition}\label{k=4 case1}
    Let $N\geq 2$ and $m,t,k,q\in [N]$. Then
\begin{align*}
\int_{{\rm O}\left( N\right)}\textsl{g}_{mk}^{2} \textsl{g}_{tq}^{2}\textrm{d}\textsl{g} = \begin{dcases}
\frac{3}{N(N+2)},\ k=q, m=t,\\
\frac{1}{N(N+2)},\ k=q, m\neq t,\\
\frac{1}{N(N+2)},\ k\neq q, m=t,\\
\frac{N+1}{N(N-1)(N+2)},\ k\neq q, m\neq t.
\end{dcases}
\end{align*}
\end{proposition}

\begin{proposition}\label{k=4 case2}
    Let $N\geq 2$ and $m,m',t,t',k,q\in [N]$. Suppose $(m,t)\neq (m',t')$. Then
\begin{align*}
\int_{{\rm O}\left( N\right)}\textsl{g}_{mk}\textsl{g}_{tq} \textsl{g}_{m'k}\textsl{g}_{t'q}\textrm{d}\textsl{g} = \begin{dcases}
\frac{1}{N(N+2)},\ k=q, m=t, m'=t', m\neq m',\\
\frac{1}{N(N+2)},\ k=q, m=t', t=m', m\neq t,\\
\frac{-1}{N(N-1)(N+2)},\ k\neq q, m=t, m'=t', m\neq m',\\
\frac{-1}{N(N-1)(N+2)},\ k\neq q, m=t', t=m', m\neq t,\\
0,\ otherwise.
\end{dcases}
\end{align*}
\end{proposition}

\subsection{Complete proof of Proposition~\ref{1st and 2nd moments of pairwise distance}}\label{secA53}
In this final subsection, we complete the computation in the proof of Proposition~\ref{1st and 2nd moments of pairwise distance}, utilizing the formulas derived in the previous subsection.

\begin{proof}[Proof of Proposition \ref{1st and 2nd moments of pairwise distance}.\nopunct]
It remains to compute $\mathbb{V}\|\hat{e}_{i}-\hat{e}_{j}\|^{2}$. From~\eqref{dahsj}, \eqref{eqofpairwisedistance}, we have
\begin{align*}
    \left(\frac{1}{n-1}\|\hat{e}_{i}-\hat{e}_{j}\|^{2}\right)^{2}& = \left[\sum\limits_{k\in [s]}\left(u_{ik}^{E}-u_{jk}^{E}\right) ^{2}\right]^{2}\\
    & = \sum\limits_{(k,q)\in [s]^{2}}\left(u_{ik}^{E}-u_{jk}^{E}\right)^{2}\left(u_{iq}^{E}-u_{jq}^{E}\right)^{2}\notag \\
    & = \sum\limits_{(k,q)\in [s]^{2}}\left(\sum_{m\in [n-1]}C_{m}^{ij}g_{mk}\cdot \sum_{t\in [n-1]}C_{t}^{ij}g_{tq}\right)^{2}\notag \\
    & = \sum\limits_{(k,q)\in [s]^{2}}\left(\sum_{(m,t)\in [n-1]^{2}}C_{m}^{ij}C_{t}^{ij}g_{mk}g_{tq}\right)^{2}\notag \\
    & \begin{aligned}
        = \sum\limits_{(k,q)\in [s]^{2}}\Bigg[\sum_{(m,t)\in [n-1]^{2}}\left(C_{m}^{ij}\right)^{2}\left(C_{t}^{ij}\right)^{2}g_{mk}^{2}g_{tq}^{2}\\
        +\sum_{(m,t)\neq (m',t')}C_{m}^{ij}C_{t}^{ij}C_{m'}^{ij}C_{t'}^{ij}g_{mk}g_{tq}g_{m'k}g_{t'q}\Bigg]\notag 
    \end{aligned}\\
    & = \mathbf{I}+\mathbf{II}+\mathbf{III}+\mathbf{IV},
\end{align*}
where 
\begin{align*}
    \mathbf{I} & = \sum_{k\in [s]}\sum_{(m,t)\in [n-1]^{2}}\left(C_{m}^{ij}\right)^{2}\left(C_{t}^{ij}\right)^{2}g_{mk}^{2}g_{tk}^{2},\\
    \mathbf{II} & = \sum_{\substack{k,q\in [s] \\ k\neq q}}\sum_{(m,t)\in [n-1]^{2}}\left(C_{m}^{ij}\right)^{2}\left(C_{t}^{ij}\right)^{2}g_{mk}^{2}g_{tq}^{2},\\
    \mathbf{III} & = \sum\limits_{k\in [s]}\sum_{(m,t)\neq (m',t')}C_{m}^{ij}C_{t}^{ij}C_{m'}^{ij}C_{t'}^{ij}g_{mk}g_{tk}g_{m'k}g_{t'k},\\
    \mathbf{IV} & = \sum_{\substack{k,q\in [s] \\ k\neq q}}\sum_{(m,t)\neq (m',t')}C_{m}^{ij}C_{t}^{ij}C_{m'}^{ij}C_{t'}^{ij}g_{mk}g_{tq}g_{m'k}g_{t'q}.
\end{align*}
Letting $N = n-1$ in Proposition~\ref{k=4 case1} yields
   \begin{align*}
    & \mathbb{E}\mathbf{I}  = \mathbb{E}\left[\sum_{k\in [s]}\sum_{m\in [n-1]}\left(C_{m}^{ij}\right)^{4}g_{mk}^{4}+\sum_{k\in [s]}\sum_{\substack{m,t\in [n-1] \\ m\neq t}}\left(C_{m}^{ij}\right)^{2}\left(C_{t}^{ij}\right)^{2}g_{mk}^{2}g_{tk}^{2}\right]\\
    & = \frac{3s}{(n-1)(n+1)}\sum_{m\in [n-1]}\left(C_{m}^{ij}\right)^{4}+\frac{s}{(n-1)(n+1)}\sum_{\substack{m,t\in [n-1] \\ m\neq t}}\left(C_{m}^{ij}\right)^{2}\left(C_{t}^{ij}\right)^{2}
   \end{align*}
and 
  \begin{align*}  
    & \mathbb{E}\mathbf{II}  = \mathbb{E}\left[\sum_{\substack{k,q\in [s] \\ k\neq q}}\sum_{m\in [n-1]}\left(C_{m}^{ij}\right)^{4}g_{mk}^{2}g_{mq}^{2}+\sum_{\substack{k,q\in [s] \\ k\neq q}}\sum_{\substack{m,t\in [n-1] \\ m\neq t}}\left(C_{m}^{ij}\right)^{2}\left(C_{t}^{ij}\right)^{2}g_{mk}^{2}g_{tq}^{2}\right]\\
    & = \frac{s^{2}-s}{(n-1)(n+1)}\sum_{m\in [n-1]}\hspace{-2.5pt}\left(C_{m}^{ij}\right)^{4}+\frac{n(s^2-s)}{(n-1)(n-2)(n+1)}\sum_{\substack{m,t\in [n-1] \\ m\neq t}}\hspace{-2.5pt}\left(C_{m}^{ij}\right)^{2}\left(C_{t}^{ij}\right)^{2}.
  \end{align*}
Similarly, letting $N = n-1$ in Proposition~\ref{k=4 case2} yields
\begin{align*}
    & \mathbb{E}\mathbf{III}  = \mathbb{E}\left[\sum\limits_{k\in [s]}\Bigg(\sum_{\substack{m=t\\ m'=t'\\ m\neq m'}}+\sum_{\substack{m=t'\\ t=m'\\ m\neq t}}\Bigg)C_{m}^{ij}C_{t}^{ij}C_{m'}^{ij}C_{t'}^{ij}g_{mk}g_{tk}g_{m'k}g_{t'k}\right]\\
    & = \mathbb{E}\left[2\sum\limits_{k\in [s]}\sum_{\substack{m,m'\in [n-1] \\ m\neq m'}}\left(C_{m}^{ij}\right)^{2}\left(C_{m'}^{ij}\right)^{2}g_{mk}^{2}g_{m'k}^{2}\right]\\
    & = \frac{2s}{(n-1)(n+1)}\sum_{\substack{m,m'\in [n-1] \\ m\neq m'}}\left(C_{m}^{ij}\right)^{2}\left(C_{m'}^{ij}\right)^{2}
\end{align*}
and
\begin{align*}
    & \mathbb{E}\mathbf{IV}  = \mathbb{E}\left[\sum_{\substack{k,q\in [s] \\ k\neq q}}\Bigg(\sum_{\substack{m=t\\ m'=t'\\ m\neq m'}}+\sum_{\substack{m=t'\\ t=m'\\ m\neq t}}\Bigg)C_{m}^{ij}C_{t}^{ij}C_{m'}^{ij}C_{t'}^{ij}g_{mk}g_{tq}g_{m'k}g_{t'q}\right]\\
    & = \mathbb{E}\left[2\sum_{\substack{k,q\in [s] \\ k\neq q}}\sum_{\substack{m,m'\in [n-1] \\ m\neq m'}}\left(C_{m}^{ij}\right)^{2}\left(C_{m'}^{ij}\right)^{2}g_{mk}g_{mq}g_{m'k}g_{m'q}\right]\\
    & = \frac{-2(s^2-s)}{(n-1)(n-2)(n+1)}\sum_{\substack{m,m'\in [n-1] \\ m\neq m'}}\left(C_{m}^{ij}\right)^{2}\left(C_{m'}^{ij}\right)^{2}.
\end{align*}
Hence
\begin{align*}
    & \mathbb{E}\left(\frac{1}{n-1}\|\hat{e}_{i}-\hat{e}_{j}\|^{2}\right)^{2}  = \mathbb{E}\mathbf{I}+\mathbb{E}\mathbf{II}+\mathbb{E}\mathbf{III}+\mathbb{E}\mathbf{IV}\\
    & = \frac{s^{2}+2s}{(n-1)(n+1)}\left[\sum_{m\in [n-1]}\left(C_{m}^{ij}\right)^{4}+ \sum_{\substack{m,t\in [n-1] \\ m\neq t}}\left(C_{m}^{ij}\right)^{2}\left(C_{t}^{ij}\right)^{2}\right]\\
    & = \frac{s^{2}+2s}{(n-1)(n+1)}\left[\sum_{m\in [n-1]}\left(C_{m}^{ij}\right)^{2}\right]^{2}\\
    & = \frac{4s^{2}+8s}{(n-1)(n+1)}.
\end{align*}
Here, we use $\sum_{m\in [n-1]}(C_{m}^{ij})^{2} = 2$ as seen in~\eqref{eqn: the explicit form of sum_square_C^ij_m}. On the other hand, we recall the reader that $\mathbb{E}\|\hat{e}_{i}-\hat{e}_{j}\|^{2} = 2s$ from the first half of the proof of Proposition~\ref{1st and 2nd moments of pairwise distance}. Consequently,
\begin{align*}
    \mathbb{V}\|\hat{e}_{i}-\hat{e}_{j}\|^{2} & = \mathbb{E}\|\hat{e}_{i}-\hat{e}_{j}\|^{4}-\left(\mathbb{E}\|\hat{e}_{i}-\hat{e}_{j}\|^{2}\right)^{2}\\
    & = (n-1)^{2}\cdot \frac{4s^{2}+8s}{(n-1)(n+1)}-4s^{2}\\
    & = \frac{-8s^{2}+8(n-1)s}{n+1}. \tag*{\qedhere}
\end{align*}
\end{proof}
\end{appendices}

\bibliography{sn-bibliography}

\end{document}